\documentclass[10pt]{article}

\usepackage[francais,english]{babel}
\usepackage[utf8]{inputenc}
\usepackage[T1]{fontenc}
\usepackage{lmodern}
\usepackage{amsmath}
\usepackage{amssymb}
\usepackage{mathrsfs}
\usepackage{amsthm}
\usepackage{enumerate}
\usepackage{dsfont} 
\usepackage{eurosym} 
\usepackage{stmaryrd} 

\usepackage{graphicx}
\usepackage{float}

\usepackage{hyperref}  
\hypersetup{                    
	colorlinks=true,                
	breaklinks=true,                
	urlcolor= blue,                 
	linkcolor= black,                
	citecolor= magenta,                
	pdfstartview = FitH,
	bookmarksopen = true               
}

\usepackage{listings}
\usepackage{xcolor}
\usepackage{framed}

\usepackage{fullpage} 

\theoremstyle{definition}

\theoremstyle{plain}
\newtheorem{demo}{Démonstration}

\newcommand{\R}{\mathbb{R}}
\newcommand{\N}{\mathbb{N}}
\newcommand{\E}{\mathbb{E}}
\newcommand{\PP}{\mathbb{P}}

\newcommand{\widetildelow}[1]{\stackrel{\sim}{\smash{#1}\rule{0pt}{1.5pt}}}
\newcommand{\widetildemid}[1]{\stackrel{\sim}{\smash{#1}\rule{0pt}{2.5pt}}}

\newcommand{\1}{\mathds{1}}

\title{A new family of one dimensional martingale couplings}
\author{B. Jourdain\thanks{Université Paris-Est, Cermics (ENPC), INRIA F-77455 Marne-la-Vallée, France. E-mails: benjamin.jourdain@enpc.fr, william.margheriti@enpc.fr} \and W. Margheriti\footnotemark[1]}
\date{\today}

\theoremstyle{plain}
\newtheorem{prooff}{Proof}[section]

\addto\captionsfrench{}

\newtheorem{lemma}[prooff]{Lemma}
\newtheorem{cor2}[prooff]{Corollary}
\newtheorem{thm2}[prooff]{Theorem}
\newtheorem{prop2}[prooff]{Proposition}

\theoremstyle{definition}
\newtheorem{example}[prooff]{Example}
\newtheorem{rk}[prooff]{Remark}

\allowdisplaybreaks

\newcommand{\jump}{\vspace{6pt}}
\newcommand{\numberthis}{\addtocounter{equation}{1}\tag{\theequation}}

\numberwithin{equation}{section}

\begin{document}
\maketitle

\begin{abstract}
	In this paper, we exhibit a new family of martingale couplings between 
	two one-dimensional probability measures $\mu$ and $\nu$ in the convex 
	order. This family is parametrised by two dimensional probability measures on the 
	unit square with respective marginal densities proportional to the positive and 
	negative parts of the difference between the quantile functions of $\mu$ 
	and $\nu$. It contains the inverse transform martingale coupling which 
	is explicit in terms of the cumulative distribution 
	functions of these marginal densities. The integral of $|x-y|$ with respect to each of these 
	couplings is smaller than twice the $W_1$ distance between $\mu$ and $\nu$. When the comonotonous coupling between $\mu$ and $\nu$ is given by a map $T$, the elements of the family minimise $\int_\R\vert y-T(x)\vert\,M(dx,dy)$ among all martingale couplings between $\mu$ and $\nu$. When $\mu$ and $\nu$ are in the decreasing (resp. increasing) convex order, the construction is generalised to exhibit super (resp. sub) martingale couplings.
\end{abstract}

\section{Introduction}


For all $d\in\N^*$, $\rho\ge1$ and $\mu,\nu$ in the set $\mathcal P_\rho(\R^d)$ of probability measures on $\R^d$ with finite $\rho$-th moment, we define the Wasserstein distance with index $\rho$ by $W_\rho(\mu,\nu)=(\inf_{P\in\Pi(\mu,\nu)}\int_{\R^d\times\R^d}\vert x-y\vert^\rho\,P(dx,dy))^{1/\rho}$, where $\Pi(\mu,\nu)$ denotes the set of couplings between $\mu$ and $\nu$, that is $\Pi(\mu,\nu)=\{P\in\mathcal P_1(\R^d\times\R^d)\mid\forall A\in\mathcal B(\R^d),\ P(A\times\R^d)=\mu(A)\ \mathrm{and}\ P(\R^d\times A)=\nu(A)\}$. Let $\Pi^{\mathrm{M}}(\mu,\nu)$ be the set of martingale couplings between $\mu$ and $\nu$, that is $\Pi^{\mathrm{M}}(\mu,\nu)=\{M\in\Pi(\mu,\nu)\mid\mu(dx)\text{-a.e.},\ \int_{\R^d}\vert y\vert\,m(x,dy)<+\infty\ \mathrm{and}\ \int_{\R^d}y\,m(x,dy)=x\}$. The celebrated Strassen theorem \cite{strassenThm} ensures that if $\mu,\nu\in\mathcal P_1(\R^d)$, then $\Pi^{\mathrm M}(\mu,\nu)\neq\emptyset$ iff $\mu$ and $\nu$ are in the convex order. We recall that two probability measures $\mu,\nu\in\mathcal P_1(\R^d)$ are in the convex order, and denote $\mu\le_{cx}\nu$, if $\int_{\R^d}f(x)\,\mu(dx)\le\int_{\R^d}f(x)\,\nu(dx)$ for any convex function $f:\R^d\to\R$. We denote $\mu<_{cx}\nu$ if $\mu\le_{cx}\nu$ and $\mu\neq\nu$. For all $\rho\ge1$ and $\mu,\nu\in\mathcal P_\rho(\R^d)$, we define $\mathcal M_\rho(\mu,\nu)$ by
\[
\mathcal M_\rho(\mu,\nu)=\left(\inf_{M\in\Pi^{\mathrm{M}}(\mu,\nu)}\int_{\R^d\times\R^d}\vert x-y\vert^\rho\,M(dx,dy)\right)^{1/\rho}.
\]

Our main result is the following stability inequality which shows that if $\mu$ and $\nu$ are in the convex order and close to each other, then there exists a martingale coupling which expresses this proximity:
\begin{equation}\label{inegaliteRecherchee}
\forall \mu,\nu\in\mathcal P_1(\R)\mbox{ such that }\mu\le_{cx}\nu,\quad\mathcal M_1(\mu,\nu)\le 2W_1(\mu,\nu).
\end{equation}

It is well known (see for instance Remark 2.19 (ii) Chapter 2 \cite{VillaniOT2}) that for all $\mu,\nu\in\mathcal P_\rho(\R)$,
\begin{equation}
W_\rho(\mu,\nu)=\left(\int_0^1\left\vert F_\mu^{-1}(u)-F_\nu^{-1}(u)\right\vert^\rho\,du\right)^{1/\rho},
\end{equation}
where we denote by $F_\eta(x)=\eta((-\infty,x])$ and $F_\eta^{-1}(u)=\inf\{x\in\R\mid F_\mu(x)\ge u\},u\in(0,1)$, the cumulative distribution function and the quantile function of a probability measure $\eta$ on $\R$. We prove the inequality \eqref{inegaliteRecherchee} by exhibiting a new family of martingale couplings $M$ such that $\int_{\R\times\R}\vert x-y\vert\,M(dx,dy)\le 2W_1(\mu,\nu)$. We will show (see the proof of Theorem \ref{GrandMMinimiseCritere}) that the constant $2$ is sharp in \eqref{inegaliteRecherchee}. We will also see that $\eqref{inegaliteRecherchee}$ cannot be generalised with $\mathcal M_1(\mu,\nu)$ and $W_1(\mu,\nu)$ replaced by $\mathcal M_\rho(\mu,\nu)$ and $W_\rho(\mu,\nu)$ for $\rho>1$. The case $\rho=2$ is easy, since for all $\mu,\nu\in\mathcal P_2(\R)$ and $M\in\Pi^{\mathrm{M}}(\mu,\nu)$, $\int_{\R\times\R}\vert x-y\vert^2\,M(dx,dy)=\int_\R y^2\,\nu(dy)-\int_\R x^2\,\mu(dx)$, which is independent from $M$. For all $n\in\N^*$, let $\mu_n$ be the centered Gaussian distribution with variance $n^2$. Then we get that $\mathcal M_2(\mu_n,\mu_{n+1})=\sqrt{2n+1}\underset{n\to+\infty}{\longrightarrow}+\infty$, whereas $W_\rho(\mu_n,\mu_{n+1})=(\int_0^1\vert nF_{\mu_1}^{-1}(u)-(n+1)F_{\mu_{1}}^{-1}(u)\vert^\rho\,du)^{1/\rho}=\E[\vert G\vert^\rho]^{1/\rho}<+\infty$ for $G\sim\mathcal N_1(0,1)$, which makes the equivalent of $\eqref{inegaliteRecherchee}$ impossible to hold. Extension to the case $\rho>2$ is immediate with the same example thanks to Jensen's inequality which provides $\mathcal M_\rho(\mu_n,\mu_{n+1})\ge\mathcal M_2(\mu_n,\mu_{n+1})=\sqrt{2n+1}$, whereas $W_\rho(\mu_n,\mu_{n+1})$ is still bounded.

This problem is motivated by the resolution of the Martingale Optimal Transport (MOT) problem introduced by Beiglböck, Henry-Labordère and Penkner \cite{beiglPHLPenk} in a discrete time setting, and Galichon, Henry-Labordère and Touzi \cite{galichonPHLTouzi} in a continuous time setting. For adaptations of celebrated results on classical optimal transport theory to the MOT problem, we refer to Henry-Labordère, Tan and Touzi \cite{PHLTanTouzi} and Henry-Labordère and Touzi \cite{PHLTouzi2016}. To tackle numerically the MOT problem, we refer to Alfonsi, Corbetta and Jourdain \cite{ApproxMOTJourdain}, Alfonsi, Corbetta and Jourdain \cite{ApproxMOTJourdain2}, De March \cite{deMarch1} and Guo and Ob\l{}\'{o}j \cite{guoObloj}. On duality, we refer to Beiglböck, Nutz and Touzi \cite{beiglbockNutzTouzi}, Beiglböck, Lim and Ob\l{}\'{o}j \cite{beiglblockLimObloj} and De March \cite{deMarch3}. We also refer to De March \cite{deMarch2} and De March and Touzi \cite{deMarchTouzi} for the multi-dimensional case. Once the martingale optimal transport problem is discretised by approximating $\mu$ and $\nu$ by probability measures with finite support and in the convex order, one can raise the question of the convergence of the discrete optimal cost towards the continuous one. The present paper is a step forward in proving the stability of the martingale optimal transport problem with respect to the marginals.  

We develop in Section \ref{MesuresQ} an abstract construction of a new family of martingale couplings between two probability measures $\mu$ and $\nu$ on the real line with finite first moments and comparable in the convex order. This family is parametrised by two dimensional probability measures on the unit square with respective marginal densities proportional to the positive and the negative parts of the difference $F_\mu^{-1}-F_\nu^{-1}$ between the quantile functions  of $\mu$ and $\nu$. 
 Moreover, each martingale coupling in the family is obtained as the image by $(u,y)\mapsto (F_\mu^{-1}(u),y)$ of $\1_{(0,1)}(u)\,du\,\widetilde m^Q(u,dy)$ where $\widetilde m^Q$ is a Markov kernel on $(0,1)\times \R$ such that $\int_{(0,1)}\widetilde m^Q(u,\{y\in\R\mid|y-F_\nu^{-1}(u)|=(y-F_\nu^{-1}(u)){\rm sg}(F_\mu^{-1}(u)-F_{\nu}^{-1}(u))\})du=1$, where ${\rm sg}(x)=\1_{\{x>0\}}-\1_{\{x<0\}}$ for $x\in\R$. Therefore, for $(U,Y)$ distributed according to $\1_{(0,1)}(u)\,du\,\widetilde m^Q(u,dy)$, $(F_\mu^{-1}(U),Y)$ is a martingale coupling and
\begin{equation}\label{egaliteYFnuMoins1DonneW1}
\E[|Y-F_\nu^{-1}(U)|]=\E[{\rm sg}(F_\mu^{-1}(U)-F_\nu^{-1}(U))\E[Y-F_\nu^{-1}(U)|U]]=\E[|F_\mu^{-1}(U)-F_\nu^{-1}(U)|]=W_1(\mu,\nu).
\end{equation}

When the comonotonous coupling between $\mu$ and $\nu$, that is the law of $(F_\mu^{-1}(U),F_\nu^{-1}(U))$, is given by a map $T$, the elements of the family minimise $\int_\R\vert y-T(x)\vert\,M(dx,dy)$ among all martingale couplings between $\mu$ and $\nu$. We deduce from \eqref{egaliteYFnuMoins1DonneW1} that $\E[|Y-F_\mu^{-1}(U)|]\le\E[|Y-F_\nu^{-1}(U)|]+\E[|F_\nu^{-1}(U)-F_\mu^{-1}(U)|]=2 W_1(\mu,\nu)$ which implies \eqref{inegaliteRecherchee} as soon as the parameter set of probability measures on the unit square in non empty.

In Section \ref{sec:ITMC}, we give  an explicit example of such a probability measure on the unit square. We call the associated martingale coupling the inverse transform martingale coupling. This coupling is explicit in terms of the cumulative distribution functions of the above-mentioned densities and their left-continuous generalised inverses. It is therefore more explicit than the left-curtain (and right-curtain) coupling introduced by Beiglböck and Juillet \cite{BeiglbockJuillet} and which under the condition that $\nu$ has no atoms and the set of local maximal values of $F_\nu-F_\mu$ is finite can be explicited according to Henry-Labordère and Touzi \cite{PHLTouzi2016} by solving two coupled ordinary differential equations starting from each right-most local maximiser. We also check that the inverse transform martingale coupling is stable  with respect to its marginals $\mu$ and $\nu$ for the Wasserstein distance. The building brick of the inverse transform martingale coupling is a martingale coupling between $\mu_{u,v}=p\delta_{F_\mu^{-1}(u)}+(1-p)\delta_{F_\mu^{-1}(v)}$ and $\nu_{u,v}=p\delta_{F_\nu^{-1}(u)}+(1-p)\delta_{F_\nu^{-1}(v)}$ with $0<u<v<1$ such that
\begin{equation}\label{buildingBrickITMC}
F_\nu^{-1}(u)<F_\mu^{-1}(u)<F_\mu^{-1}(v)<F_\nu^{-1}(v),
\end{equation}
where we choose a common weight $p$ (resp. $1-p$) for $F_\mu^{-1}(u)$ and $F_\nu^{-1}(u)$  (resp. $F_\mu^{-1}(v)$ and $F_\nu^{-1}(v)$) to help ensuring that the second marginal is equal to $\nu$ when the first is equal to $\mu$. Then $p$ is given by the equality of the means which in view of the condition \eqref{buildingBrickITMC} on the supports is equivalent to the convex order between $\mu_{u,v}$ and $\nu_{u,v}$: $\frac{1-p}{p}=\frac{F_\mu^{-1}(u)-F_\nu^{-1}(u)}{F_\nu^{-1}(v)-F_\mu^{-1}(v)}$. 
We rely on the necessary condition of Theorem 3.A.5 Chapter 3 \cite{ShakedShanthikumar}: $\mu,\nu\in\mathcal P_1(\R)$ are such that $\mu\le_{cx}\nu$ iff for all $u\in[0,1]$, $\int_0^uF_\mu^{-1}(v)\,dv\ge\int_0^uF_\nu^{-1}(v)dv$ with equality for $u=1$. This implies that for all $u\in[0,1]$, $\Psi_+(u):=\int_0^u(F_\mu^{-1}-F_\nu^{-1})^+(v)\,dv\ge \int_0^u(F_\mu^{-1}-F_\nu^{-1})^-(v)\,dv:=\Psi_-(u)$ where $x^+:=\max(x,0)$ and $x^-:=\max(-x,0)$ respectively denote the positive and negative parts of a real number $x$. We now choose $v=\Psi_-^{-1}(\Psi_+(u))$ where $\Psi_-^{-1}$ is the left-continuous generalised inverse of $\Psi_-$. Then $d\Psi_+(u)$ a.e. $u<v$ (consequence of $\Psi_-\le\Psi_+$) and $F_\nu^{-1}(u)<F_\mu^{-1}(u)<F_\mu^{-1}(v)<F_\nu^{-1}(v)$ (consequence of the definitions of $\Psi_+$ and $\Psi_-$, see Section \ref{sousSec:DefITMC}). Moreover the key equality $\frac{dv}{du}=\frac{(F_\mu^{-1}-F_\nu^{-1})^+(u)}{(F_\mu^{-1}-F_\nu^{-1})^-(v)}=\frac{1-p}{p}$ explains why the construction succeeds. More details are given in Section \ref{sec:ITMC}.

The cardinality of this new family of martingale couplings between $\mu$ and $\nu$ is discussed in Section \ref{subsec:InfiniteAmountMartingaleCouplings}. This family is shown to be convex and is therefore either a singleton like when $\nu$ only weighs two points, or uncountably infinite like when $\mu(\{x\})=\nu(\{x\})=0$ for all $x\in\R$.

The construction is finally generalised in Section \ref{sec:SubSuperMartingaleCase} to exhibit super (resp. sub) martingale couplings as soon as $\mu$ is smaller than $\nu$ in the decreasing (resp. increasing) convex order. We recall that two probability measures $\mu,\nu\in\mathcal P_1(\R)$ are in the decreasing (resp. increasing) convex order and denote $\mu\le_{dcx}\nu$ (resp. $\mu\le_{icx}\nu$) if $\int_{\R}f(x)\,\mu(dx)\le\int_{\R}f(x)\,\nu(dx)$ for any decreasing (resp. increasing) convex function $f:\R\to\R$.

In all the paper, a capital letter $M$ which denotes a coupling between $\mu$ and $\nu$ is associated to its small letter $m$ which denotes the regular conditional probability distribution of $M$ with respect to $\mu$, that is the ($\mu$-a.e.) unique Markov kernel such that $M(dx,dy)=\mu(dx)\,m(x,dy)$.



{\bf Acknowledgements:} We thank Jean-François Delmas (CERMICS) for numerous and fruitful discussions.

\section{A new family of martingale couplings}
\label{MesuresQ}

\subsection{A simple example}
\label{sousSec:aSimpleExample}

Let us construct a coupling in dimension $1$ which shows that \eqref{inegaliteRecherchee} holds true in a simple case. We say that a centered probability measure $\mu\in\mathcal P_1(\R)$ is symmetric if $\mu=\bar{\mu}$, where $\bar\mu$ denotes the image of $\mu$ by $x\mapsto -x$. Let then $\mu$ and $\nu$ be centered and symmetric probability measures on $\R$ such that $ F_\mu^{-1}(u)\ge F_\nu^{-1}(u)$ for all $u\in(0,1/2]$ and $F_\mu^{-1}(u)\le F_\nu^{-1}(u)$ for all $u\in(1/2,1)$. Let $U$ be a random variable uniformly distributed on $(0,1)$. According to the inverse transform sampling, the probability distributions of $F_\mu^{-1}(U)$ and $F_\nu^{-1}(U)$ are respectively $\mu$ and $\nu$. Let $Y$ be the random variable defined by
	\begin{equation}\label{defCouplageCasSymetrique}
	Y=F_\nu^{-1}(U)\1_{\{F_\nu^{-1}(U)\neq0,V\le\frac{F_\mu^{-1}(U)+F_\nu^{-1}(U)}{2F_\nu^{-1}(U)}\}}-F_\nu^{-1}(U)\1_{\{F_\nu^{-1}(U)\neq0,V>\frac{F_\mu^{-1}(U)+F_\nu^{-1}(U)}{2F_\nu^{-1}(U)}\}},
	\end{equation}
	where $V$ is a random variable uniformly distributed on $(0,1)$ independent from $U$. It is clear by symmetry of $\mu$ that $F_\mu(0)\ge1/2$, so $F_\mu^{-1}(1/2)\le0$. Moreover, for all $x\in\R$ and $u>1/2$, $F_\mu(x)\ge u$ implies $x\ge0$, so $F_\mu^{-1}(u)\ge0$. Therefore, we have
	\begin{equation}\label{inegaliteCasSymetrique}
	\forall u\in(0,1/2],\quad F_\nu^{-1}(u)\le F_\mu^{-1}(u)\le 0\quad\textrm{and}\quad\forall u\in(1/2,1),\quad0\le F_\mu^{-1}(u)\le F_\nu^{-1}(u).
	\end{equation}
	
	In particular, when $F_\nu^{-1}(U)=0$, then $F_\mu^{-1}(U)=0$ and $Y=0$. Let us check that $Y$ is distributed according to $\nu$. Using that $(F_\mu^{-1}(U),F_\nu^{-1}(U))$ and $(-F_\mu^{-1}(U),-F_\nu^{-1}(U))$ are identically distributed (see Lemma \ref{caracterisationLoiSymetriqueFonctionQuantile} in Section \ref{Appendix}), we have for all measurable and bounded function $h:\R\to\R$,
	\begin{align*}
	\E[h(Y)]&=\E\bigg[h(F_\nu^{-1}(U))\1_{\{F_\nu^{-1}(U)\neq0,V\le\frac{F_\mu^{-1}(U)+F_\nu^{-1}(U)}{2F_\nu^{-1}(U)}\}}\bigg]+\E\bigg[h(F_\nu^{-1}(U))\1_{\{F_\nu^{-1}(U)\neq0,V>\frac{F_\mu^{-1}(U)+F_\nu^{-1}(U)}{2F_\nu^{-1}(U)}\}}\bigg]\\
	&\phantom{=}+h(0)\PP(F_\nu^{-1}(U)=0)\\
	&=\E[h(F_\nu^{-1}(U))].
	\end{align*}
	
	Moreover, according to \eqref{inegaliteCasSymetrique}, we have $\frac{F_\mu^{-1}(u)+F_\nu^{-1}(u)}{2F_\nu^{-1}(u)}\in[0,1]$ for all $u\in(0,1)$ such that $F_\nu^{-1}(u)\neq0$. In addition to that, we have $F_\nu^{-1}(u)\frac{F_\mu^{-1}(u)+F_\nu^{-1}(u)}{2F_\nu^{-1}(u)}-F_\nu^{-1}(u)\left(1-\frac{F_\mu^{-1}(u)+F_\nu^{-1}(u)}{2F_\nu^{-1}(u)}\right)=F_\mu^{-1}(u)$ for all $u\in(0,1)$ such that $F_\nu^{-1}(u)\neq0$. So $\E[Y\vert U]=F_\mu^{-1}(U)\1_{\{F_\nu^{-1}(U)\neq0\}}=F_\mu^{-1}(U)$ since $F_\nu^{-1}(U)=0$ implies $F_\mu^{-1}(U)=0$. So we deduce that $\E[Y\vert F_\mu^{-1}(U)]=F_\mu^{-1}(U)$. Therefore, the law of $(F_\mu^{-1}(U),Y)$ is an explicit martingale coupling between $\mu$ and $\nu$.
	
	Furthermore, remarking that $|Y-F_\nu^{-1}(U)|=(Y-F_\nu^{-1}(U)){\rm sg}(F_\mu^{-1}(U)-F_\nu^{-1}(U))$, we deduce from the equality \eqref{egaliteYFnuMoins1DonneW1} that $\E[|Y-F_\mu^{-1}(U)|]\le\E[|Y-F_\nu^{-1}(U)|]+\E[|F_\nu^{-1}(U)-F_\mu^{-1}(U)|]=2 W_1(\mu,\nu)$, so \eqref{inegaliteRecherchee} holds.

\subsection{Definition}
\label{sousSec:DefMesuresQ}

Let $\mu$ and $\nu$ be two probability measures on $\R$ with finite first moment such that $\int_\R x\,\mu(dx)=\int_\R y\,\nu(dy)$ and $\mu\neq\nu$. We recall that $\Psi_+$ and $\Psi_-$ are defined for all $u\in[0,1]$ by $\Psi_+(u)=\int_0^u(F_\mu^{-1}-F_\nu^{-1})^+(v)\,dv$ and $\Psi_-(u)=\int_0^u(F_\mu^{-1}-F_\nu^{-1})^-(v)\,dv$. Let $\mathcal U_+$, $\mathcal U_-$ and $\mathcal U_0$ be defined by
\begin{multline}\label{defUmoinsUplusUzero}
\mathcal U_+=\{u\in(0,1)\mid F_\mu^{-1}(u)>F_\nu^{-1}(u)\},\quad\mathcal U_-=\{u\in(0,1)\mid F_\mu^{-1}(u)<F_\nu^{-1}(u)\}\\
\textrm{and}\quad\mathcal U_0=\{u\in(0,1)\mid F_\mu^{-1}(u)=F_\nu^{-1}(u)\}.
\end{multline}

Notice that $d\Psi_+(u)$-a.e. (resp. $d\Psi_-(u)$-a.e.), we have $u\in\mathcal U_+$ (resp. $u\in\mathcal U_-$). Since $\mu$ and $\nu$ have equal means, we can set $\gamma=\int_0^1(F_\mu^{-1}-F_\nu^{-1})^+(u)\,du=\int_0^1(F_\mu^{-1}-F_\nu^{-1})^-(u)\,du\in(0,+\infty)$.  We note $\mathcal Q$ the set of probability measures $Q(du,dv)$ on $(0,1)^2$ such that
\begin{enumerate}[(i)]
	\item $Q$ has first marginal $\frac1\gamma(F_\mu^{-1}-F_\nu^{-1})^+(u)\,du=\frac1\gamma\,d\Psi_+(u)$;
	\item $Q$ has second marginal $\frac1\gamma(F_\mu^{-1}-F_\nu^{-1})^-(v)\,dv=\frac1\gamma\,d\Psi_-(v)$;
	\item $Q\left(\{(u,v)\in (0,1)^2\mid u<v\}\right)=1$.
\end{enumerate}

\begin{example}\label{exempleQ1} Let $\mu,\nu\in\mathcal P_1(\R)$ be such that $\mu<_{cx}\nu$. Suppose that the difference of the quantile functions changes sign only once, that is there exists $p\in(0,1)$ such that $u\mapsto\int_0^u(F_\mu^{-1}(v)-F_\nu^{-1}(v))\,dv$ is nondecreasing on $[0,p]$ and nonincreasing on $[p,1]$. Then one can easily see that any probability measure $Q$ defined on $(0,1)$ satisfying properties $(i)$ and $(ii)$ of the definition of $\mathcal Q$ is concentrated on $(0,p)\times(p,1)$ and therefore satisfies $(iii)$. In particular, the probability measure $Q_1$ defined on $(0,1)^2$ by
	\begin{equation}
	\label{exempleMesureQ1}
	Q_1(du,dv)=\frac{1}{\gamma^2}(F_\mu^{-1}-F_\nu^{-1})^+(u)\,du\,(F_\mu^{-1}-F_\nu^{-1})^-(v)\,dv
	\end{equation}
	is an element of $\mathcal Q$.
\end{example}

In view of $(i)$ and $(ii)$, one could rewrite $(iii)$ as $Q\left(\{(u,v)\in\mathcal U_+\times\mathcal U_-\mid u<v\}\right)=1$. A characterisation of the support of $Q$ in terms of the irreducible components of $\mu$ and $\nu$ is given by Proposition \ref{ComposantesIrreductiblesSupportMesuresQ} below. In the general case, the construction of a probability measure $Q\in\mathcal Q$ is not straightforward, but a direct consequence of Proposition \ref{ITMCcasParticulier} below is that $\mathcal Q$ is non-empty as long as $\mu,\nu\in\mathcal P_1(\R)$ are such that $\mu<_{cx}\nu$. Moreover, the convexity of $\mathcal Q$ is clear.

\begin{prop2}\label{QnonVide} Let $\mu,\nu\in\mathcal P_1(\R)$ be such that $\mu<_{cx}\nu$. Then $\mathcal Q$ is a non-empty convex set.
\end{prop2}

Let $Q$ be an element of $\mathcal Q$. Let $\pi^Q_-$ and $\pi^Q_+$ be two sub-Markov kernels on $(0,1)$ such that for $du$-almost all $u\in\mathcal U_+$ and $dv$-almost all $v\in\mathcal U_-$, $\pi^Q_+(u,(0,1))=1$, $\pi^Q_-(v,(0,1))=1$ and
\begin{align*}
Q(du,dv)&=\frac1\gamma(F_\mu^{-1}-F_\nu^{-1})^+(u)\,du\,\pi^Q_+(u,dv)=\frac1\gamma(F_\mu^{-1}-F_\nu^{-1})^-(v)\,dv\,\pi^Q_-(v,du).
\end{align*}

Let $(\widetilde m^Q(u,dy))_{u\in(0,1)}$ be the Markov kernel defined by

\begin{equation}\label{defPetitm}
\left\{\begin{array}{r}
\displaystyle\int_{v\in(0,1)}\frac{F_\mu^{-1}(u)-F_\nu^{-1}(u)}{F_\nu^{-1}(v)-F_\nu^{-1}(u)}\,\delta_{F_\nu^{-1}(v)}(dy)\,\pi^Q_+(u,dv)+\int_{v\in(0,1)}\frac{F_\nu^{-1}(v)-F_\mu^{-1}(u)}{F_\nu^{-1}(v)-F_\nu^{-1}(u)}\,\pi^Q_+(u,dv)\,\delta_{F_\nu^{-1}(u)}(dy)\jump\\
\mathrm{for}\ u\in\mathcal U_+\ \textrm{such that}\ \pi^Q_+(u,\{v\in(0,1)\mid F_\nu^{-1}(v)>F_\mu^{-1}(u)\})=1;\\
\\
\displaystyle\int_{v\in(0,1)}\frac{F_\mu^{-1}(u)-F_\nu^{-1}(u)}{F_\nu^{-1}(v)-F_\nu^{-1}(u)}\,\delta_{F_\nu^{-1}(v)}(dy)\,\pi^Q_-(u,dv)+\int_{v\in(0,1)}\frac{F_\nu^{-1}(v)-F_\mu^{-1}(u)}{F_\nu^{-1}(v)-F_\nu^{-1}(u)}\,\pi^Q_-(u,dv)\,\delta_{F_\nu^{-1}(u)}(dy)\jump\\
\mathrm{for}\ u\in\mathcal U_-\ \textrm{such that}\ \pi^Q_-(u,\{v\in(0,1)\mid F_\nu^{-1}(v)<F_\mu^{-1}(u)\})=1;\\
\\
\delta_{F_\nu^{-1}(u)}(dy)\quad\quad\quad\quad \mathrm{otherwise.}\\
\end{array}
\right.
\end{equation}

For any Markov kernel $(\widetilde m(u,dy))_{u\in(0,1)}$, we denote by $(m(x,dy))_{x\in\R}$ the Markov kernel defined by
\begin{equation}\label{defGrandM}
\left\{
\begin{array}{cr}
\delta_x(dy)&\text{if}\ F_\mu(x)=0\ \text{or}\ F_\mu(x_-)=1;\\
\\
\displaystyle\frac{1}{\mu(\{x\})}\int_{u=F_\mu(x_-)}^{F_\mu(x)}\widetilde m(u,dy)\,du&\text{if}\ \mu(\{x\})>0;\\
\\
\displaystyle \widetilde m(F_\mu(x),dy)&\text{otherwise.}
\end{array}
\right.
\end{equation}

For all $x\in\R$ such that $F_\mu(x)>0$ and $F_\mu(x_-)<1$, $m(x,dy)$ can be rewritten as
\begin{equation}\label{defAlternativeGrandM}
m(x,dy)=\int_{v=0}^1\widetilde m(F_\mu(x_-)+v(F_\mu(x)-F_\mu(x_-)),dy)\,dv.
\end{equation}

Conversely, let $(p(x,dy))_{x\in\R}$ be a Markov kernel. Let then $(\widetilde m(u,dy))_{u\in(0,1)}$ be the Markov kernel defined for all $u\in(0,1)$ by $\widetilde m(u,dy)=p(F_\mu^{-1}(u),dy)$. Let $(m(x,dy))_{x\in\R}$ be the Markov kernel defined by \eqref{defGrandM}. Let $x\in\R$ be such that $F_\mu(x_-)>0$ and $F_\mu(x)<1$. If $\mu(\{x\})>0$, then for all $u\in(F_\mu(x_-),F_\mu(x)]$, $F_\mu^{-1}(u)=x$. Hence $m(x,dy)=\frac{1}{\mu(\{x\})}\int_{u=F_\mu(x_-)}^{F_\mu(x)}\widetilde m(u,dy)\,du=\frac{1}{\mu(\{x\})}\int_{u=F_\mu(x_-)}^{F_\mu(x)}p(x,dy)\,du=p(x,dy)$. By Lemma \ref{FmoinsUnrondF}, $F_\mu^{-1}(F_\mu(x))=x$, $\mu(dx)$-almost everywhere. So for $\mu(dx)$-almost all $x\in\R$ such that $F_\mu(x_-)>0$, $F_\mu(x)<1$ and $\mu(\{x\})=0$, $m(x,dy)=p(F_\mu^{-1}(F_\mu(x)),dy)=p(x,dy)$. Therefore, for $\mu(dx)$-almost all $x\in\R$, $p(x,dy)=m(x,dy)$.

In all the paper, for any $Q\in\mathcal Q$, $(m^Q(x,dy))_{x\in\R}$ and $M^Q$ will respectively denote the Markov kernel given by \eqref{defGrandM} when $(\widetilde m(u,dy))_{u\in(0,1)}=(\widetilde m^Q(u,dy))_{u\in(0,1)}$ and the probability measure on $\R^2$ defined by $M^Q(dx,dy)=\mu(dx)\,m^Q(x,dy)$.

\begin{prop2}\label{muMcouplageMartingale} Let $\mu$ and $\nu$ be two distinct probability measures on $\R$ with finite first moment and equal means such that $\mathcal Q$ is non-empty. Let $Q\in\mathcal Q$. Then the  probability measure $M^Q$ is a martingale coupling between $\mu$ and $\nu$.
\end{prop2}

One can easily check thanks to Jensen's inequality that the existence of a martingale coupling between $\mu$ and $\nu$ implies that $\mu\le_{cx}\nu$ (see Remark \ref{rqPreuveConstructiveStrassen} for a proof). A direct consequence of the latter fact and the two last propositions is an easy characterisation of the emptiness of $\mathcal Q$.

\begin{cor2}\label{caracterisationQNonVide} Let $\mu$ and $\nu$ be two distinct probability measures on $\R$ with finite first moment and equal means. Then $\mathcal Q\neq\emptyset$ iff $\mu\le_{cx}\nu$.
\end{cor2}

The proof of Proposition \ref{muMcouplageMartingale} relies on the two following lemmas.

\begin{lemma}\label{Petitmdupp} Let $Q\in\mathcal Q$. For $du$-almost all $u\in(0,1)$,
	\[
	\left\{
	\begin{array}{rcl}
	u\in\mathcal U_+&\implies&F_\nu^{-1}(v)>F_\mu^{-1}(u),\ \pi^Q_+(u,dv)\text{-a.e;}\jump\\
	u\in\mathcal U_-&\implies&F_\nu^{-1}(v)<F_\mu^{-1}(u),\ \pi^Q_-(u,dv)\text{-a.e.}
	\end{array}
	\right.
	\]
\end{lemma}
\begin{proof}[Proof of Lemma \ref{Petitmdupp}] We have
	\begin{align*}
	&\int_{(0,1)}\left(\int_{(0,1)}\1_{\{F_\nu^{-1}(v)\le F_\mu^{-1}(u)\}}\pi^Q_+(u,dv)\right)(F_\mu^{-1}-F_\nu^{-1})^+(u)\,du=\gamma\int_{(0,1)^2}\1_{\{F_\nu^{-1}(v)\le F_\mu^{-1}(u)\}}\,Q(du,dv)\\
	&\le\gamma\int_{(0,1)^2}\1_{\{F_\nu^{-1}(v)\le F_\mu^{-1}(v)\}}\,Q(du,dv)=\int_{(0,1)^2}\1_{\{F_\mu^{-1}(v)-F_\nu^{-1}(v)\ge0\}}(F_\mu^{-1}-F_\nu^{-1})^-(v)\,dv\,\pi^Q_-(v,du)\\
	&=0,
	\end{align*}
	where we used for the inequality that $u<v$, $Q(du,dv)$-almost everywhere and that $F_\mu^{-1}$ is nondecreasing. So for $du$-almost all $u\in\mathcal U_+$, $\pi^Q_+(u,dv)$-a.e., $F_\nu^{-1}(v)>F_\mu^{-1}(u)$. With a symmetric reasoning, we obtain that for $du$-almost all $u\in\mathcal U_-$, $\pi^Q_-(u,dv)$-a.e., $F_\nu^{-1}(v)<F_\mu^{-1}(u)$.
\end{proof}

\begin{lemma}\label{lienPetitmGrandM} Let $(\widetilde m(u,dy))_{u\in(0,1)}$ be a Markov kernel and let $(m(x,dy))_{x\in\R}$ be given by \eqref{defGrandM}. Then
	\[
	\mu(dx)\,m(x,dy)=(F_\mu^{-1}(u),y)_\sharp\left(\1_{(0,1)}(u)\,du\,\widetilde m(u,dy)\right),
	\]
	where $\sharp$ denotes the pushforward operation.
\end{lemma}
\begin{proof}[Proof of Lemma \ref{lienPetitmGrandM}] Let $h:\R^2\to\R$ be a measurable and nonnegative function. By Lemma \ref{FcomprisEntre0et1} below, $F_\mu(x)>0$ and $F_\mu(x_-)<1$, $\mu(dx)$-almost everywhere. So using \eqref{defAlternativeGrandM}, we have
	\begin{multline*}
	\int_{\R\times\R}h(x,y)\,\mu(dx)\,m(x,dy)\\
	=\int_{\R\times\R\times(0,1)}h(x,y)\1_{\{0<F_\mu(x),F_\mu(x_-)<1\}}\,\mu(dx)\,\widetilde m(F_\mu(x_-)+v(F_\mu(x)-F_\mu(x_-)),dy)\,dv.
	\end{multline*}
	
	Let $\theta:(x,v)\mapsto F_\mu(x_-)+v(F_\mu(x)-F_\mu(x_-))$. By Lemma \ref{FmuXCorrigeUniforme} below, $x=F_\mu^{-1}(\theta(x,v))$, $\mu(dx)\otimes dv$-almost everywhere on $\R\times(0,1)$ and $\theta(x,v)_\sharp(\mu(dx)\otimes\1_{(0,1)}(v)\,dv)=\1_{(0,1)}(u)\,du$. So
	\begin{multline*}
	\int_{\R\times\R}h(x,y)\,\mu(dx)\,m(x,dy)\\
	=\int_{\R\times\R\times(0,1)}h(F_\mu^{-1}(\theta(x,v)),y)\1_{\{0<F_\mu(F_\mu^{-1}(\theta(x,v))),F_\mu(F_\mu^{-1}(\theta(x,v)_-))<1\}}\,\mu(dx)\,\widetilde m(\theta(x,v),dy)\,dv\\
	=\int_{\R\times(0,1)}h(F_\mu^{-1}(u),y)\1_{\{0<F_\mu(F_\mu^{-1}(u)),F_\mu(F_\mu^{-1}(u)_-)<1\}}\,\widetilde m(u,dy)\,du.
	\end{multline*}
	
	By Lemma \ref{FcomprisEntre0et1} below and the inverse transform sampling, $F_\mu(F_\mu^{-1}(u))>0$ and $F_\mu(F_\mu^{-1}(u)_-)<1$, $du$-almost everywhere on $(0,1)$, hence
	\[
	\int_{\R\times\R}h(x,y)\,\mu(dx)\,m(x,dy)=\int_{\R\times(0,1)}h(F_\mu^{-1}(u),y)\,\widetilde m(u,dy)\,du.
	\]

\end{proof}

\begin{proof}[Proof of Proposition \ref{muMcouplageMartingale}] Let us show that $M^Q$ defines a coupling between $\mu$ and $\nu$. Let $h:\R\to\R$ be a measurable and nonnegative (or bounded) function. We want to show that
	\[
	\int_{\R\times\R}h(y)\,\mu(dx)m^Q(x,dy)=\int_\R h(y)\,\nu(dy),
	\]
	
	which by Lemma \ref{lienPetitmGrandM} and the inverse transform sampling is equivalent to
	\begin{equation}\label{equivalentPreuveMartingale}
	\int_0^1\int_\R h(y)\,\widetilde m^Q(u,dy)\,du=\int_0^1h(F_\nu^{-1}(u))\,du.
	\end{equation}
	
Thanks to Lemma \ref{Petitmdupp}, we get for $du$-almost all $u\in(0,1)$,
	\begin{align*}
	&\int_\R h(y)\,\widetilde m^Q(u,dy)\\
	&=\int_{(0,1)}\left(1-\frac{F_\mu^{-1}(u)-F_\nu^{-1}(u)}{F_\nu^{-1}(v)-F_\nu^{-1}(u)}\right)h(F_\nu^{-1}(u))(\pi^Q_+(u,dv)\1_{\{F_\mu^{-1}(u)>F_\nu^{-1}(u)\}}+\pi^Q_-(u,dv)\1_{\{F_\mu^{-1}(u)<F_\nu^{-1}(u)\}})\\
	&\phantom{=}+\int_{(0,1)}\left(\frac{F_\mu^{-1}(u)-F_\nu^{-1}(u)}{F_\nu^{-1}(v)-F_\nu^{-1}(u)}\right)h(F_\nu^{-1}(v))(\pi^Q_+(u,dv)\1_{\{F_\mu^{-1}(u)>F_\nu^{-1}(u)\}}+\pi^Q_-(u,dv)\1_{\{F_\mu^{-1}(u)<F_\nu^{-1}(u)\}})\\
	&\phantom{=}+h(F_\nu^{-1}(u))\1_{\{F_\mu^{-1}(u)=F_\nu^{-1}(u)\}}\\
	&=h(F_\nu^{-1}(u))+\int_{(0,1)}\frac{(F_\mu^{-1}-F_\nu^{-1})^+(u)}{F_\nu^{-1}(v)-F_\nu^{-1}(u)}(h(F_\nu^{-1}(v))-h(F_\nu^{-1}(u)))\,\pi^Q_+(u,dv)\\
	&\phantom{=}+\int_{(0,1)}\frac{(F_\mu^{-1}-F_\nu^{-1})^-(u)}{F_\nu^{-1}(u)-F_\nu^{-1}(v)}(h(F_\nu^{-1}(v))-h(F_\nu^{-1}(u)))\,\pi^Q_-(u,dv).\numberthis\label{calculIntegralehMtilde}
	\end{align*}
	
	Since
	\begin{align*}
	&\int_{(0,1)^2}\frac{(F_\mu^{-1}-F_\nu^{-1})^+(u)}{F_\nu^{-1}(v)-F_\nu^{-1}(u)}(h(F_\nu^{-1}(v)-h(F_\nu^{-1}(u)))\,\pi^Q_+(u,dv)\,du\\
	&=\gamma\int_{(0,1)^2}\frac{h(F_\nu^{-1}(v))-h(F_\nu^{-1}(u))}{F_\nu^{-1}(v)-F_\nu^{-1}(u)}Q(du,dv)\\
	&=\int_{(0,1)^2}\frac{h(F_\nu^{-1}(v))-h(F_\nu^{-1}(u))}{F_\nu^{-1}(v)-F_\nu^{-1}(u)}(F_\mu^{-1}-F_\nu^{-1})^-(v)\,\pi^Q_-(v,du)\,dv\\
	&=-\int_{(0,1)^2}\frac{(F_\mu^{-1}-F_\nu^{-1})^-(u)}{F_\nu^{-1}(u)-F_\nu^{-1}(v)}(h(F_\nu^{-1}(v))-h(F_\nu^{-1}(u)))\,\pi^Q_-(u,dv)\,du,
	\end{align*}
	we deduce that $\int_0^1\int_\R h(y)\,\widetilde m^Q(u,dy)\,du=\int_0^1h(F_\nu^{-1}(u))\,du$. We conclude that $M^Q$ is a coupling between $\mu$ and $\nu$. In particular for $h:y\mapsto\vert y\vert$, using the inverse transform sampling, we have
	\[
	\int_0^1\int_\R\vert y\vert\,\widetilde m^Q(u,dy)\,du=\int_0^1\vert F_\nu^{-1}(u)\vert\,du=\int_\R\vert y\vert\,\nu(dy)<+\infty.
	\]
	
	So $\int_\R y\,\widetilde m^Q(u,dy)$ is well defined $du$-almost everywhere on $(0,1)$.
	
	Let us show now that $M^Q$ defines a martingale coupling between $\mu$ and $\nu$. By Lemma \ref{Petitmdupp}, for $du$-almost all $u\in\mathcal U_+$,
	\begin{align}\label{calculMartingaleUmoins}
	\int_\R y\,\widetilde m^Q(u,dy)&=\int_{(0,1)}\left(1-\frac{F_\mu^{-1}(u)-F_\nu^{-1}(u)}{F_\nu^{-1}(v)-F_\nu^{-1}(u)}\right)F_\nu^{-1}(u)\,\pi^Q_+(u,dv)\nonumber\\
	&+\int_{(0,1)}\left(\frac{F_\mu^{-1}(u)-F_\nu^{-1}(u)}{F_\nu^{-1}(v)-F_\nu^{-1}(u)}\right)F_\nu^{-1}(v)\,\pi^Q_+(u,dv)\nonumber\\
	&=\int_{(0,1)}(F_\nu^{-1}(u)+F_\mu^{-1}(u)-F_\nu^{-1}(u))\,\pi^Q_+(u,dv)\nonumber\\
	&=F_\mu^{-1}(u).
	\end{align}
	
	In the same way, for $du$-almost all $u\in\mathcal U_-$,
	\begin{align}\label{calculMartingaleUplus}
	\int_\R y\,\widetilde m^Q(u,dy)&=F_\mu^{-1}(u).
	\end{align}
	
	Else if $u\in\mathcal U_0$, then by definition of $\widetilde m^Q(u,dy)$,
	\[
	\int_\R y\,\widetilde m^Q(u,dy)=F_\nu^{-1}(u)=F_\mu^{-1}(u),
	\]
	
	so for $du$-almost all $u\in(0,1)$, $\int_\R y\,\widetilde m^Q(u,dy)=F_\mu^{-1}(u)$.
	
	Let $h:\R\to\R$ be a measurable and bounded function. By Lemma \ref{lienPetitmGrandM},
	\[
	\int_{\R\times\R}h(x)(y-x)\,\mu(dx)\,m^Q(x,dy)=\int_0^1h(F_\mu^{-1}(u))\left(\int_\R(y-F_\mu^{-1}(u))\,\widetilde m^Q(u,dy)\right)\,du=0.
	\]
	
	%
	%
	
	So $\mu(dx)\,m^Q(x,dy)$ is a martingale coupling between $\mu$ and $\nu$.
\end{proof}

According to Theorem A.4 \cite{BeiglbockJuillet}, there exist $N\in\N^*\cup\{+\infty\}$ and a sequence of disjoint open intervals $((\underline t_n,\overline t_n))_{1\le n\le N}$ such that
\begin{equation}\label{defComposantesIrreductibles}
\left\{t\in\R\mid\int_{-\infty}^tF_\mu(x)\,dx<\int_{-\infty}^tF_\nu(x)\,dx\right\}=\bigcup_{n=1}^N(\underline t_n,\overline t_n).
\end{equation}

These intervals are called the irreducible components of the couple $(\mu,\nu)$. Moreover, there exists a unique decomposition of probability measures $(\mu_n,\nu_n)_{1\le n\le N}$, such that the choice of any martingale coupling $M$ between $\mu$ and $\nu$ reduces to the choice of a sequence of martingale couplings $(M_n)_{1\le n\le N}$. More precisely, for all $1\le n\le N$, 
\begin{equation}\label{comparaisonRepartitionsComposantesIrreductibles}
F_\mu(\underline t_n)\le F_\nu(\underline t_n)\le F_\nu((\overline t_n)_-)\le F_\mu((\overline t_n)_-),\quad F_\mu(\underline t_n)<F_\mu((\overline t_n)_-),
\end{equation}
and $\mu_n$ and $\nu_n$ are given by
\begin{equation}\label{defNuNComposantesIrreductibles}
\left\{
\begin{array}{rcl}
\mu_n(dx)&=&\frac{1}{F_\mu((\overline t_n)_-)-F_\mu(\underline t_n)}\1_{(\underline t_n,\overline t_n)}(x)\,\mu(dx);\\
\nu_n(dy)&=&\frac{1}{F_\mu((\overline t_n)_-)-F_\mu(\underline t_n)}\left(\1_{(\underline t_n,\overline t_n)}(y)\,\nu(dy)+(F_\nu(\underline t_n)-F_\mu(\underline t_n))\,\delta_{\underline t_n}(dy)+(F_\mu((\overline t_n)_-)-F_\nu((\overline t_n)_-))\,\delta_{\overline t_n}(dy)\right).
\end{array}
\right.
\end{equation}

Then a probability measure $M$ on $\R^2$ is a martingale coupling between $\mu$ and $\nu$ if and only if there exists a sequence $(M_n)_{1\le n\le N}$ such that for all $1\le n\le N$, $M_n$ is a martingale coupling between $\mu_n$ and $\nu_n$ and
	\[
	M(dx,dy)=\1_{\R\backslash\bigcup_{n=1}^N(\underline t_n,\overline t_n)}(x)\,\mu(dx)\,\delta_x(dy)+\sum_{n=1}^N\mu((\underline t_n,\overline t_n))\,M_n(dx,dy).
	\]
	
	We can establish a strong connection between the support of any probability measure $Q\in\mathcal Q$ and the irreducible components of $(\mu,\nu)$.

\begin{prop2}\label{ComposantesIrreductiblesSupportMesuresQ} Let $\mu,\nu\in\mathcal P_1(\R)$ be such that $\mu<_{cx}\nu$. Let $(\underline t_n,\overline t_n)_{1\le n\le N}$ denote the irreducible components of $(\mu,\nu)$. Then for all $Q\in\mathcal Q$, we have
	\[
	Q\left(\bigcup_{1\le n\le N}\left(F_\mu(\underline t_n),F_\mu((\overline t_n)_-)\right)^2\right)=1.
	\]
\end{prop2}
\begin{proof} Let $Q\in\mathcal Q$. By Lemma A.8 \cite{ApproxMOTJourdain2}, we have
	\[
	\mathcal W:=\bigcup_{n=1}^N\left(F_\mu(\underline t_n),F_\mu((\overline t_n)_-)\right)=\left\{u\in(0,1)\mid\int_0^uF_\mu^{-1}(v)\,dv>\int_0^uF_\nu^{-1}(v)\,dv\right\}.
	\]
	
	Let $u\in(0,1)$ be such that $F_\mu^{-1}(u)>F_\nu^{-1}(u)$, that is $u\in\mathcal U_+$. Since $\mu\le_{cx}\nu$, according to the necessary condition of Theorem 3.A.5 Chapter 3 \cite{ShakedShanthikumar} (see also Remark \ref{rqPreuveConstructiveStrassen} for a proof), for all $q\in[0,1]$, $\int_0^qF_\mu^{-1}(v)\,dv\ge\int_0^qF_\nu^{-1}(v)\,dv$. By left-continuity of $F_\mu^{-1}$ and $F_\nu^{-1}$, we deduce that $\int_0^uF_\mu^{-1}(v)\,dv>\int_0^uF_\nu^{-1}(v)\,dv$, that is $u\in\mathcal W$. So $\mathcal U_+\subset\mathcal W$.
	
	Let $1\le n\le N$. Then $M^Q$ transports $(\underline t_n,\overline t_n)$ to $[\underline t_n,\overline t_n]$, namely for $\mu(dx)$-almost all $x\in(\underline t_n,\overline t_n)$, $m^Q(x,[\underline t_n,\overline t_n])=1$. So using Lemma \ref{lienPetitmGrandM} for the last equality, we have
	\begin{align*}
	\int_{F_\mu(\underline t_n)}^{F_\mu((\overline t_n)_-)}du&=\mu((\underline t_n,\overline t_n))=\int_{\R}\1_{\{\underline t_n<x<\overline t_n\}}\mu(dx)\\
	&=\int_{\R^2}\1_{\{\underline t_n<x<\overline t_n\}}\1_{\{\underline t_n\le y\le\overline t_n\}}\,\mu(dx)\,m^Q(x,dy)\\
	&=\int_{(0,1)\times\R}\1_{\{\underline t_n<F_\mu^{-1}(u)<\overline t_n\}}\1_{\{\underline t_n\le y\le\overline t_n\}}\,du\,\widetilde m^Q(u,dy).
	\end{align*}
	
	Using Lemma \ref{FmoinsUnrondF}, one can easily see that for all $u\in(0,1)$, $\1_{\{F_\mu(\underline t_n)<u<F_\mu((\overline t_n)_-)\}}\le\1_{\{\underline t_n<F_\mu^{-1}(u)<\overline t_n\}}\le\1_{\{F_\mu(\underline t_n)<u\le F_\mu((\overline t_n)_-)\}}$. So	
	\begin{align*}
	\int_{F_\mu(\underline t_n)}^{F_\mu((\overline t_n)_-)}du&=\int_{(0,1)\times\R}\1_{\{F_\mu(\underline t_n)<u<F_\mu((\overline t_n)_-)\}}\1_{\{\underline t_n\le y\le\overline t_n\}}\,du\,\widetilde m^Q(u,dy)\\
	&=\int_{F_\mu(\underline t_n)}^{F_\mu((\overline t_n)_-)}\widetilde m^Q(u,[\underline t_n,\overline t_n])\,du.
	\end{align*}
	
	So for $du$-almost all $u\in(F_\mu(\underline t_n),F_\mu((\overline t_n)_-))$, $\widetilde m^Q(u,[\underline t_n,\overline t_n])=1$. By Lemma \ref{Petitmdupp}, $d\Psi_+(u)$-almost everywhere on $(F_\mu(\underline t_n),F_\mu((\overline t_n)_-))$,
	\begin{align*}
	1&=\pi^Q_+(u,\{v\in(0,1)\mid F_\nu^{-1}(v)\in[\underline t_n,\overline t_n]\})\\
	&=\pi^Q_+(u,\mathcal U_-\cap(u,1)\cap\{v\in(0,1)\mid F_\nu^{-1}(v)\in[\underline t_n,\overline t_n]\}),
	\end{align*}
	where the last equality derives from conditions $(ii)$ and $(iii)$ satisfied by $Q$. Let $u\in(F_\mu(\underline t_n),F_\mu((\overline t_n)_-))$. Let us check that
	\begin{equation}\label{EgaliteEnsembleComposantesIrreductibles}
	\mathcal U_-\cap(u,1)\cap\{v\in(0,1)\mid F_\nu^{-1}(v)\in[\underline t_n,\overline t_n]\}\subset\mathcal U_-\cap(u,1)\cap(F_\mu(\underline t_n),F_\mu((\overline t_n)_-)].
	\end{equation}
	
	Let $v\in(0,1)$ be such that $F_\nu^{-1}(v)\in[\underline t_n,\overline t_n]$. First of all, if $v>u$ then $v>F_\mu(\underline t_n)$. Second, if $v>F_\mu((\overline t_n)_-)$, then according to \eqref{comparaisonRepartitionsComposantesIrreductibles} and Lemma \ref{FmoinsUnrondF}, we have $F_\nu((\overline t_n)_-)\le F_\mu((\overline t_n)_-)<v\le F_\nu(\overline t_n)$. In that case, if $v\le F_\mu(\overline t_n)$, then $v\in(F_\nu((\overline t_n)_-),F_\nu(\overline t_n)]\cap(F_\mu((\overline t_n)_-),F_\mu(\overline t_n)]$, so $F_\nu^{-1}(v)=F_\mu^{-1}(v)=\overline t_n$ and $v\in\mathcal U_0$. Else if $v>F_\mu(\overline t_n)$, then $v\in(F_\mu(\overline t_n),F_\nu(\overline t_n)]$ so $F_\nu^{-1}(v)\le\overline t_n< F_\mu^{-1}(v)$ and $v\in\mathcal U_+$. This proves \eqref{EgaliteEnsembleComposantesIrreductibles}.
	
	Using conditions $(ii)$ and $(iii)$ satisfied by $Q$ again and the fact that the second marginal of $Q$ has a density, we get that $d\Psi_+(u)$-almost everywhere on $(F_\mu(\underline t_n),F_\mu((\overline t_n)_-))$,
	\begin{align*}
	1&=\pi^Q_+(u,\mathcal U_-\cap(u,1)\cap\{v\in(0,1)\mid F_\nu^{-1}(v)\in[\underline t_n,\overline t_n]\})\\
	&\le\pi^Q_+(u,\mathcal U_-\cap(u,1)\cap(F_\mu(\underline t_n),F_\mu((\overline t_n)_-)])\\
	&=\pi^Q_+(u,(F_\mu(\underline t_n),F_\mu((\overline t_n)_-))).
	\end{align*}
	
	We deduce that
	\begin{align*}
	Q\left(\bigcup_{1\le n\le N}\left(F_\mu(\underline t_n),F_\mu((\overline t_n)_-)\right)^2\right)&=\sum_{n=1}^NQ\left(\left(F_\mu(\underline t_n),F_\mu((\overline t_n)_-)\right)^2\right)\\
	&=\frac1\gamma\sum_{n=1}^N\int_{F_\mu(\underline t_n)}^{F_\mu((\overline t_n)_-)}d\Psi_+(u)\,\pi^Q_+(u,(F_\mu(\underline t_n),F_\mu((\overline t_n)_-)))\\
	&=\frac1\gamma\sum_{n=1}^N\int_{F_\mu(\underline t_n)}^{F_\mu((\overline t_n)_-)}d\Psi_+(u)=\frac1\gamma\,d\Psi_+(\mathcal W)\\
	&\ge\frac1\gamma\,d\Psi_+(\mathcal U_+)\\
	&=1,
	\end{align*}
	where we used the fact that $\mathcal U_+\subset\mathcal W$ for the inequality.
\end{proof}

The next proposition clarifies the structure of the set of martingale couplings deriving from $\mathcal Q$ and states a linearity property of the map $Q\in\mathcal Q\mapsto M^Q$. In particular, it ensures that the set of martingale couplings deriving from $\mathcal Q$ is either a singleton, or uncountably infinite.

\begin{prop2}\label{couplagesQconvexes} Let $\mu,\nu\in\mathcal P_1(\R)$ be such that $\mu<_{cx}\nu$. Then for all $Q,Q'\in\mathcal Q$ and $\lambda\in[0,1]$,
	\[
	M^{\lambda Q+(1-\lambda)Q'}=\lambda M^Q+(1-\lambda)M^{Q'}.
	\]
	
	In particular, the set $\{M^Q\mid Q\in\mathcal Q\}$ is convex.
\end{prop2}
\begin{proof} Let $Q,Q'\in\mathcal Q$ and let $\lambda\in[0,1]$. It is straightforward that for $du$-almost all $u\in\mathcal U_+$ and $dv$-almost all $v\in\mathcal U_-$,
	\begin{align*}
	\pi^{\lambda Q+(1-\lambda)Q'}_+(u,dy)&=\lambda\pi^Q_+(u,dy)+(1-\lambda)\pi^{Q'}_+(u,dy);\\
	\pi^{\lambda Q+(1-\lambda)Q'}_-(v,dy)&=\lambda\pi^Q_-(v,dy)+(1-\lambda)\pi^{Q'}_-(v,dy).
	\end{align*}
	
	Using Lemma \ref{Petitmdupp}, we get that for $du$-almost all $u\in(0,1)$,
	\[
	\widetilde m^{\lambda Q+(1-\lambda)Q'}(u,dy)=\lambda\widetilde m^Q(u,dy)+(1-\lambda)\widetilde m^{Q'}(u,dy).
	\]
	
	Let $h:\R^2\to\R$ be a measurable and bounded function. By Lemma \ref{lienPetitmGrandM},
	\begin{align*}
	&\int_{\R\times\R}h(x,y)\,M^{\lambda Q+(1-\lambda)Q'}(dx,dy)\\
	&=\int_{\R\times\R}h(x,y)\,\mu(dx)\,m^{\lambda Q+(1-\lambda)Q'}(x,dy)=\int_0^1\left(\int_\R h(F_\mu^{-1}(u),y)\,\widetilde m^{\lambda Q+(1-\lambda)Q'}(u,dy)\right)\,du\\
	&=\lambda\int_0^1\left(\int_\R h(F_\mu^{-1}(u),y)\,\widetilde m^Q(u,dy)\right)\,du+(1-\lambda)\int_0^1\left(\int_\R h(F_\mu^{-1}(u),y)\,\widetilde m^{Q'}(u,dy)\right)\,du\\
	&=\lambda\int_{\R\times\R}h(x,y)\,\mu(dx)\,m^Q(x,dy)+(1-\lambda)\int_{\R\times\R}h(x,y)\,\mu(dx)\,m^{Q'}(x,dy)\\
	&=\int_{\R\times\R}h(x,y)\,(\lambda M^Q+(1-\lambda)M^{Q'})(dx,dy).
	\end{align*}
	
	So $M^{\lambda Q+(1-\lambda)Q'}=\lambda M^Q+(1-\lambda)M^{Q'}$.
\end{proof}

We deduce that if $Q,Q'\in\mathcal Q$ are such that $M^Q\neq M^{Q'}$, then there exists a whole segment of martingale couplings between $\mu$ and $\nu$, all parametrised by $\mathcal Q$. More details are given in Section \ref{subsec:InfiniteAmountMartingaleCouplings}. Let us complete this section by revisiting the example given in Section \ref{sousSec:aSimpleExample}.

\begin{example} Suppose now that $\mu,\nu\in\mathcal P_1(\R)$ are symmetric with common mean $\alpha\in\R$, that is $(x-\alpha)_\sharp\mu(dx)=(\alpha-x)_\sharp\mu(dx)$ and $(y-\alpha)_\sharp\nu(dy)=(\alpha-y)_\sharp\nu(dy)$ where $\sharp$ denotes the pushforward operation. Suppose in addition that their respective quantile functions satisfy $F_\mu^{-1}\ge F_\nu^{-1}$ on $(0,1/2]$ and $F_\mu^{-1}\le F_\nu^{-1}$ on $(1/2,1)$. We saw in Section \ref{sousSec:aSimpleExample} that when $U$ is a random variable uniformly distributed on $[0,1]$ and $Z$ is given by \eqref{defCouplageCasSymetrique}, $(F_\mu^{-1}(U),Z)$ is an explicit coupling between $\mu$ and $\nu$ in the case $\alpha=0$. Let us show here that this coupling is in fact associated to a particular element of $\mathcal Q$. According to Lemma \ref{caracterisationLoiSymetriqueFonctionQuantile} below, we have $F_\mu^{-1}(u)=2\alpha-F_\mu^{-1}(1-u)$ and $F_\nu^{-1}(u)=2\alpha-F_\nu^{-1}(1-u)$ for $du$-almost all $u\in(0,1)$, which is helpful in order to see that the probability measure $Q_2$ defined on $(0,1)^2$ by
\begin{equation}
\label{exempleMesureQ2}
Q_2(du,dv)=\frac1\gamma(F_\mu^{-1}-F_\nu^{-1})^+(u)\,du\,\delta_{1-u}(dv)
\end{equation}
is an element of $\mathcal Q$ (in particular to check that it satisfies $(ii)$). For that element $Q_2$, using \eqref{defPetitm}, Lemma \ref{Petitmdupp} and Lemma \ref{caracterisationLoiSymetriqueFonctionQuantile}, we have for $du$-almost all $u\in\mathcal U_+\cup\mathcal U_-$,
\begin{equation}\label{mTildeCasSymetrique}
\widetilde m^{Q_2}(u,dy)=\frac{F_\mu^{-1}(u)+F_\nu^{-1}(u)-2\alpha}{2(F_\nu^{-1}(u)-\alpha)}\delta_{F_\nu^{-1}(u)}(dy)+\frac{F_\nu^{-1}(u)-F_\mu^{-1}(u)}{2(F_\nu^{-1}(u)-\alpha)}\delta_{2\alpha-F_\nu^{-1}(u)}(dy),
\end{equation}
and $\widetilde m^{Q_2}(u,dy)=\delta_{F_\nu^{-1}(u)}(dy)$ if $u\in\mathcal U_0$. Let $u\in(0,1)$. If $F_\mu^{-1}(u)=F_\nu^{-1}(u)\neq\alpha$, then $\delta_{F_\nu^{-1}(u)}(dy)$ coincides with the right-hand side of \eqref{mTildeCasSymetrique}. Furthermore if $F_\nu^{-1}(u)=\alpha$, since $\alpha\ge F_\mu^{-1}(u)\ge F_\nu^{-1}(u)$ or $\alpha\le F_\mu^{-1}(u)\le F_\nu^{-1}(u)$ by an easy generalisation of \eqref{inegaliteCasSymetrique}, then $F_\mu^{-1}(u)=\alpha$. Therefore \eqref{mTildeCasSymetrique} holds for $du$-almost all $u\in(0,1)$ such that $F_\nu^{-1}(u)\neq\alpha$ and $\widetilde m^{Q_2}(u,dy)=\delta_{F_\nu^{-1}(u)}(dy)$ for $du$-almost all $u\in(0,1)$ such that $F_\nu^{-1}(u)=\alpha$.


Let $U$ and $V$ be two independent random variables uniformly distributed on $(0,1)$ and let $Y$ be defined as in \eqref{defCouplageCasSymetrique} but with the mean $\alpha$ taken into account, that is
\[
Y=F_\nu^{-1}(U)\1_{\{F_\nu^{-1}(U)\neq \alpha,V\le\frac{F_\mu^{-1}(U)+F_\nu^{-1}(U)-2\alpha}{2(F_\nu^{-1}(U)-\alpha)}\}}+(2\alpha-F_\nu^{-1}(U))\1_{\{F_\nu^{-1}(U)\neq \alpha,V>\frac{F_\mu^{-1}(U)+F_\nu^{-1}(U)-2\alpha}{2(F_\nu^{-1}(U)-\alpha)}\}} +\alpha\1_{\{F_\nu^{-1}(U)=\alpha\}}.
\]

Then $(U,Y)$ is distributed according to $\1_{(0,1)}(u)\,du\,\widetilde m^{Q_2}(u,dy)$. By Lemma \ref{lienPetitmGrandM}, $(F_\mu^{-1}(U),Y)$ is distributed according to $\mu(dx)\,m(x,dy)$.

\end{example}

\subsection{Optimality property}

Let $\mu,\nu\in\mathcal P_1(\R)$. It is well known that $F_\nu^{-1}$ is constant on the jumps of $F_\mu$, that is $F_\nu^{-1}$ is constant on the intervals of the form $(F_\mu(x_-),F_\mu(x)]$, iff the comonotonous coupling between $\mu$ and $\nu$ is concentrated on the graph of a map $T:\R\to\R$, and then 
\begin{equation}\label{defTransportMonge}
T=F_\nu^{-1}\circ F_\mu.
\end{equation}

We will refer to $T$ as the Monge transport map.

\begin{prop2}\label{minimisationTransportMonge} Let $\mu,\nu\in\mathcal P_1(\R)$ be such that $\mu<_{cx}\nu$. Suppose in addition that $F_\nu^{-1}$ is constant on the intervals of the form $(F_\mu(x_-),F_\mu(x)]$. Let $T$ be the Monge transport map. Let $Q\in\mathcal Q$. Then
	\[
	\inf_{M\in\Pi^{\mathrm M}(\mu,\nu)}\int_{\R\times\R}\vert y-T(x)\vert\,M(dx,dy)=\int_{\R\times\R}\vert y-T(x)\vert\,M^Q(dx,dy)=W_1(\mu,\nu).
	\]
\end{prop2}
\begin{proof} This is a particular case of Proposition \ref{mMinimiseCritere} below. Indeed, let $M(dx,dy)=\mu(dx)\,m(x,dy)$ be a martingale coupling between $\mu$ and $\nu$. Let $(\widetilde m(u,dy))_{u\in(0,1)}$ be the kernel defined for all $u\in(0,1)$ by $\widetilde m(u,dy)=m(F_\mu^{-1}(u),dy)$. Using the inverse transform sampling, we have
	\begin{align*}
	\int_{\R\times\R}\vert y-T(x)\vert\,\mu(dx)\,m(x,dy)&=\int_{(0,1)\times\R}\vert y-T(F_\mu^{-1}(u))\vert\,du\,m(F_\mu^{-1}(u),dy)\\
	&=\int_{(0,1)\times\R}\vert y-F_\nu^{-1}(F_\mu((F_\mu^{-1}(u))))\vert\,du\,\widetilde m(u,dy),
	\end{align*}
	where we used for the last equality that $T=F_\nu^{-1}\circ F_\mu$. Let $u\in(0,1)$. If there exists $x\in\R$ such that $u=F_\mu(x)$, then $F_\mu(F_\mu^{-1}(u))=F_\mu(F_\mu^{-1}(F_\mu(x)))=F_\mu(x)=u$, so $F_\nu^{-1}(F_\mu(F_\mu^{-1}(u)))=F_\nu^{-1}(u)$. Else there exists $x$ in the set of discontinuities of $F_\mu$ such that $F_\mu(x_-)\le u<F_\mu(x)$. In that case, if $u>F_\mu(x_-)$ then $x=F_\mu^{-1}(u)$, so $F_\nu^{-1}(F_\mu(F_\mu^{-1}(u)))=F_\nu^{-1}(F_\mu(x))=F_\nu^{-1}(u)$ since $F_\nu^{-1}$ is constant on the jumps of $F_\mu$. Hence 
	\begin{equation}\label{LemmeMongeTransportMap}
	du\textrm{-a.e. on }(0,1),\quad F_\nu^{-1}(F_\mu(F_\mu^{-1}(u)))=F_\nu^{-1}(u).
	\end{equation}
	
	We deduce that
	\[
	\int_{(0,1)\times\R}\vert y-F_\nu^{-1}(F_\mu((F_\mu^{-1}(u))))\vert\,du\,\widetilde m(u,dy)=\int_{(0,1)\times\R}\vert y-F_\nu^{-1}(u)\vert\,du\,\widetilde m(u,dy).
	\]
	
	With a similar reasoning, we have $\int_{\R\times\R}\vert y-T(x)\vert\,\mu(dx)\,m^Q(x,dy)=\int_{(0,1)\times\R}\vert y-F_\nu^{-1}(u)\vert\,du\,\widetilde m^Q(u,dy)$. Therefore, using Proposition \ref{mMinimiseCritere} combined with Remark \ref{mMinimiseCritereRemarque} below, we get that $\int_{\R\times\R}\vert y-T(x)\vert\,M(dx,dy)$ is minimised when $M=M^Q$, for which we have $\int_{\R\times\R}\vert y-T(x)\vert\,M^Q(dx,dy)=W_1(\mu,\nu)$.
\end{proof}

\subsection{Stability inequality}

We can now state our main result. In the minimisation of the cost function $(x,y)\mapsto\vert x-y\vert$ with respect to the couplings between $\mu$ and $\nu$, the addition of the martingale constraint does not cost more than a factor $2$.


\begin{thm2}\label{GrandMMinimiseCritere} For all $\mu,\nu\in\mathcal P_1(\R)$ such that $\mu<_{cx}\nu$ and for all $Q$ in the non-empty set $\mathcal Q$,
	\begin{equation}
	\label{inegaliteControleNouvelleFamille}
	\int_{\R\times\R}\vert x-y\vert\,M^Q(dx,dy)\le 2W_1(\mu,\nu).
	\end{equation}
	
	Consequently,
	\begin{equation}
	\label{inegaliteControle}
	\mathcal M_1(\mu,\nu)\le2W_1(\mu,\nu).
	\end{equation}
	
	Moreover, the constant $2$ is sharp.	
\end{thm2}

The proof of Theorem \ref{GrandMMinimiseCritere} relies on Proposition \ref{mMinimiseCritere} below. Note that since $\Pi^{\mathrm{M}}(\mu,\nu)\subset\Pi(\mu,\nu)$, we always have $W_1(\mu,\nu)\le\mathcal M_1(\mu,\nu)$. Moreover, the stability inequality \eqref{inegaliteControle} can be tensorised: it holds in greater dimension when the marginals are independent, as the next corollary states.

\begin{cor2}\label{TensorisationInegaliteStabilite} Let $d\in\N^*$ and $\mu_1,\cdots,\mu_d,\nu_1,\cdots,\nu_d\in\mathcal P_1(\R)$ be such that for all $1\le i\le d$, $\mu_i\le_{cx}\nu_i$. Let $\mu=\mu_1\otimes\cdots\otimes\mu_d$ and $\nu=\nu_1\otimes\cdots\otimes\nu_d$. Then $\mu\le_{cx}\nu$ and
	\[
	\mathcal M_1(\mu,\nu)\le2W_1(\mu,\nu),
	\]
	when $\R^d$ is endowed with the $L^1$-norm.
\end{cor2}
\begin{proof}[Proof of Corollary \ref{TensorisationInegaliteStabilite}] For all $1\le i\le d$, since $\mu_i\le_{cx}\nu_i$, Strassen's theorem or Proposition \ref{muMcouplageMartingale} and Corollary \ref{caracterisationQNonVide} ensure the existence of a martingale coupling $M_i(dx_i,dy_i)=\mu_i(dx_i)\,m_i(x_i,dy_i)$ between $\mu_i$ and $\nu_i$. Let then $M$ be the probability measure on $\R^d\times\R^d$ defined by $M(dx,dy)=\mu(dx)\,m_1(x_1,dy_1)\cdots\,m_d(x_d,dy_d)$. Then it is clear that $M$ is a martingale coupling between $\mu$ and $\nu$, which shows that $\mu\le_{cx}\nu$, and
	\[
	\mathcal M_1(\mu,\nu)\le\int_{\R^d\times\R^d}\vert x-y\vert\,M(dx,dy)=\sum_{i=1}^d\int_{\R^d\times\R^d}\vert x_i-y_i\vert\,M(dx,dy)=\sum_{i=1}^d\int_{\R\times\R}\vert x_i-y_i\vert\,M_i(dx_i,dy_i).
	\]
	
	For all $1\le i\le d$, let $\mathcal Q_i$ denote the set $\mathcal Q$ with respect to $\mu=\mu_i$ and $\nu=\nu_i$ and let $Q_i\in\mathcal Q_i$. Then for $M_1=M^{Q_1},\cdots,M_d=M^{Q_d}$, we deduce from Theorem \ref{GrandMMinimiseCritere} that
	\[
	\mathcal M_1(\mu,\nu)\le\sum_{i=1}^d\int_{\R\times\R}\vert x_i-y_i\vert\,M^{Q_i}(dx_i,dy_i)\le2\sum_{i=1}^dW_1(\mu_i,\nu_i).
	\]
	
	Let $P\in\Pi(\mu,\nu)$ be a coupling between $\mu$ and $\nu$. For all $1\le i\le d$, let $P_i$ be the marginals of $P$ with respect to the coordinates $i$ and $i+d$, so that $P_i$ is a coupling between $\mu_i$ and $\nu_i$. Then
	\[
	\sum_{i=1}^dW_1(\mu_i,\nu_i)\le\sum_{i=1}^d\int_{\R\times\R}\vert x_i-y_i\vert\,P_i(dx_i,dy_i)=\int_{\R^d\times\R^d}\sum_{i=1}^d\vert x_i-y_i\vert\,P(dx,dy)=\int_{\R^d\times\R^d}\vert x-y\vert\,P(dx,dy).
	\]
	
	Since the inequality above is true for any coupling $P$ between $\mu$ and $\nu$, we deduce that $\sum_{i=1}^dW_1(\mu_i,\nu_i)\le W_1(\mu,\nu)$, which proves the assertion.
\end{proof}

 In the following remarks, we first look in which case the minimiser of \eqref{inegaliteControle}, studied by Hobson and Klimmek \cite{HobsonKlimmek}, derives from $\mathcal Q$. Second, we see that the left-curtain martingale coupling introduced by Beiglböck and Juillet \cite{BeiglbockJuillet} does not always satisfy \eqref{inegaliteControle}.

\begin{rk} The optimal martingale coupling $M\in\Pi^{\mathrm{M}}(\mu,\nu)$ which minimises $\int_{\R\times\R}\vert x-y\vert\,M(dx,dy)$ was actually characterised by Hobson and Klimmek \cite{HobsonKlimmek} under the dispersion assumption that there exists a bounded interval $E$ of positive length such that $(\mu-\nu)^+(E^\complement)=(\nu-\mu)^+(E)=0$. They show that the optimal coupling $M^{HK}$ is unique. Moreover, in the simpler case where $\mu\wedge\nu=0$, if $a<b$ denote the endpoints of $E$, then there exist two nonincreasing functions $R:(0,1)\to(-\infty,a]$ and $S:(0,1)\to[b,+\infty)$ such that for all $u\in(0,1)$, denoting $\widetilde m^{HK}(u,dy)=m^{HK}(F_\mu^{-1}(u),dy)$ where $m^{HK}(x,dy)\,\mu(dx)=M^{HK}(dx,dy)$, one has
\[
\widetilde m^{HK}(u,dy)=\frac{S(u)-F_\mu^{-1}(u)}{S(u)-R(u)}\delta_{R(u)}(dy)+\frac{F_\mu^{-1}(u)-R(u)}{S(u)-R(u)}\delta_{S(u)}(dy).
\]

We can discuss in which case $M^{HK}$ derives from $\mathcal Q$. Suppose first that $F_\nu^{-1}$ takes at least three different values, that is there exist $u,v,w\in(0,1)$ such that $F_\nu^{-1}(u)<F_\nu^{-1}(v)<F_\nu^{-1}(w)$. By left-continuity of $F_\nu^{-1}$, there exists $\varepsilon>0$ such that $F_\nu^{-1}(u)<F_\nu^{-1}(v-\varepsilon)$ and $F_\nu^{-1}(v)<F_\nu^{-1}(w-\varepsilon)$. Let $I_1=(0,u]$, $I_2=(v-\varepsilon,v]$ and $I_3=(w-\varepsilon,1]$. Those three intervals are such that for all $s\in I_1$ (resp. $s\in I_2$) and $t\in I_2$ (resp. $t\in I_3$), we have $F_\nu^{-1}(s)<F_\nu^{-1}(t)$. Since $R$ is nonincreasing, if the graph of $R$ meets the graph of $F_\nu^{-1}$ on one of those three intervals, then they cannot meet on the two others. We can assert the same with the graph of $S$ since $S$ is nonincreasing as well. Therefore, there exists $k\in{\{1,2,3\}}$ such that the intersection of $F_\nu^{-1}(I_k)$ and $R(I_k)\cup S(I_k)$ is empty. In particular, for all $t\in I_k$, $\widetilde m^{HK}(t,\{F_\nu^{-1}(t)\})=0$. However, thanks to Lemma \ref{Petitmdupp}, we can see that for all $Q\in\mathcal Q$, the Markov kernel $\widetilde m^Q$ is such that $\widetilde m^Q(u,\{F_\nu^{-1}(u)\})>0$ for $du$-almost all $u\in(0,1)$. Therefore, $M^{HK}$ does not derive from $\mathcal Q$.

If $F_\nu^{-1}$ does not take more than two different values, that is if $\nu$ is reduced to two atoms at most, then there exists a unique martingale coupling between $\mu$ and $\nu$, so $M^{HK}$ derives of course from $\mathcal Q$.

Note that the maximisation problem $\sup_{M\in\Pi^{\mathrm{M}}(\mu,\nu)}\int_{\R\times\R}\vert x-y\vert\,M(dx,dy)$ is discussed by Hobson and Neuberger \cite{HobsonNeuberger}.
\end{rk}
\begin{example} For instance, if $\mu(dx)=\frac12\1_{[-1,1]}(x)\,dx$ and $\nu(dx)=\frac12(\1_{[-2,-1)}+\1_{(1,2]})(x)\,dx$, then (see Example 6.1 \cite{HobsonKlimmek} for an equivalent calculation)
\[
m^{HK}(x,dy)=\left(\frac12-\frac{3x}{2\sqrt{12-3x^2}}\right)\delta_{-\frac12(x+\sqrt{12-3x^2})}(dy)+\left(\frac12+\frac{3x}{2\sqrt{12-3x^2}}\right)\delta_{\frac12(-x+\sqrt{12-3x^2})}(dy),
\]
which satisfies $m^{HK}(x,\{F_\nu^{-1}(F_\mu(x))\})>0$ iff $x\in\{(3-\sqrt{33})/6,(\sqrt{33}-3)/6\}$. On the other hand, for all $Q\in\mathcal Q$, the Markov kernel $m^Q$ is such that $m^Q(x,\{F_\nu^{-1}(F_\mu(x))\})>0$ for $dx$-almost all $x\in(-1,1)$.
\end{example}

\begin{rk} We investigate an example where the left-curtain martingale coupling introduced by Beiglböck and Juillet \cite{BeiglbockJuillet} does not satisfy \eqref{inegaliteControleNouvelleFamille}. Let $\mu\in\mathcal P_1(\R)$ be with density $f_\mu$ and let $u>1$ and $d>0$. Let $M^{LC}$ be defined by
	\[
	M^{LC}(dx,dy)=\mu(dx)\,\left(\1_{\{x>0\}}\left(q\,\delta_{ux}(dy)+(1-q)\,\delta_{-dx}(dy)\right)+\1_{\{x\le0\}}\delta_{x}(dy)\right),
	\]
	where $q=\frac{1+d}{u+d}$. Let $\nu$ denote the second marginal of $M^{LC}$. So $\nu$ has density $f_\nu$ defined by $f_\nu(x)=\frac quf_\mu(\frac xu)$ for all $x>0$ and $f_\nu(x)=f_\mu(x)+\frac{1-q}{d}f_\mu(-\frac{x}{d})$ for all $x\le0$. Then $M^{LC}$ is the left-curtain martingale coupling between $\mu$ and $\nu$. One can easily compute $\int_{\R^d}\vert y-x\vert\,M^{LC}(dx,dy)=2\frac{(u-1)(1+d)}{u+d}\int_{\R_+}xf_\mu(x)\,dx$. On the other hand, $W_1(\mu,\nu)=\int_\R\vert F_\mu(t)-F_\nu(t)\vert\,dt$ (see for instance Remark 2.19 (iii) Chapter 2 \cite{VillaniOT2}). From the relation between $f_\nu$ and $f_\mu$, one can deduce that for all $x\ge0$, $F_\nu(x)=1-q+qF_\mu(x/u)$, and for all $x\le0$, $F_\nu(x)=F_\mu(x)+(1-q)\overline F_\mu(-x/d)$, where $\overline F_\mu:x\mapsto\mu((x,+\infty))=1-F_\mu(x)$. Using $\vert x\vert=x+2x^-$, we have
	\begin{align*}
	W_1(\mu,\nu)&=\int_{\R_-}(1-q)\overline F_\mu(-x/d)\,dx+\int_{\R_+}\vert\overline F_\mu(x)-q\overline F_\mu(x/u)\vert\,dx\\
	&=\int_{\R_-}(1-q)\overline F_\mu(-x/d)\,dx+\int_{\R_+}(\overline F_\mu(x)-q\overline F_\mu(x/u))\,dx+2\int_{\R_+}(\overline F_\mu(x)-q\overline F_\mu(x/u))^-\,dx\\
	&=d(1-q)\int_{\R+}xf_\mu(x)\,dx+(1-qu)\int_{\R_+}xf_\mu(x)\,dx+2\int_{\R_+}(\overline F_\mu(x)-q\overline F_\mu(x/u))^-\,dx\\
	&=2\int_{\R_+}(\overline F_\mu(x)-q\overline F_\mu(x/u))^-\,dx.
	\end{align*}
	
	Then $M^{LC}$ satisfies \eqref{inegaliteControleNouvelleFamille} iff
	\begin{equation}\label{contreExempleLCInegaliteControle}
	\frac{(u-1)(1+d)}{u+d}\int_{\R_+}xf_\mu(x)\,dx\le2\int_{\R_+}\left(\bar F_\mu(x)-q\bar F_\mu(x/u)\right)^-\,dx.
	\end{equation}

	%
\end{rk}

The next example illustrates a contradiction of \eqref{contreExempleLCInegaliteControle} and therefore \eqref{inegaliteControleNouvelleFamille} for $M^{LC}$.

\begin{example} Let $\mu(dx)=\lambda\exp(-\lambda x)\1_{\{x>0\}}\,dx$, where $\lambda>0$, and let $\nu$ be the probability distribution with density $f_\nu$ given by $f_\nu(x)=\frac quf_\mu(x/u)$ for $x>0$ and $f_\nu(x)=\frac{1-q}{d}f_\mu(-x/d)$ for $x\le0$. Then for all $x\in\R$, $\overline F_\mu(x)=\exp(-\lambda x)$, and \eqref{contreExempleLCInegaliteControle} is equivalent to 
	\begin{align*}
	&\frac{(u-1)(1+d)}{u+d}\times\frac1\lambda>2\int_{\R_+}\left(\exp(-\lambda x)-q\exp(-\lambda x/u)\right)^-\,dx=2\int_{\frac{\ln q}{\lambda(\frac1u-1)}}^{+\infty}(q\exp(-\lambda x/u)-\exp(-\lambda x))\,dx\\
	&\iff\frac{(u-1)q}{\lambda}>2\left(\frac{qu}{\lambda}\exp\left(-\frac{\ln q}{1-u}\right)-\frac1\lambda\exp\left(-\frac{\ln q}{\frac 1u-1}\right)\right)=2\frac q\lambda(u-1)q^{-1/(1-u)}\\
	&\iff2^{1-u}>q=\frac{1+d}{u+d},
	\end{align*}
	
	which can be satisfied for example with $u=\frac54$ and $d=\frac14$. Note that this condition does not depend on the value of $\lambda$. Therefore, the left-curtain martingale coupling
	\[
	M^{LC}(dx,dy)=\lambda\exp(-\lambda x)\1_{\{x>0\}}\,dx\left(\frac56\,\delta_{\frac{5x}{4}}(dy)+\frac16\,\delta_{-\frac{x}{4}}(dy)\right)
	\]
	does not satisfy \eqref{inegaliteControleNouvelleFamille}, for any $\lambda>0$.
\end{example}

\begin{prop2}\label{mMinimiseCritere} Let $\mu,\nu\in\mathcal P_1(\R)$ be such that $\mu<_{cx}\nu$. Let $Q\in\mathcal Q$. Then the Markov kernel $(\widetilde m^Q(u,dy))_{u\in(0,1)}$ minimises
	\[
	\int_0^1\int_\R\vert F_\nu^{-1}(u)-y\vert\,\widetilde m(u,dy)\,du
	\]
	among all Markov kernels $(\widetilde m(u,dy))_{u\in(0,1)}$ such that
	\begin{equation}\label{defPetitCouplage}
	\left\{
	\begin{array}{rcl}
	\int_{u\in(0,1)}\widetilde m(u,dy)\,du&=&\nu(dy)\\
	\int_\R\vert y\vert\,\widetilde m(u,dy)<+\infty\quad\textrm{and}\quad\int_\R y\,\widetilde m(u,dy)&=&F_\mu^{-1}(u)\text{, $du$-almost everywhere on $(0,1)$}
	\end{array}
	\right.
	\end{equation}
	
	Moreover, $\int_0^1\int_\R\vert F_\nu^{-1}(u)-y\vert\,\widetilde m^Q(u,dy)\,du=W_1(\mu,\nu)$.
\end{prop2}
\begin{rk}\label{mMinimiseCritereRemarque} If $(\widetilde m(u,dy))_{u\in(0,1)}$ is a Markov kernel satisfying \eqref{defPetitCouplage}, then using Lemma \ref{lienPetitmGrandM}, we get that $\mu(dx)\,m(x,dy)$ with $(m(x,dy))_{x\in\R}$ denoting the Markov kernel given by \eqref{defGrandM} is a martingale coupling between $\mu$ and $\nu$.
	
	Conversely, if $\mu(dx)\,m(x,dy)$ is a martingale coupling between $\mu$ and $\nu$, then using the inverse transform sampling, we get that the Markov kernel $(m(F_\mu^{-1}(u),dy))_{u\in(0,1)}$ satisfies \eqref{defPetitCouplage}.	
	%
\end{rk}
\begin{rk} The martingale couplings parametrised by $Q\in\mathcal Q$ are not the only ones to minimise $\int_0^1\int_\R\vert F_\nu^{-1}(u)-y\vert\,\widetilde m(u,dy)\,du$ among all Markov kernels $(\widetilde m(u,dy))_{u\in(0,1)}$ which satisfy \eqref{defPetitCouplage}. Indeed, let $\mu=\frac12\delta_{-1}+\frac12\delta_1$, $\nu=\frac18\delta_{-8}+\frac14\delta_{-6}+\frac58\delta_4$ and
	\[
	M=\frac18\left(2\delta_{(-1,-6)}+2\delta_{(-1,4)}+\delta_{(1,-8)}+3\delta_{(1,4)}\right).
	\]
	
	 For $m(-1,dy)=\frac12\delta_{-6}+\frac12\delta_4$ and $m(1,dy)=\frac14\delta_{-8}+\frac34\delta_4$, we have $M(dx,dy)=\mu(dx)\,m(x,dy)$. Let $(\widetilde m(u,dy))_{u\in(0,1)}$ be defined for all $u\in(0,1)$ by $\widetilde m(u,dy)=m(F_\mu^{-1}(u),dy)$. It is easy to see that $M$ is a martingale coupling between $\mu$ and $\nu$, so $(\widetilde m(u,dy))_{u\in(0,1)}$ satisfies \eqref{defPetitCouplage}. For all $u\in(0,1)$, we have $F_\mu^{-1}(u)=\1_{\{u\le1/2\}}(-1)+\1_{\{u>1/2\}}$ and $F_\nu^{-1}(u)=\1_{\{u\le1/8\}}(-8)+\1_{\{1/8<u\le3/8\}}(-6)+\1_{\{u>3/8\}}\times4$. So for all $u\in(0,1)$, we have
	\[
	\widetilde m(u,dy)=\1_{\{u\le\frac12\}}\left(\frac12\delta_{-6}+\frac12\delta_4\right)+\1_{\{u>\frac12\}}\left(\frac14\delta_{-8}+\frac34\delta_4\right).
	\]
	
	We can compute $\int_0^1\int_\R\vert F_\nu^{-1}(u)-y\vert\,\widetilde m(u,dy)\,du=\frac{17}{4}=\int_0^1\vert F_\mu^{-1}(u)-F_\nu^{-1}(u)\vert\,du=W_1(\mu,\nu)$, so $(\widetilde m(u,dy))_{u\in(0,1)}$ is optimal.
	
	Thanks to Lemma \ref{Petitmdupp}, we can see that for all $Q\in\mathcal Q$, the Markov kernel $\widetilde m^Q$ is such that $\widetilde m^Q(u,\{F_\nu^{-1}(u)\})>0$ for $du$-almost all $u\in(0,1)$. However for all $u\in(0,1/8]$, we have $\widetilde m(u,\{F_\nu^{-1}(u)\})=\widetilde m(u,\{-8\})=0$. Therefore, $\widetilde m$ does not derive from $\mathcal Q$.
\end{rk}


\begin{proof}[Proof of Proposition \ref{mMinimiseCritere}] Let $\widetilde m$ be a Markov kernel satisfying \eqref{defPetitCouplage}. By Jensen's inequality, for $du$-almost every $u\in(0,1)$,
	\[
	\vert F_\nu^{-1}(u)-F_\mu^{-1}(u)\vert=\left\vert\int_\R (F_\nu^{-1}(u)-y)\,\widetilde m(u,dy)\right\vert\le\int_\R\vert F_\nu^{-1}(u)-y\vert\,\widetilde m(u,dy).
	\]
	
	So $\int_0^1\vert F_\nu^{-1}(u)-F_\mu^{-1}(u)\vert\,du\le\int_0^1\int_\R\vert F_\nu^{-1}(u)-y\vert\,\widetilde m(u,dy)\,du$.
	
	Therefore, to conclude, it is sufficient to prove that $\int_\R\vert F_\nu^{-1}(u)-y\vert\,\widetilde m^Q(u,dy)=\vert F_\nu^{-1}(u)-F_\mu^{-1}(u)\vert$, $du$-almost everywhere on $(0,1)$.
	
	Applying \eqref{calculIntegralehMtilde} to the measurable and nonnegative function $h:y\mapsto\vert F_\nu^{-1}(u)-y\vert$ yields for $du$-almost all $u\in(0,1)$
	\begin{align*}
	\int_\R\vert F_\nu^{-1}(u)-y\vert\,\widetilde m^Q(u,dy)&=\int_{(0,1)}\frac{(F_\mu^{-1}-F_\nu^{-1})^+(u)}{F_\nu^{-1}(v)-F_\nu^{-1}(u)}\vert F_\nu^{-1}(u)-F_\nu^{-1}(v)\vert\pi^Q_+(u,dv)\\
	&\phantom{=}+\int_{(0,1)}\frac{(F_\mu^{-1}-F_\nu^{-1})^-(u)}{F_\nu^{-1}(u)-F_\nu^{-1}(v)}\vert F_\nu^{-1}(u)-F_\nu^{-1}(v)\vert\pi^Q_-(u,dv).
	\end{align*}
	
	Using Lemma \ref{Petitmdupp}, we deduce that for $du$-almost all $u\in(0,1)$
	\begin{align*}
	\int_\R\vert F_\nu^{-1}(u)-y\vert\,\widetilde m^Q(u,dy)&=\int_{(0,1)}(F_\mu^{-1}-F_\nu^{-1})^+(u)\pi^Q_+(u,dv)+\int_{(0,1)}(F_\mu^{-1}-F_\nu^{-1})^-(u)\pi^Q_-(u,dv)\\
	&=\vert F_\nu^{-1}(u)-F_\mu^{-1}(u)\vert.
	\end{align*}
\end{proof}

\begin{proof}[Proof of Theorem \ref{GrandMMinimiseCritere}] Let $Q\in\mathcal Q$ and let $\widetilde m^Q$ be the Markov kernel defined by \eqref{defPetitm}. By Lemma \ref{lienPetitmGrandM} and Proposition \ref{mMinimiseCritere},
	\begin{align*}
	\int_{\R\times\R}\vert y-x\vert\,\mu(dx)\,m^Q(x,dy)&=\int_0^1\int_\R\vert y-F_\mu^{-1}(u)\vert\,\widetilde m^Q(u,dy)\,du\\
	&\le \int_0^1\int_\R\vert y-F_\nu^{-1}(u)\vert\,\widetilde m^Q(u,dy)\,du\\
	&\phantom{=}+\int_0^1\int_\R\vert F_\nu^{-1}(u)-F_\mu^{-1}(u)\vert\,\widetilde m^Q(u,dy)\,du\\
	&=2 W_1(\mu,\nu).
	\end{align*}

Since $M^Q(dx,dy)=\mu(dx)\,m^Q(x,dy)$ is a martingale coupling between $\mu$ and $\nu$ (Proposition \ref{muMcouplageMartingale}), we get \eqref{inegaliteControle}.

Let us show now that the constant $2$ in sharp, that is
\[
	\sup_{\substack{\mu,\nu\in\mathcal P_1(\R)\\\mu<_{cx}\nu}}\frac{\mathcal M_1(\mu,\nu)}{W_1(\mu,\nu)}=2.
	\]
	\label{constante2Optimale}

Let $a,b\in\R$ such that $0<a<b$. Let $\mu=\frac12\delta_{-a}+\frac12\delta_a$ and $\nu=\frac12\delta_{-b}+\frac12\delta_b$. Since $\mu$ and $\nu$ are two probability measures with equal means such that $\mu$ is concentrated on $[-a,a]$ and $\nu$ on $\R\backslash[-a,a]$, then $\mu<_{cx}\nu$. Any coupling $H$ between $\mu$ and $\nu$ is of the form
	\[
	H=r\delta_{(-a,-b)}+r'\delta_{(-a,b)}+p\delta_{(a,b)}+p'\delta_{(a,-b)},
	\]
	where $r,r',p,p'\ge0$ and $p+p'=r+r'=p+r'=p'+r=1/2$. One can easily see that $H$ is a martingale coupling iff $b(p-p')=a/2$ and $b(r'-r)=-a/2$, that is
	\begin{equation}\label{uniqueCouplageMartRademacher}
	H=\frac{(b+a)}{4b}\delta_{(-a,-b)}+\frac{(b-a)}{4b}\delta_{(-a,b)}+\frac{(b+a)}{4b}\delta_{(a,b)}+\frac{(b-a)}{4b}\delta_{(a,-b)}.
	\end{equation}
	
	Since there is only one martingale coupling, we trivially have 
	\[
	\mathcal M_1(\mu,\nu)=\int_{\R\times\R}\vert x-y\vert\,H(dx,dy)=\frac{b^2-a^2}{b}\cdot
	\]
	
	On the other hand, since $W_1(\mu,\nu)=\int_\R\vert F_\mu(t)-F_\nu(t)\vert\,dt$ (see for instance Remark 2.19 (iii) Chapter 2 \cite{VillaniOT2}),
	\[
	W_1(\mu,\nu)=\int_{-\infty}^{-b}0\,dt+\int_{-b}^{-a}\frac12\,dt+\int_{-a}^a0\,dt+\int_a^b\frac12\,dt+\int_b^{+\infty}0\,dt=b-a.
	\]

So, we have
\[
\frac{\mathcal M_1(\mu,\nu)}{W_1(\mu,\nu)}=1+\frac{a}{b},
\]
which tends to $2$ as $b$ tends to $a$.
\end{proof}

Also, the stability inequality \eqref{inegaliteControle} does not generalise with $\mathcal M_1(\mu,\nu)$ and $W_1(\mu,\nu)$ replaced by $\mathcal M_\rho(\mu,\nu)$ and $W_\rho(\mu,\nu)$ for $\rho>1$, as shown in the next proposition in general dimension.

\begin{prop2} Let $d\ge1$ and $\rho>1$. Then
	\[
	\sup_{\substack{\mu,\nu\in\mathcal P_\rho(\R^d)\\\mu<_{cx}\nu}}\frac{\mathcal M_\rho(\mu,\nu)}{W_\rho(\mu,\nu)}=+\infty.
	\]
	\label{rho1optimal}
\end{prop2}

The proof of Proposition \ref{rho1optimal} will use the following lemma for the case $1<\rho<2$.
\begin{lemma}\label{minorationXmoinsYpuissanceRho} Let $d\ge1$ and $\rho\in(1,2)$. Let $\vert\cdot\vert$ denote the Euclidean norm on $\R^d$. Then there exists $C_\rho>0$ such that
	\begin{equation}\label{inegaliteProduitScalaire}
	\forall (x,y)\in\R^d\times\R^d,\quad\vert x-y\vert^\rho\ge C_\rho\left(\vert x\vert^\rho-\frac{\rho}{\rho-1}\vert x\vert^{\rho-2}\langle x,y\rangle_{\R^d}+\frac1{\rho-1}\vert y\vert^\rho\right),
	\end{equation}
	where, by convention, for all $y\in\R^d$ and for $x=0$ we choose $\vert x\vert^{\rho-2}\langle x,y\rangle_{\R^d}$ equal to its limit $0$ as $x\to0$.
	
	When $\rho=2$, both sides of the inequality are equal with $C_2=1$.
\end{lemma}
\begin{proof}[Proof of Lemma \ref{minorationXmoinsYpuissanceRho}] If $x=0$, any $C_\rho\le\rho-1$ suits. Else, dividing by $\vert x\vert^\rho$ and using that $y/\vert x\vert$ explores $\R^d$ when $y$ explores $\R^d$, we see that the statement reduces to show that for all $x,y\in\R^d$ such that $\vert x\vert=1$,
\[
\vert x-y\vert^\rho\ge C_\rho\left(1-\frac{\rho}{\rho-1}\langle x,y\rangle_{\R^d}+\frac1{\rho-1}\vert y\vert^\rho\right).
\]

For all $x,y\in\R^d$ such that $\vert x\vert=1$, there exist $y_1,y_2\in\R$ such that $y=y_1x+y_2x^\perp$, where $x^\perp$ is an element of $\operatorname{span}(x)^\perp$ such that $\vert x^\perp\vert=1$. The inequality to prove becomes
\begin{equation}\label{inegaliteProduitScalairePreuve}
\forall (y_1,y_2)\in\R^2,\quad\left((1-y_1)^2+y_2^2\right)^{\rho/2}\ge C_\rho\left(1-\frac{\rho}{\rho-1}y_1+\frac{1}{\rho-1}(y_1^2+y_2^2)^{\rho/2}\right).
\end{equation}

Let $L:(y_1,y_2)\mapsto\left((1-y_1)^2+y_2^2\right)^{\rho/2}$ and $R:(y_1,y_2)\mapsto1-\frac{\rho}{\rho-1}y_1+\frac{1}{\rho-1}(y_1^2+y_2^2)^{\rho/2}$. When $(y_1,y_2)\to(1,0)$, we have
\begin{align*}
	R(y_1,y_2)&=\frac{1}{\rho-1}\left(\rho-1-\rho(y_1-1+1)+(1+2(y_1-1)+(y_1-1)^2+y_2^2)^{\rho/2}\right)\\
	&=\frac{1}{\rho-1}\left(-1-\rho(y_1-1)+1+\rho(y_1-1)+\frac{\rho}{2}(y_1-1)^2+\frac{\rho}{2}y_2^2\right.\\
	&\phantom{=\frac{1}{\rho-1}(}\left.+\rho(\frac{\rho}{2}-1)(y_1-1)^2+o((y_1-1)^2+y_2^2)\right)\\
	&=\frac{1}{\rho-1}\left(\frac{\rho}{2}(y_1-1)^2+\frac{\rho}{2}y_2^2-\rho(1-\frac{\rho}{2})(y_1-1)^2+o((y_1-1)^2+y_2^2)\right).
\end{align*}

Since $\rho<2$, then $L(y_1,y_2)\ge (1-y_1)^2+y_2^2$ for any $(y_1,y_2)$ in the ball centered at $(1,0)$ with radius $1$. So
\begin{align*}
	\limsup_{\substack{(y_1,y_2)\to(1,0)\\(y_1,y_2)\neq(1,0)}}\frac{R(y_1,y_2)}{L(y_1,y_2)}&\le\frac{\rho}{2(\rho-1)},
\end{align*}

On the other hand, when $y_1^2+y_2^2\to+\infty$,
\begin{align*}
	\frac{R(y_1,y_2)}{L(y_1,y_2)}&\sim\frac{(y_1^2+y_2^2)^{\rho/2}}{(\rho-1)(y_1^2+y_2^2)^{\rho/2}}=\frac{1}{\rho-1}\cdot
\end{align*}

So $(y_1,y_2)\mapsto R(y_1,y_2)/L(y_1,y_2)$ is defined and continuous on $(\R^d\times\R^d)\backslash\{(1,0)\}$, bounded from above in the ball centered at $(1,0)$ with radius $1$ and has a finite limit when the norm of $(y_1,y_2)$ tends to $+\infty$. Therefore this function is bounded from above on $(\R^d\times\R^d)\backslash\{(1,0)\}$ by a certain constant $K\ge\frac{1}{\rho-1}$. Since both sides of \eqref{inegaliteProduitScalairePreuve} vanish for $(y_1,y_2)=(1,0)$, we conclude that this inequality holds with constant $C_\rho=\frac1K$ and \eqref{inegaliteProduitScalaire} with constant $C_\rho=\frac1K$.
\end{proof}

\begin{proof}[Proof of Proposition \ref{rho1optimal}] Since all norms on $\R^d$ are equivalent, we can suppose that $\R^d$ is endowed with the Euclidean norm. The case $\rho\ge2$ was addressed in the introduction in the one dimensional case. Its extension to dimension $d$ is immediate. Indeed, for all $n\in\N^*$, let $\mu_n=\mathcal N_1(0,n^2)$ and $\mu'_n(dx_1,\cdots,dx_d)=(x_1,0,\cdots,0)_\sharp\mu_n(dx_1)$ where $\sharp$ denotes the pushforward operation. By reduction to the one dimensional case, we have
	\begin{align*}
	\frac{\mathcal M_\rho(\mu'_n,\mu'_{n+1})}{W_\rho(\mu'_n,\mu'_{n+1})}&=\frac{\mathcal M_\rho(\mu_n,\mu_{n+1})}{W_\rho(\mu_n,\mu_{n+1})}\underset{n\to+\infty}{\longrightarrow}+\infty\cdot
	\end{align*}
	
	We now consider the case $1<\rho<2$. Let $\mu,\nu\in\mathcal P_\rho(\R^d)$ be such that $\mu<_{cx}\nu$, and let $M$ be a martingale coupling between $\mu$ and $\nu$, which exists according to Strassen's theorem or Proposition \ref{muMcouplageMartingale} and Corollary \ref{caracterisationQNonVide}. Thanks to Lemma \ref{minorationXmoinsYpuissanceRho}, there exists $C_\rho>0$ such that
	\begin{align*}
	\int_{\R^d\times\R^d}\vert x-y\vert^\rho\,M(dx,dy)&\ge C_\rho\left(\int_{\R^d}\vert x\vert^\rho\,\mu(dx)-\frac{\rho}{\rho-1}\int_{\R^d\times\R^d}\vert x\vert^{\rho-2}\langle x,y\rangle_{\R^d}\,M(dx,dy)\right.\\
	&\phantom{\ge C_\rho(}\left.+\frac{1}{\rho-1}\int_{\R^d}\vert y\vert^\rho\,\nu(dx)\right).
	\end{align*}
	
	Since $M(dx,dy)=\mu(dx)\,m(x,dy)$ is a martingale coupling, we have for $\mu(dx)$-almost all $x\in\R^d$, $\int_{\R^d}\vert x\vert^{\rho-2}\langle x,y\rangle_{\R^d}\,m(x,dy)=\vert x\vert^\rho$, where both sides are equal to $0$ when $x=0$. So we get
	\[
	\int_{\R^d\times\R^d}\vert x-y\vert^\rho\,M(dx,dy)\ge\frac{C_\rho}{\rho-1}\left(\int_{\R^d}\vert y\vert^\rho\,\nu(dx)-\int_{\R^d}\vert x\vert^\rho\,\mu(dx)\right).
	\]
	
	For all $n\in\N$, let $\mu_n=\mathcal N_d(0,n^2I_d)$. Let $G\sim\mathcal N_d(0,I_d)$. Then for all $n\in\N$, $W_\rho^\rho(\mu_n,\mu_{n+1})\le\E[\vert G\vert^\rho]$ and
	\begin{align*}
	\frac{\mathcal M_\rho^\rho(\mu_n,\mu_{n+1})}{W_\rho^\rho(\mu_n,\mu_{n+1})}&\ge\frac{C_\rho}{\rho-1}\frac{(\E[\vert (n+1)G\vert^\rho]-\E[\vert nG\vert^\rho])}{\E[\vert G\vert^\rho]}\\
	&=\frac{((n+1)^\rho-n^\rho)C_\rho}{\rho-1}\\
	&\sim_{n\to+\infty}\frac{\rho C_\rho}{\rho-1}n^{\rho-1}\underset{n\to+\infty}{\longrightarrow}+\infty.
	\end{align*}

\end{proof}

\section{The inverse transform martingale coupling}
\label{sec:ITMC}

\subsection{Definition of the inverse transform martingale coupling}
\label{sousSec:DefITMC}

Let $\mu,\nu\in\mathcal P_1(\R)$ be such that $\mu<_{cx}\nu$.  We recall that $\Psi_+$ and $\Psi_-$ are defined for all $u\in[0,1]$ by $\Psi_+(u)=\int_0^u(F_\mu^{-1}-F_\nu^{-1})^+(v)\,dv$ and $\Psi_-(u)=\int_0^u(F_\mu^{-1}-F_\nu^{-1})^-(v)\,dv$. Let $\Psi_-^{-1}$ (resp. $\Psi_+^{-1}$) denote the left continuous generalised inverse of $\Psi_-$ (resp. $\Psi_+$). Let $\varphi:[0,1]\to[0,1]$ and $\widetilde\varphi:[0,1]\to[0,1]$ be defined for all $u\in(0,1)$ by
\begin{align*}
\varphi(u)&=\Psi_-^{-1}(\Psi_+(u))=\inf\{r\in[0,1]\mid\Psi_-(r)\ge\Psi_+(u)\};\\
\widetilde\varphi(u)&=\Psi_+^{-1}(\Psi_-(u))=\inf\{r\in[0,1]\mid\Psi_+(r)\ge\Psi_-(u)\},
\end{align*}
which are well defined thanks to the equality $\Psi_-(1)=\Psi_+(1)$, consequence of the equality of the means. 

Let $Q^{IT}$ be the measure defined on $(0,1)^2$ by
\begin{equation}
Q^{IT}(du,dv)=\frac{1}{\gamma}(F_\mu^{-1}-F_\nu^{-1})^+(u)\,du\,\pi^{IT}_+(u,dv)\quad\textrm{where}\ \pi^{IT}_+(u,dv)=\1_{\{0<\varphi(u)<1\}}\,\delta_{\varphi(u)}(dv),
\end{equation}
with $\gamma=\Psi_-(1)=\Psi_+(1)$. According to the next proposition, this measure belongs to $\mathcal Q$.

\begin{prop2}\label{ITMCcasParticulier} Let $\mu,\nu\in\mathcal P_1(\R)$ be such that $\mu<_{cx}\nu$. The measure $Q^{IT}$ is an element of $\mathcal Q$ as defined in Section \ref{MesuresQ}. Moreover,
	\[
	Q^{IT}(du,dv)=\frac{1}{\gamma}(F_\mu^{-1}-F_\nu^{-1})^-(v)\,dv\,\pi^{IT}_-(v,du)\quad\textrm{where $\pi^{IT}_-(v,du)=\1_{\{0<\widetildelow{\varphi}(v)<1\}}\,\delta_{\widetildelow\varphi(v)}(du)$}.
	\]
\end{prop2}

Let us then write $(\widetilde m^{IT}(u,dy))_{u\in(0,1)}$ instead of $(\widetilde m^{Q^{IT}}(u,dy))_{u\in(0,1)}$ and $(m^{IT}(x,dy))_{x\in\R}$ instead of $(m^{Q^{IT}}(x,dy))_{x\in\R}$. By Proposition \ref{muMcouplageMartingale}, the probability measure $M^{IT}(dx,dy)=\mu(dx)\,m^{IT}(x,dy)$ is a martingale coupling between $\mu$ and $\nu$, which we call the inverse transform martingale coupling.

We deduce from the expression of $\pi^{IT}_-$ given in Proposition \ref{ITMCcasParticulier} that the definition of $(\widetilde m^{IT}(u,dy))_{u\in(0,1)}$ reduces to
\begin{equation}
\left\{\begin{array}{r}
\displaystyle\frac{F_\mu^{-1}(u)-F_\nu^{-1}(u)}{F_\nu^{-1}(\varphi(u))-F_\nu^{-1}(u)}\delta_{F_\nu^{-1}(\varphi(u))}(dy)+\left(1-\frac{F_\mu^{-1}(u)-F_\nu^{-1}(u)}{F_\nu^{-1}(\varphi(u))-F_\nu^{-1}(u)}\right)\delta_{F_\nu^{-1}(u)}(dy)\jump\\
\mathrm{if}\ F_\nu^{-1}(\varphi(u))>F_\mu^{-1}(u)>F_\nu^{-1}(u)\text{ and }\varphi(u)<1;\\
\\
\displaystyle\frac{F_\nu^{-1}(u)-F_\mu^{-1}(u)}{F_\nu^{-1}(u)-F_\nu^{-1}(\widetilde\varphi(u))}\delta_{F_\nu^{-1}(\widetildelow\varphi(u))}(dy)+\left(1-\frac{F_\nu^{-1}(u)-F_\mu^{-1}(u)}{F_\nu^{-1}(u)-F_\nu^{-1}(\widetilde\varphi(u))}\right)\delta_{F_\nu^{-1}(u)}(dy)\jump\\
\mathrm{if}\ F_\nu^{-1}(\widetilde\varphi(u))<F_\mu^{-1}(u)<F_\nu^{-1}(u)\text{ and }\widetilde\varphi(u)<1;\\
\\
\delta_{F_\nu^{-1}(u)}(dy)\quad\quad\quad\quad \mathrm{otherwise.}\\
\end{array}
\right.
\label{defPetitmITMC}
\end{equation}

Note that if $F_\mu^{-1}(u)>F_\nu^{-1}(u)$, then by left-continuity of $F_\mu^{-1}$ and $F_\nu^{-1}$, $\Psi_+(u)>0$, which implies $\varphi(u)>0$. Therefore $F_\mu^{-1}(u)>F_\nu^{-1}(u)$ implies $\varphi(u)>0$ so that with the condition $\varphi(u)<1$, $F_\nu^{-1}(\varphi(u))$ makes sense. For similar reasons, if $F_\mu^{-1}(u)<F_\nu^{-1}(u)$ and $\widetilde\varphi(u)<1$ then $F_\nu^{-1}(\widetilde\varphi(u))$ makes sense.

\begin{rk}\label{rqPreuveConstructiveStrassen} We recall the celebrated Strassen theorem: if $\mu,\nu\in\mathcal P_1(\R)$, then $\mu\le_{cx}\nu$ iff there exists a martingale coupling between $\mu$ and $\nu$. The sufficient condition is a straightforward consequence of Jensen's inequality. Indeed, if $M(dx,dy)=\mu(dx)\,m(x,dy)$ is a martingale coupling between $\mu,\nu\in\mathcal P_1(\R)$, then for all convex functions $f:\R\to\R$,
	\[
	\int_\R f(x)\,\mu(dx)=\int_\R f\left(\int_\R y\,m(x,dy)\right)\,\mu(dx)\le\int_{\R^2}f(y)\,m(x,dy)\,\mu(dx)=\int_\R f(y)\,\nu(dy).
	\]
	
	Conversely, suppose that $\mu,\nu\in\mathcal P_1(\R)$ are such that $\mu\le_{cx}\nu$. For $t\in\R$, $\int_\R(t-x)^+\,\mu(dx)\le\int_\R(t-x)^+\,\nu(dx)$ by convexity of $x\in\R\mapsto(t-x)^+$. By the Fubini-Tonelli theorem, $\int_\R(t-x)^+\,\mu(dx)=\int_{-\infty}^tF_\mu(x)\,dx$. Hence $\varphi_\mu(t)=\int_{-\infty}^tF_\mu(x)\,dx\le\varphi_\nu(t)=\int_{-\infty}^tF_\nu(x)\,dx$ for all $t\in\R$. Hence the respective Fenchel-Legendre transforms $\varphi_\mu^*$ and $\varphi_\nu^*$ of $\varphi_\mu$ and $\varphi_\nu$ satisfy $\varphi_\mu^*\ge \varphi_\nu^*$. For all $u\in(0,1)$ and for all $t\in\R$, $F_\mu^{-1}(u)\le t\iff u\le F_\mu(t)$, so
	\[
	\sup_{q\in[0,1]}\left(qt-\int_0^qF_\mu^{-1}(u)\,du\right)=\int_0^{F_\mu(t)}(t-F_\mu^{-1}(u))\,du=\int_0^1(t-F_\mu^{-1}(u))^+\,du=\varphi_\mu(t).
	\]
	
	Since $q\mapsto(\int_0^qF_\mu^{-1}(u)\,du)$ is convex on $[0,1]$, we get the well known fact (see for instance Lemma A.22 \cite{FollmerSchied}) that for all $q\in\R$, $\varphi_\mu^*(q)=(\int_0^qF_\mu^{-1}(u)\,du)\1_{[0,1]}(q)+(+\infty)\1_{[0,1]^\complement}(q)$. Hence
	\begin{equation}\label{inegalitesIntegralesFonctionsQuantiles}
	\int_0^qF_\mu^{-1}(u)\,du\ge\int_0^qF_\nu^{-1}(u)\,du\quad\textrm{for all $q\in[0,1]$, with equality for $q=1$}.
	\end{equation}
	
	We will see in the proof of Proposition \ref{ITMCcasParticulier} that if $\mu\neq\nu$, then \eqref{inegalitesIntegralesFonctionsQuantiles} implies that $Q^{IT}$ belongs to $\mathcal Q$, which ensures that the inverse transform martingale coupling $M^{IT}$ exists. If $\mu=\nu$, the existence of a martingale coupling is straightforward. Therefore, the construction of the inverse transform martingale coupling gives a constructive proof of the necessary condition in Strassen's theorem in dimension $1$.
	
%
   
\end{rk}

\begin{proof}[Proof of Proposition \ref{ITMCcasParticulier}] By Lemma \ref{chgtVariablesf+f-} below,
	\begin{align*}
	Q^{IT}((0,1)^2)=\frac1\gamma\int_0^1(F_\mu^{-1}-F_\nu^{-1})^+(u)\1_{\{0<\varphi(u)<1\}}\,du=\frac1\gamma\int_0^1(F_\mu^{-1}-F_\nu^{-1})^-(u)\1_{\{0<u<1\}}\,du=1,
	\end{align*}
	so $Q^{IT}$ is a probability measure on $(0,1)^2$. Let $h:(0,1)^2\to\R$ be a measurable and bounded function. We have
	\begin{equation}\label{demoITMCcasParticulier1}
		\int_{(0,1)^2}h(u,v)\,Q^{IT}(du,dv=\frac{1}{\gamma}\int_{(0,1)}h(u,\varphi(u))(F_\mu^{-1}-F_\nu^{-1})^+(u)\1_{\{0<\varphi(u)<1\}}\,du.
	\end{equation}
	
	Since $\Psi_-$ is continuous, one has $\Psi_-(\Psi_-^{-1}(u))=u$ for all $u\in(0,1)$. By Lemma \ref{FmoinsUnrondF} below, we deduce that $\widetilde\varphi(\varphi(u))=u$, $(F_\mu^{-1}-F_\nu^{-1})^+(u)\,du$-almost everywhere on $(0,1)$. Therefore, by Lemma \ref{chgtVariablesf+f-} below,
	\begin{align*}
		&\int_{(0,1)}h(u,\varphi(u))(F_\mu^{-1}-F_\nu^{-1})^+(u)\1_{\{0<\varphi(u)<1\}}\,du\\
		&=\int_{(0,1)}h(\widetilde\varphi(\varphi(u)),\varphi(u))(F_\mu^{-1}-F_\nu^{-1})^+(u)\1_{\{0<\varphi(u)<1\}}\1_{\{0<\widetildelow{\varphi}(\varphi(u))<1\}}\,du\\
		&=\int_{(0,1)}h(\widetilde\varphi(v),v)(F_\mu^{-1}-F_\nu^{-1})^-(v)\1_{\{0<v<1\}}\1_{\{0<\widetildelow\varphi(v)<1\}}\,dv\\
		&=\int_{(0,1)^2}h(u,v)(F_\mu^{-1}-F_\nu^{-1})^-(v)\1_{\{0<\widetildelow\varphi(v)<1\}}\,\delta_{\widetildelow\varphi(v)}(du)\,dv.
	\end{align*}

%

So
\begin{align*}
	\int_{(0,1)^2}h(u,v)\,Q^{IT}(du,dv)&=\frac{1}{\gamma}\int_{(0,1)^2}h(u,v)(F_\mu^{-1}-F_\nu^{-1})^-(v)\1_{\{0<\widetildelow\varphi(v)<1\}}\,\delta_{\widetildelow\varphi(v)}(du)\,dv.
\end{align*}
Hence $Q^{IT}(du,dv)=\frac{1}{\gamma}(F_\mu^{-1}-F_\nu^{-1})^-(v)\,dv\,\pi^{IT}_-(v,du)$, where $\pi^{IT}_-(v,du)=\1_{\{0<\widetildelow\varphi(v)<1\}}\,\delta_{\widetildelow\varphi(v)}(du)$. Moreover, since $Q^{IT}$ is a probability measure on $(0,1)^2$, it proves that
\begin{equation}\label{phiEntre0et1}
d\Psi_+(u)\textrm{-a.e.}\ (\textrm{resp. $d\Psi_-(u)$-a.e.)},\quad0<\varphi(u)<1\ \textrm{(resp. $0<\widetilde\varphi(u)<1$)}.
\end{equation}

Therefore, it is clear that $Q^{IT}$ has first marginal $\frac{1}{\gamma}(F_\mu^{-1}-F_\nu^{-1})^+(u)\,du$ and second marginal $\frac{1}{\gamma}(F_\mu^{-1}-F_\nu^{-1})^-(v)\,dv$. For $h:(u,v)\mapsto\1_{\{u<v\}}$, \eqref{demoITMCcasParticulier1} writes
\begin{align*}
	Q^{IT}\left(\{(u,v)\in(0,1)^2\mid u<v\}\right)&=\frac{1}{\gamma}\int_0^1\1_{\{u<\varphi(u)\}}(F_\mu^{-1}-F_\nu^{-1})^+(u)\1_{\{0<\varphi(u)<1\}}\,du.
\end{align*}

Let us show that $u<\varphi(u)$, $(F_\mu^{-1}-F_\nu^{-1})^+(u)\,du$-almost everywhere on $(0,1)$. By the definition of $\varphi$ and Lemma \ref{FmoinsUnrondF}, for all $u\in(0,1)$, $\varphi(u)\le u\iff\Psi_-^{-1}(\Psi_+(u))\le u\iff\Psi_+(u)\le\Psi_-(u)$. Recall that since $\mu\le_{cx}\nu$, according \eqref{inegalitesIntegralesFonctionsQuantiles}, for all $u\in(0,1)$, $\int_0^uF_\mu^{-1}(v)\,dv\ge\int_0^uF_\nu^{-1}(v)\,dv$, so $\Psi_+(u)\ge\Psi_-(u)$. Therefore, we get that
\begin{equation}\label{phiPlusPetitQueU}
\forall u\in(0,1),\quad\varphi(u)\le u\iff \Psi_+(u)=\Psi_-(u). 
\end{equation}

Suppose $F_\mu^{-1}(u)>F_\nu^{-1}(u)$. Since $F_\mu^{-1}$ and $F_\nu^{-1}$ are left continuous, this implies $F_\mu^{-1}(u-\varepsilon)>F_\nu^{-1}(u-\varepsilon)$ for $\varepsilon>0$ small enough. So, for $\varepsilon>0$ small enough, $\Psi_-(u)=\Psi_-(u-\varepsilon)\le\Psi_+(u-\varepsilon)<\Psi_+(u)$, which implies 
\begin{equation}\label{phiPlusGrandQueU}
u<\varphi(u),\quad(F_\mu^{-1}-F_\nu^{-1})^+(u)\,du\textrm{-almost everywhere on $(0,1)$}.
\end{equation}

So
\begin{align*}
	Q^{IT}\left(\{(u,v)\in(0,1)^2\mid u<v\}\right)&=\frac{1}{\gamma}\int_0^1(F_\mu^{-1}-F_\nu^{-1})^+(u)\1_{\{0<\varphi(u)<1\}}\,du=Q^{IT}((0,1)^2)=1,
\end{align*}
since $Q^{IT}$ is a probability measure on $(0,1)^2$.
\end{proof}

Let us now suppose that $\mu,\nu\in\mathcal P_1(\R)$ are such that $\mu<_{cx}\nu$ and there exists $p\in(0,1)$ such that $u\mapsto\int_0^u(F_\mu^{-1}(v)-F_\nu^{-1}(v))\,dv$ is nondecreasing on $[0,p]$ and nonincreasing on $[p,1]$. We saw in Example \ref{exempleQ1} a concrete example of an element $Q_1\in\mathcal Q$. Any probability measure $Q$ defined on $(0,1)$ satisfying properties $(i)$ and $(ii)$ of the definition of $\mathcal Q$ is concentrated on $(0,p)\times(p,1)$ and therefore satisfies $(iii)$. The probability measure $Q_1$ is a simple example that comes to mind. The inverse transform martingale coupling presented in this section is a valid example as well and inspires another coupling which is sort of the nonincreasing twin of the inverse transform martingale coupling.

Let $\chi_-:u\in[0,1]\mapsto\int_{1-u}^1(F_\mu^{-1}-F_\nu^{-1})^-(v)\,dv=\int_0^u(F_\mu^{-1}-F_\nu^{-1})^-(1-v)\,dv$, $\chi_+:u\in[0,1]\mapsto\int_0^{u}(F_\mu^{-1}-F_\nu^{-1})^+(v)\,dv$ and $\Gamma=\chi_-^{-1}\circ\chi_+$ where $\chi_-^{-1}$ denotes the left continuous generalised inverse of $\chi_-$, that is
\[
\Gamma:u\in[0,1]\mapsto\inf\{r\in[0,1]\mid\chi_-(r)\ge\chi_+(u)\},
\]
which is well defined since $\chi_+(1)=\chi_+(p)=\chi_-(1-p)=\gamma$, consequence of the equality of the means. Let $Q^{NIT}$ be the probability measure defined on $(0,1)^2$ by
\begin{equation}
\label{QnonincreasingITMC}
Q^{NIT}(du,dv)=\frac{1}{\gamma}(F_\mu^{-1}-F_\nu^{-1})^+(u)\,du\,\pi^{NIT}_+(u,dv)\quad\textrm{where $\pi^{NIT}_+(u,dv)=\1_{\{\Gamma(u)>0\}}\,\delta_{1-\Gamma(u)}(dv)$}.
\end{equation}

\begin{prop2}\label{QNITappartientAQ} Let $\mu,\nu\in\mathcal P_1(\R)$ be such that $\mu<_{cx}\nu$. Assume that there exists $p\in(0,1)$ such that $u\mapsto\int_0^u(F_\mu^{-1}(v)-F_\nu^{-1}(v))\,dv$ is nondecreasing on $[0,p]$ and nonincreasing on $[p,1]$. Then $Q^{NIT}\in\mathcal Q$.
\end{prop2}

In the symmetric case, that is when $\mu$ and $\nu$ are symmetric and $p=1/2$, we have $\Gamma(u)=u$ and therefore $Q^{NIT}=Q_2$ (see \eqref{exempleMesureQ2}). Hence $Q^{NIT}$ is a generalisation of the symmetric coupling.

\begin{proof}[Proof of Proposition \ref{QNITappartientAQ}] Note that $\Gamma(1)\le1-p$, hence $\Gamma(u)<1$ for all $u\in(0,1)$. It is clear that $Q^{NIT}$ satisfies property $(i)$ of the definition of $\mathcal Q$. By Lemma \ref{chgtVariablesf+f-} below applied with the functions $f_1:u\in(0,1)\mapsto(F_\mu^{-1}-F_\nu^{-1})^+(u)$ and $f_2:u\in(0,1)\mapsto(F_\mu^{-1}-F_\nu^{-1})^-(1-u)$, we have
\begin{align*}
\frac{1}{\gamma}\int_0^1(F_\mu^{-1}-F_\nu^{-1})^+(u)h(1-\Gamma(u))\1_{\{\Gamma(u)>0\}}\,du&=\frac{1}{\gamma}\int_0^1(F_\mu^{-1}-F_\nu^{-1})^-(1-v)h(1-v)\,dv\\
&=\frac1\gamma\int_0^1(F_\mu^{-1}-F_\nu^{-1})^{-1}(v)h(v)\,dv,
\end{align*}
for any measurable and bounded function $h:(0,1)\to\R$. So $Q^{NIT}$ satisfies $(ii)$ as well, and therefore $(iii)$.
\end{proof}

We saw with Proposition \ref{mMinimiseCritere} that for all $Q\in\mathcal Q$, $\int_0^1\int_\R\vert F_\nu^{-1}(u)-y\vert\,\widetilde m^Q(u,dy)\,du=W_1(\mu,\nu)$. The next proposition shows that the inverse transform martingale coupling and its nonincreasing twin, when it exists, play particular roles among the martingale couplings which derive from $\mathcal Q$ when $\vert F_\nu^{-1}(u)-y\vert$ is replaced by $\vert F_\nu^{-1}(u)-y\vert^\rho$ with $\rho\in\R$.

\begin{prop2}\label{thmpropertyMaxMinITMC} Let $\mu,\nu\in\mathcal P_1(\R)$ be such that $\mu<_{cx}\nu$. For all $\rho\in\R$ and for any Markov kernel $(\widetilde m(u,dy))_{u\in(0,1)}$, let $\mathcal C_\rho(\widetilde m)$ be defined by
\begin{equation}
\mathcal C_\rho(\widetilde m)=\int_{\R\times(0,1)}\vert F_\nu^{-1}(u)-y\vert^\rho\1_{\{y\neq F_\nu^{-1}(u)\}}\,\widetilde m(u,dy)\,du.
\end{equation}

Then, for all $Q\in\mathcal Q$,
\begin{align}\label{ITminimiseOuMaximise}
\begin{split}
\forall\rho\in(-\infty,1]\cup[2,+\infty),&\quad\mathcal C_\rho(\widetilde m^{IT})\le\mathcal C_\rho(\widetilde m^Q);\\
\forall\rho\in[1,2],&\quad\mathcal C_\rho(\widetilde m^Q)\le\mathcal C_\rho(\widetilde m^{IT});\\
\forall\rho\in\{1,2\},&\quad\mathcal C_\rho(\widetilde m^{IT})=\mathcal C_\rho(\widetilde m^Q).
\end{split}
\end{align}

Let us now assume that there exists $p\in(0,1)$ such that $u\mapsto\int_0^u(F_\mu^{-1}(v)-F_\nu^{-1}(v))\,dv$ is nondecreasing on $[0,p]$ and nonincreasing on $[p,1]$ and denote $(\widetilde m^{NIT}(u,dy))_{u\in(0,1)}$ for $(\widetilde m^{Q^{NIT}}(u,dy))_{u\in(0,1)}$. Then, for all $Q\in\mathcal Q$,
\begin{align}\label{NITminimiseOuMaximise}
\begin{split}
\forall\rho\in(-\infty,1]\cup[2,+\infty),&\quad\mathcal C_\rho(\widetilde m^Q)\le\mathcal C_\rho(\widetilde m^{NIT});\\
\forall\rho\in[1,2],&\quad\mathcal C_\rho(\widetilde m^{NIT})\le\mathcal C_\rho(\widetilde m^Q);\\
\forall\rho\in\{1,2\},&\quad\mathcal C_\rho(\widetilde m^{NIT})=\mathcal C_\rho(\widetilde m^Q).
\end{split}
\end{align}
\end{prop2}
\begin{rk} Let $\mu,\nu\in\mathcal P_1(\R)$ be such that $\mu<_{cx}\nu$. By Proposition \ref{thmpropertyMaxMinITMC} for $\rho=0$, we deduce that
	\[
	\sup_{Q\in\mathcal Q}\{\PP(Y=F_\nu^{-1}(U))\mid(U,Y)\sim\1_{(0,1)}(u)\,du\,\widetilde m^Q(u,dy)\}
	\]
	is attained for the inverse transform martingale coupling.
	
	Suppose in addition that $F_\nu^{-1}$ is constant on the intervals of the form $(F_\mu(x_-),F_\mu(x)]$. Let $M(dx,dy)=\mu(dx)\,m(x,dy)$ be a martingale coupling between $\mu$ and $\nu$. Let $(\widetilde m(u,dy))_{u\in(0,1)}$ be the kernel defined for all $u\in(0,1)$ by $\widetilde m(u,dy)=m(F_\mu^{-1}(u),dy)$.  Let $T$ be the Monge transport map. According to \eqref{LemmeMongeTransportMap}, $F_\nu^{-1}(u)=F_\nu^{-1}(F_\mu(F_\mu^{-1}(u)))$ for $du$-almost all $u\in(0,1)$. So by Lemma \ref{lienPetitmGrandM}, for all $\rho\in\R$,
	\begin{align*}
	\int_{\R\times\R}\vert y-T(x)\vert^\rho\1_{\{y\neq T(x)\}}\,\mu(dx)\,m(x,dy)&=\int_0^1\int_\R\vert y-T(F_\mu^{-1}(u))\vert^\rho\1_{\{y\neq T(F_\mu^{-1}(u))\}}\,\widetilde m(u,dy)\,du\\
	&=\int_0^1\int_\R\vert y-F_\nu^{-1}(u)\vert^\rho\1_{\{y\neq F_\nu^{-1}(u)\}}\,\widetilde m(u,dy)\,du.
	\end{align*}
	
	We deduce that the supremum of $\PP(Y=T(X))$ among all random variables $X$ and $Y$ such that $(X,Y)\sim M^Q$ for $Q\in\mathcal Q$ is attained for the inverse transform martingale coupling.
\end{rk}
\begin{proof}[Proof of Proposition \ref{thmpropertyMaxMinITMC}] Let $\rho\in\R$ and $Q\in\mathcal Q$. Let $\varepsilon>0$ and $f_\varepsilon:\R\to\R$ be defined for all $x\in\R$ by
	\[
	f_\varepsilon(x)=\varepsilon^{\rho-2}\left((\rho-1)x+(2-\rho)\varepsilon\right)\1_{\{x\le\varepsilon\}}+x^{\rho-1}\1_{\{x>\varepsilon\}}.
	\]
	
	It is clear that $f_\varepsilon$ is convex for $\rho\in(-\infty,1]\cup[2,+\infty)$ and concave for $\rho\in[1,2]$. Let $c_\varepsilon:(0,1)^2\to\R$ be the right-continuous function defined for all $(u,v)\in(0,1)^2$ by $c_\varepsilon(u,v)=f_\varepsilon(\vert F_\nu^{-1}(u_+)-F_\nu^{-1}(v_+)\vert)$.
	
	If $\rho\in(-\infty,1]\cup[2,+\infty)$, then $c_\varepsilon$ satisfies the Monge condition, that is for all $u,u',v,v'\in(0,1)$ such that $u\le u'$ and $v\le v'$,
	\[
	c_\varepsilon(u',v')-c_\varepsilon(u,v')-c_\varepsilon(u',v)+c_\varepsilon(u,v)\le0,
	\]
	which follows from the monotonicity of $F_\nu^{-1}$ and the fact that $(x,y)\mapsto f_\varepsilon(\vert x-y\vert)$ is convex and therefore satisfies the Monge condition. Since $Q$ has marginals $d\Psi_+/\gamma$ and $d\Psi_-/\gamma$, by Theorem 3.1.2 Chapter 3 \cite{RachevRuschendorf}, we have
	\[
	\int_0^1c_\varepsilon(\Psi_+^{-1}(\gamma u),\Psi_-^{-1}(\gamma u))\,du\le\int_{(0,1)^2}c_\varepsilon(u,v)\,Q(du,dv)\le\int_0^1c_\varepsilon(\Psi_+^{-1}(\gamma u),\Psi_-^{-1}(\gamma(1-u)))\,du.
	\]
	
	It is easy to check that for all $u,v\in(0,1)$, the map $(0,1)\ni\varepsilon\mapsto c_\varepsilon(u,v)$ is nonincreasing, bounded from below by $2-\rho$ and converges to $\vert F_\nu^{-1}(u_+)-F_\nu^{-1}(v_+)\vert^{\rho-1}$ when $\varepsilon\to0$ where by convention, we choose $0^0=1$ and for all $\alpha<0$ and $x=0$, we choose $x^\alpha$ equal to its limit $+\infty$ as $x\to0_+$. Therefore, by the monotone convergence theorem for $\varepsilon\to0$, we have
	\begin{multline}\label{doubleInegaliteConditionDeMonge}
	\forall\rho\in(-\infty,1]\cup[2,\infty),\quad\int_0^1\vert F_\nu^{-1}(\Psi_+^{-1}(\gamma u)_+)-F_\nu^{-1}(\Psi_-^{-1}(\gamma u)_+)\vert^{\rho-1}\,du\\
	\le \int_{(0,1)^2}\vert F_\nu^{-1}(u_+)-F_\nu^{-1}(v_+)\vert^{\rho-1}\,Q(du,dv)\\
	\le \int_0^1\vert F_\nu^{-1}(\Psi_+^{-1}(\gamma u)_+)-F_\nu^{-1}(\Psi_-^{-1}(\gamma(1-u))_+)\vert^{\rho-1}\,du.
	\end{multline}
	
	If $1\le\rho\le2$, then $(x,y)\mapsto f_\varepsilon(\vert x-y\vert)$ is concave so $-c_\varepsilon$ satisfies the Monge conditions and a symmetric reasoning shows that
	\begin{equation}\label{doubleEncadrementCrhoEpsilon}
	\int_0^1c_\varepsilon(\Psi_+^{-1}(\gamma u),\Psi_-^{-1}(\gamma(1-u)))\,du\le\int_{(0,1)^2}c_\varepsilon(u,v)\,Q(du,dv)\le\int_0^1c_\varepsilon(\Psi_+^{-1}(\gamma u),\Psi_-^{-1}(\gamma u))\,du.
	\end{equation}
	
	It is easy to check that for all $u,v\in(0,1)$, the map $(0,1)\ni\varepsilon\mapsto c_\varepsilon(u,v)$ is bounded from above by $1+\vert F_\nu^{-1}(u_+)-F_\nu^{-1}(v_+)\vert^{\rho-1}$ and converges to its lower bound $\vert F_\nu^{-1}(u_+)-F_\nu^{-1}(v_+)\vert^{\rho-1}$ when $\varepsilon\to0$. Consider one of the three integrals in \eqref{doubleEncadrementCrhoEpsilon}. If the pointwise limit for $\varepsilon\to0$ of its integrand is integrable, then we can apply the dominated convergence theorem. Otherwise, the integral is infinite for all $\varepsilon\in(0,1)$. Therefore, for $\varepsilon\to0$, we have
	\begin{multline}\label{doubleInegaliteConditionDeMonge2}
	\forall 1\le\rho\le2,\quad\int_0^1\vert F_\nu^{-1}(\Psi_+^{-1}(\gamma u)_+)-F_\nu^{-1}(\Psi_-^{-1}(\gamma(1-u))_+)\vert^{\rho-1}\,du\\
	\le \int_{(0,1)^2}\vert F_\nu^{-1}(u_+)-F_\nu^{-1}(v_+)\vert^{\rho-1}\,Q(du,dv)\\
	\le\int_0^1\vert F_\nu^{-1}(\Psi_+^{-1}(\gamma u)_+)-F_\nu^{-1}(\Psi_-^{-1}(\gamma u)_+)\vert^{\rho-1}\,du.
	\end{multline}
	
	For all $\rho\in\R$, applying \eqref{calculIntegralehMtilde} to the measurable and nonnegative function $h:y\mapsto\vert F_\nu^{-1}(u)-y\vert^\rho\1_{\{y\neq F_\nu^{-1}(u)\}}$ yields $du$-almost everywhere on $(0,1)$,
	\begin{align*}
	&\int_\R\vert F_\nu^{-1}(u)-y\vert^\rho\1_{\{y\neq F_\nu^{-1}(u)\}}\,\widetilde m^Q(u,dy)&\\
	&=\int_{(0,1)}\frac{(F_\mu^{-1}-F_\nu^{-1})^+(u)}{F_\nu^{-1}(v)-F_\nu^{-1}(u)}\vert F_\nu^{-1}(u)-F_\nu^{-1}(v)\vert^\rho\1_{\{F_\nu^{-1}(v)\neq F_\nu^{-1}(u)\}}\,\pi^Q_+(u,dv)\\
	&\phantom{=\ }+\int_{(0,1)}\frac{(F_\mu^{-1}-F_\nu^{-1})^-(u)}{F_\nu^{-1}(u)-F_\nu^{-1}(v)}\vert F_\nu^{-1}(u)-F_\nu^{-1}(v)\vert^\rho\1_{\{F_\nu^{-1}(v)\neq F_\nu^{-1}(u)\}}\,\pi^Q_-(u,dv),
	\end{align*}
	
	where according to Lemma \ref{Petitmdupp}, for $(F_\mu^{-1}-F_\nu^{-1})^+(u)\,du$-almost all $u\in(0,1)$, $\pi^Q_+(u,dv)$-a.e., $F_\nu^{-1}(v)>F_\nu^{-1}(u)$ and for $(F_\mu^{-1}-F_\nu^{-1})^-(u)\,du$-almost all $u\in(0,1)$, $\pi^Q_-(u,dv)$-a.e, $F_\nu^{-1}(v)<F_\nu^{-1}(u)$. We deduce that
	\begin{align}\label{expressionCrho}
	\begin{split}
	C_\rho(\widetilde m^Q)&=\int_{(0,1)^2}(F_\mu^{-1}-F_\nu^{-1})^+(u)\vert F_\nu^{-1}(u)-F_\nu^{-1}(v)\vert^{\rho-1}\,du\,\pi^Q_+(u,dv)\\
	&\phantom{=\ }+\int_{(0,1)^2}(F_\mu^{-1}-F_\nu^{-1})^-(u)\vert F_\nu^{-1}(u)-F_\nu^{-1}(v)\vert^{\rho-1}\,du\,\pi^Q_-(u,dv)\\
	&=2\gamma\int_{(0,1)^2}\vert F_\nu^{-1}(u)-F_\nu^{-1}(v)\vert^{\rho-1}\,Q(du,dv).
	\end{split}
	\end{align}
	
	Since the set of discontinuities of $F_\nu^{-1}$ is at most countable and since the marginals of $Q$ have densities, we have
	\begin{equation}\label{CrhoExpression}
	C_\rho(\widetilde m^Q)=2\gamma\int_{(0,1)^2}\vert F_\nu^{-1}(u_+)-F_\nu^{-1}(v_+)\vert^{\rho-1}\,Q(du,dv).
	\end{equation}
	
	Let us show that 
	\begin{equation}\label{CrhoIT}
	\mathcal C_\rho(\widetilde m^{IT})=2\gamma\int_0^1\vert F_\nu^{-1}(\Psi_+^{-1}(\gamma u)_+)-F_\nu^{-1}(\Psi_-^{-1}(\gamma u)_+)\vert^{\rho-1}\,du.
	\end{equation}
	
	By Lemma \ref{FmoinsUnrondF}, $\Psi_+^{-1}(\Psi_+(u))=u$, $d\Psi_+(u)$-almost everywhere on $(0,1)$, so using \eqref{expressionCrho}, Proposition \ref{prop410RevuzYor} and the fact that $0<\Psi_\pm^{-1}(u)<1$ for all $u\in(0,\gamma)$, we have
	\begin{align*}
	\mathcal C_\rho(\widetilde m^{IT})&=2\int_0^1(F_\mu^{-1}-F_\nu^{-1})^+(u)\vert F_\nu^{-1}(u)-F_\nu^{-1}(\varphi(u))\vert^{\rho-1}\1_{\{0<\varphi(u)<1\}}\,du\\
	&=2\int_0^1\vert F_\nu^{-1}(\Psi_+^{-1}(\Psi_+(u)))-F_\nu^{-1}(\Psi_-^{-1}(\Psi_+(u)))\vert^{\rho-1}\1_{\{0<\Psi_-^{-1}(\Psi_+(u))<1\}}\,d\Psi_+(u)\\
	&=2\int_0^\gamma \vert F_\nu^{-1}(\Psi_+^{-1}(u))-F_\nu^{-1}(\Psi_-^{-1}(u))\vert^{\rho-1}\1_{\{0<\Psi_-^{-1}(u)<1\}}\,du\\
	&=2\gamma\int_0^1\vert F_\nu^{-1}(\Psi_+^{-1}(\gamma u))-F_\nu^{-1}(\Psi_-^{-1}(\gamma u))\vert^{\rho-1}\,du.
	\end{align*}
	
	Since the set of discontinuities of $\Psi_+^{-1}$, $\Psi_-^{-1}$, $(\Psi_+\circ F_\nu)^{-1}=F_\nu^{-1}\circ\Psi_+^{-1}$ and $(\Psi_-\circ F_\nu)^{-1}=F_\nu^{-1}\circ\Psi_-^{-1}$ are at most countable, we get that for $du$-almost all $u\in(0,1)$, $F_\nu^{-1}(\Psi_+^{-1}(\gamma u))=F_\nu^{-1}\circ\Psi_+^{-1}(\gamma u_+)=F_\nu^{-1}(\Psi_+^{-1}(\gamma u)_+)$ and $F_\nu^{-1}(\Psi_-^{-1}(\gamma u))=F_\nu^{-1}(\Psi^{-1}(\gamma u)_+)$, which proves \eqref{CrhoIT}. Then \eqref{ITminimiseOuMaximise} is deduced from \eqref{doubleInegaliteConditionDeMonge}, \eqref{doubleInegaliteConditionDeMonge2}, \eqref{CrhoExpression} and \eqref{CrhoIT}. 
	
	Assume now that there exists $p\in(0,1)$ such that $u\mapsto\int_0^u(F_\mu^{-1}(v)-F_\nu^{-1}(v))\,dv$ is nondecreasing on $[0,p]$ and nonincreasing on $[p,1]$. For all $u\in(0,1)$, $\chi_+(u)=\Psi_+(u)$ and $\chi_-(u)=\gamma-\Psi_-(1-u)$. If $U$ is a random variable uniformly distributed on $(0,1)$, one can easily check that $1-\Psi_-^{-1}(\gamma(1-U))$ has distribution $d\chi_-/\gamma$. Since $u\mapsto 1-\Psi_-^{-1}(\gamma(1-u))$ is nondecreasing, it is shown in Lemma A.3 \cite{ApproxMOTJourdain2} that $1-\Psi_-^{-1}(\gamma(1-u))=\chi_-^{-1}(\gamma u)$, $du$-almost everywhere on $(0,1)$. So we show with similar arguments as above that
	\begin{equation}\label{CrhoNIT}
	\mathcal C_\rho(\widetilde m^{NIT})=2\gamma\int_0^1\vert F_\nu^{-1}(\Psi_+^{-1}(\gamma u)_+)-F_\nu^{-1}(\Psi_-^{-1}(\gamma(1-u))_+)\vert^{\rho-1}\,du.
	\end{equation}
	
	Then \eqref{NITminimiseOuMaximise} is deduced from \eqref{doubleInegaliteConditionDeMonge}, \eqref{doubleInegaliteConditionDeMonge2}, \eqref{CrhoExpression} and \eqref{CrhoNIT}.
\end{proof}

\subsection{Stability of the inverse transform martingale coupling with respect to the marginal laws}

In this section, we show that the inverse transform martingale coupling is stable with respect to its marginals $\mu$ and $\nu$ for the Wasserstein distance topology.

\begin{prop2} Let $\mu,\nu\in\mathcal P_1(\R)$ be such that $\mu<_{cx}\nu$. Let $(\mu_n)_{n\in\N}$ and $(\nu_n)_{n\in\N}$ be two sequences of probability measures on $\R$ with finite first moments such that for all $n\in\N$, $\mu_n<_{cx}\nu_n$. For all $n\in\N$, let $M^{IT}_n$ (resp. $M^{IT}$) be the inverse transform martingale coupling between $\mu_n$ and $\nu_n$ (resp. between $\mu$ and $\nu$).
	
	If $W_1(\mu_n,\mu)\underset{n\to+\infty}{\longrightarrow}0$ and $W_1(\nu_n,\nu)\underset{n\to+\infty}{\longrightarrow}0$, then
	\[
	W_1(M^{IT}_n,M^{IT})\underset{n\to+\infty}{\longrightarrow}0.
	\]
\end{prop2}
\begin{proof} For all $n\in\N$, let $\Psi_{n+}:u\in[0,1]\mapsto\int_0^u(F_{\mu_n}^{-1}-F_{\nu_n}^{-1})^+(v)\,dv$, $\Psi_{n-}:u\in[0,1]\mapsto\int_0^u(F_{\mu_n}^{-1}-F_{\nu_n}^{-1})^-(v)\,dv$, $\varphi_n=\Psi_{n-}^{-1}\circ\Psi_{n+}$ and $\widetilde\varphi_n=\Psi_{n+}^{-1}\circ\Psi_{n-}$. Let $(\widetilde m^{IT}_n(u,dy))_{u\in(0,1)}$ be the Markov kernel defined as in \eqref{defPetitmITMC} with $F_\mu^{-1},F_\nu^{-1},\varphi$ and $\widetilde\varphi$ respectively replaced by $F_{\mu_n}^{-1},F_{\nu_n}^{-1},\varphi_n$ and $\widetilde\varphi_n$, and let $(m^{IT}_n(u,dy))_{u\in(0,1)}$ be defined as in \eqref{defGrandM} with $\widetilde m$ replaced by $\widetilde m^{IT}_n$. Let $h:\R^2\to\R$ be a bounded and continuous function such that $h$ is Lipschitz continuous with respect to its second variable. Then by Lemma \ref{lienPetitmGrandM},
	\begin{align*}
	\int_{\R\times\R}h(x,y)\,\mu_n(dx)\,m^{IT}_n(x,dy)=\int_{(0,1)}\int_\R h(F_{\mu_n}^{-1}(u),y)\,\widetilde m^{IT}_n(u,dy)\,du.
	\end{align*}
	
	According to \eqref{calculIntegralehMtilde}, for any measurable and bounded function $\widetilde h:\R\to\R$ and for $du$-almost all $u\in(0,1)$,
	\begin{align*}
	\int_\R\widetilde h(y)\,\widetilde m^{IT}_n(u,dy)&=\widetilde h(F_{\nu_n}^{-1}(u))+\int_{(0,1)}\frac{(F_{\mu_n}^{-1}(u)-F_{\nu_n}^{-1}(u))^+}{F_{\nu_n}^{-1}(v)-F_{\nu_n}^{-1}(u)}(\widetilde h(F_{\nu_n}^{-1}(v))-\widetilde h(F_{\nu_n}^{-1}(u)))\,\pi^{IT}_{n+}(u,dv)\\
	&\phantom{=}+\int_{(0,1)}\frac{(F_{\mu_n}^{-1}(u)-F_{\nu_n}^{-1}(u))^-}{F_{\nu_n}^{-1}(u)-F_{\nu_n}^{-1}(v)}(\widetilde h(F_{\nu_n}^{-1}(v))-\widetilde h(F_{\nu_n}^{-1}(u)))\,\pi^{IT}_{n-}(u,dv),
	\end{align*}
	where $\pi^{IT}_{n+}(u,dv)=\1_{\{0<\varphi_n(u)<1\}}\,\delta_{\varphi_n(u)}(dv)$ and $\pi^{IT}_{n-}(v,du)=\1_{\{0<\widetildelow\varphi_n(v)<1\}}\,\delta_{\widetildelow\varphi_n(v)}(du)$, and all the integrands are well defined thanks to Lemma \ref{Petitmdupp}. According to \eqref{phiEntre0et1}, $0<\varphi_n(u)<1$, $(F_{\mu_n}^{-1}(u)-F_{\nu_n}^{-1}(u))^+\,du$-almost everywhere and $0<\widetilde\varphi_n(u)<1$, $(F_{\mu_n}^{-1}(u)-F_{\nu_n}^{-1}(u))^-\,du$ almost-everywhere. So
	\begin{align*}
	&\int_{(0,1)}\int_\R h(F_{\mu_n}^{-1}(u),y)\,\widetilde m^{IT}_n(u,dy)\,du\\
	&=\int_{(0,1)}h(F_{\mu_n}^{-1}(u),F_{\nu_n}^{-1}(u))\,du\\
	&\phantom{=}+\int_{(0,1)}\frac{(F_{\mu_n}^{-1}(u)-F_{\nu_n}^{-1}(u))^+}{F_{\nu_n}^{-1}(\varphi_n(u))-F_{\nu_n}^{-1}(u)}(h(F_{\mu_n}^{-1}(u),F_{\nu_n}^{-1}(\varphi_n(u)))-h(F_{\mu_n}^{-1}(u),F_{\nu_n}^{-1}(u)))\,du\\
	&\phantom{=}+\int_{(0,1)}\frac{(F_{\mu_n}^{-1}(u)-F_{\nu_n}^{-1}(u))^-}{F_{\nu_n}^{-1}(u)-F_{\nu_n}^{-1}(\widetilde\varphi_n(u))}(h(F_{\mu_n}^{-1}(u),F_{\nu_n}^{-1}(\widetilde\varphi_n(u)))-h(F_{\mu_n}^{-1}(u),F_{\nu_n}^{-1}(u)))\,du.
	\end{align*}
	
	Since $\mu_n$ converges weakly towards $\mu$, then $F_{\mu_n}^{-1}(u)$ (resp. $F_{\nu_n}^{-1}(u)$) converges towards $F_\mu^{-1}(u)$ (resp. $F_\nu^{-1}(u)$) $du$-almost everywhere on $(0,1)$. Since $h$ is continuous and bounded, by the dominated convergence theorem,
	\[
	\int_{(0,1)}h(F_{\mu_n}^{-1}(u),F_{\nu_n}^{-1}(u))\,du\underset{n\to+\infty}{\longrightarrow}\int_{(0,1)}h(F_\mu^{-1}(u),F_\nu^{-1}(u))\,du.
	\]
	
	On the other hand, using Lemma \ref{FmoinsUnrondF} for the first equality, then Proposition \ref{prop410RevuzYor} for the second equality and the change of variables $u=\Psi_{n+}(1)v$ with the equality $\Psi_{n+}(1)=\Psi_{n-}(1)$ for the last equality, we have
	\begin{align*}
	&\int_{(0,1)}\frac{(F_{\mu_n}^{-1}(u)-F_{\nu_n}^{-1}(u))^+}{F_{\nu_n}^{-1}(\varphi_n(u))-F_{\nu_n}^{-1}(u)}(h(F_{\mu_n}^{-1}(u),F_{\nu_n}^{-1}(\varphi_n(u)))-h(F_{\mu_n}^{-1}(u),F_{\nu_n}^{-1}(u)))\,du\\
	&=\int_{(0,1)}\frac{h(F_{\mu_n}^{-1}(\Psi_{n+}^{-1}(\Psi_{n+}(u))),F_{\nu_n}^{-1}(\Psi_{n-}^{-1}(\Psi_{n+}(u))))-h(F_{\mu_n}^{-1}(\Psi_{n+}^{-1}(\Psi_{n+}(u))),F_{\nu_n}^{-1}(\Psi_{n+}^{-1}(\Psi_{n+}(u))))}{F_{\nu_n}^{-1}(\Psi_{n-}^{-1}(\Psi_{n+}(u)))-F_{\nu_n}^{-1}(\Psi_{n+}^{-1}(\Psi_{n+}(u)))}\,d\Psi_{n+}(u)\\
	&=\int_{(0,\Psi_{n+}(1))}\frac{h(F_{\mu_n}^{-1}(\Psi_{n+}^{-1}(u)),F_{\nu_n}^{-1}(\Psi_{n-}^{-1}(u)))-h(F_{\mu_n}^{-1}(\Psi_{n+}^{-1}(u)),F_{\nu_n}^{-1}(\Psi_{n+}^{-1}(u)))}{F_{\nu_n}^{-1}(\Psi_{n-}^{-1}(u))-F_{\nu_n}^{-1}(\Psi_{n+}^{-1}(u))}\,du\\
	&=\Psi_{n+}(1)\int_{(0,1)}\frac{h(F_{\mu_n}^{-1}(\Psi_{n+}^{-1}(\Psi_{n+}(1)v)),F_{\nu_n}^{-1}(\Psi_{n-}^{-1}(\Psi_{n-}(1)v)))-h(F_{\mu_n}^{-1}(\Psi_{n+}^{-1}(\Psi_{n+}(1)v)),F_{\nu_n}^{-1}(\Psi_{n+}^{-1}(\Psi_{n+}(1)v)))}{(F_{\nu_n}^{-1}(\Psi_{n-}^{-1}(\Psi_{n-}(1)v))-F_{\nu_n}^{-1}(\Psi_{n+}^{-1}(\Psi_{n+}(1)v)))}\,dv.
	\end{align*}
	
	Since $h$ is Lipschitz continuous with respect to its second variable, then the integrand above is bounded. Since for all $u\in[0,1]$, $x\mapsto x^+$ is Lipschitz continuous with constant $1$,
	\begin{align*}
	\vert\Psi_{n+}(u)-\Psi_+(u)\vert&\le\int_0^u\vert(F_{\mu_n}^{-1}-F_{\nu_n}^{-1})^+(v)-(F_\mu^{-1}-F_\nu^{-1})^+(v)\vert\,dv\\
	&\le\int_0^u\vert F_{\mu_n}^{-1}(v)-F_\mu^{-1}(v)\vert\,dv+\int_0^u\vert F_{\nu_n}^{-1}(v)-F_\nu^{-1}(v)\vert\,dv\\
	&\le W_1(\mu_n,\mu)+W_1(\nu_n,\nu),
	\end{align*} 
	so $\Psi_{n+}$ converges uniformly to $\Psi_+$ on $[0,1]$. Moreover, for all $x\in\R$, $\vert\Psi_{n+}(F_{\mu_n}(x))-\Psi_+(F_\mu(x))\vert\le\sup_{[0,1]}\vert\Psi_{n+}-\Psi_+\vert+\vert\Psi_+(F_{\mu_n}(x))-\Psi_+(F_\mu(x))\vert$, so $\Psi_{n+}(F_{\mu_n}(x))/\Psi_{n+}(1)\underset{n\to+\infty}{\to}\Psi_+(F_\mu(x))/\Psi_+(1)$ for all $x\in\R$ outside the at most countable set of discontinuities of $F_\mu$. This implies that $d(\Psi_{n+}(F_{\mu_n}(x))/\Psi_{n+}(1))$ converges to $d(\Psi_+(F_\mu(x))/\Psi_+(1))$ for the weak convergence topology. We deduce the pointwise convergence of the left continuous pseudo-inverses $du$-almost everywhere on $(0,1)$, that is $F_{\mu_n}^{-1}(\Psi_{n+}^{-1}(\Psi_{n+}(1)u))\underset{n\to+\infty}{\longrightarrow}F_\mu^{-1}(\Psi_+^{-1}(\Psi_+(1)u))$ for $du$-almost all $u\in(0,1)$. In the same way, $F_{\nu_n}^{-1}(\Psi_{n+}^{-1}(\Psi_{n+}(1)u))\underset{n\to+\infty}{\longrightarrow}F_\nu^{-1}(\Psi_+^{-1}(\Psi_+(1)u))$ and $F_{\nu_n}^{-1}(\Psi_{n-}^{-1}(\Psi_{n-}(1)u))\underset{n\to+\infty}{\longrightarrow}F_\nu^{-1}(\Psi_-^{-1}(\Psi_-(1)u))$ for $du$-almost all $u\in(0,1)$. Therefore, by the dominated convergence theorem,
	\begin{align*}
	&\int_{(0,1)}\frac{(F_{\mu_n}^{-1}(u)-F_{\nu_n}^{-1}(u))^+}{F_{\nu_n}^{-1}(\varphi_n(u))-F_{\nu_n}^{-1}(u)}(h(F_{\mu_n}^{-1}(u),F_{\nu_n}^{-1}(\varphi_n(u)))-h(F_{\mu_n}^{-1}(u),F_{\nu_n}^{-1}(u)))\,du\\
	&\underset{n\to+\infty}{\longrightarrow}\Psi_{+}(1)\int_{(0,1)}\frac{h(F_{\mu}^{-1}(\Psi_{+}^{-1}(\Psi_{+}(1)v)),F_{\nu}^{-1}(\Psi_{-}^{-1}(\Psi_{-}(1)v)))-h(F_{\mu}^{-1}(\Psi_{+}^{-1}(\Psi_{+}(1)v)),F_{\nu}^{-1}(\Psi_{+}^{-1}(\Psi_{+}(1)v)))}{(F_{\nu}^{-1}(\Psi_{-}^{-1}(\Psi_{-}(1)v))-F_{\nu}^{-1}(\Psi_{+}^{-1}(\Psi_{+}(1)v)))}\,dv\\
	&=\int_{(0,1)}\frac{(F_{\mu}^{-1}(u)-F_{\nu}^{-1}(u))^+}{F_{\nu}^{-1}(\varphi(u))-F_{\nu}^{-1}(u)}(h(F_{\mu}^{-1}(u),F_{\nu}^{-1}(\varphi(u)))-h(F_{\mu}^{-1}(u),F_{\nu}^{-1}(u)))\,du.
	\end{align*}
	
	We can show in the same way that
	\begin{align*}
	&\int_{(0,1)}\frac{(F_{\mu_n}^{-1}(u)-F_{\nu_n}^{-1}(u))^-}{F_{\nu_n}^{-1}(u)-F_{\nu_n}^{-1}(\widetilde\varphi_n(u))}(h(F_{\mu_n}^{-1}(u),F_{\nu_n}^{-1}(\widetilde\varphi_n(u)))-h(F_{\mu_n}^{-1}(u),F_{\nu_n}^{-1}(u)))\,du\\
	&\underset{n\to+\infty}{\longrightarrow}\int_{(0,1)}\frac{(F_{\mu}^{-1}(u)-F_{\nu}^{-1}(u))^-}{F_{\nu}^{-1}(u)-F_{\nu}^{-1}(\widetilde\varphi(u))}(h(F_{\mu}^{-1}(u),F_{\nu}^{-1}(\widetilde\varphi(u)))-h(F_{\mu}^{-1}(u),F_{\nu}^{-1}(u)))\,du.
	\end{align*}
	
	Finally, we showed that
	\[
	\int_{\R\times\R}h(x,y)\,\mu_n(dx)\,m^{IT}_n(x,dy)\underset{n\to+\infty}{\longrightarrow}\int_{\R\times\R}h(x,y)\,\mu(dx)\,m^{IT}(x,dy),
	\]
	for any bounded and continuous function $h:\R^2\to\R$ which is Lipschitz continuous with respect to its second variable, that is $M^{IT}_n\underset{n\to+\infty}{\longrightarrow}M^{IT}$ for the weak convergence topology. Since the convergence for the Wasserstein distance topology is equivalent to the convergence for the weak convergence topology and the convergence of the first order moments (see for instance Theorem 6.9 Chapter 6 \cite{VillaniOT}), $\int_\R\vert x\vert\,\mu_n(dx)\underset{n\to+\infty}{\longrightarrow}\int_\R\vert x\vert\,\mu(dx)$ and $\int_\R\vert y\vert\,\nu_n(dy)\underset{n\to+\infty}{\longrightarrow}\int_\R\vert y\vert\,\nu(dy)$.
	
	Therefore, $W_1(M^{IT}_n,M^{IT})\underset{n\to+\infty}{\longrightarrow}0$ when $\R^2$ is endowed with the $L^1$-norm. Since all norms on $\R^2$ are equivalent, $W_1(M^{IT}_n,M^{IT})\underset{n\to+\infty}{\longrightarrow}0$ when $\R^2$ is endowed with any norm.
\end{proof}

\section{On the uniqueness of martingale couplings parametrised by \texorpdfstring{$\mathcal Q$}{Q}}
\label{subsec:InfiniteAmountMartingaleCouplings}

Let $\mu,\nu\in\mathcal P_1(\R)$ be such that $\mu<_{cx}\nu$. A direct consequence of Proposition \ref{couplagesQconvexes} is that the set of martingale couplings between $\mu$ and $\nu$ parametrised by $\mathcal Q$ is either a singleton, or uncountably infinite. Since $\mathcal Q$ is convex, we deduce from Proposition \ref{QZeta} below that $\mathcal Q$ is infinite as soon as $\mu<_{cx}\nu$. When $\mu$ and $\nu$ are such that $F_\mu$ and $F_\nu$ are continuous, Corollary \ref{corInfiniteCouplagesMartingalesQ2} below ensures that there exist uncountably many martingale couplings between $\mu$ and $\nu$ parametrised by $\mathcal Q$. However this does not necessarily hold in the general case. We saw that when $\nu$ is reduced to two atoms only, there exists a unique martingale coupling between $\mu$ and $\nu$. Suppose now that the comonotonous coupling is a martingale coupling between $\mu$ and $\nu$, and $\mu,\nu\in\mathcal P_2(\R)$. For any martingale coupling $M\in\Pi^{\mathrm M}(\mu,\nu)$, we have $\int_{\R\times\R}\vert x-y\vert^2\,M(dx,dy)=\int_\R y^2\,\nu(dy)-\int_\R x^2\,\mu(dx)$. So all the martingale couplings between $\mu$ and $\nu$ yield the same quadratic cost. In particular, they yield the same quadratic cost as the comonotonous coupling, which is the only minimiser of the quadratic cost among $\Pi(\mu,\nu)$. So the comonotonous coupling is the only martingale coupling between $\mu$ and $\nu$. The next proposition states that this conclusion still holds when $\mu$ and $\nu$ only have finite first order moments.

\begin{prop2}\label{CouplageComonotoneMartImpliqueUniqueMart} Let $\mu,\nu\in\mathcal P_1(\R)$ be such that $\mu<_{cx}\nu$. If the comonotonous coupling between $\mu$ and $\nu$ is a martingale coupling, that is for $U$ a random variable uniformly distributed on $(0,1)$,
	\[
	\E[F_\nu^{-1}(U)\vert F_\mu^{-1}(U)]=F_\mu^{-1}(U)\quad\textrm{almost surely},
	\]
	then it is the only martingale coupling between $\mu$ and $\nu$.
\end{prop2}
\begin{proof} Let $U$ be a random variable uniformly distributed on $(0,1)$. The couple $(U,F_\nu^{-1}(U))$ is distributed according to $\1_{(0,1)}(u)\,du\,\delta_{F_\nu^{-1}(u)}(dy)$. By Lemma \ref{lienPetitmGrandM} applied with the Markov kernel $(\widetilde m(u,dy))_{u\in(0,1)}=(\delta_{F_\nu^{-1}(u)}(dy))_{u\in(0,1)}$, we get that $(F_\mu^{-1}(U),F_\nu^{-1}(U))$ is distributed according to $\mu(dx)\,m(x,dy)$ where $(m(x,dy))_{x\in\R}$ is given by \eqref{defGrandM}. By Lemma \ref{FcomprisEntre0et1} combined with the inverse transform sampling and \eqref{defAlternativeGrandM}, we get that $(F_\mu^{-1}(U),F_\nu^{-1}(U))$ is distributed according to $\mu(dx)\,\int_{v=0}^1\delta_{F_\nu^{-1}(F_\mu(x_-)+v(F_\mu(x)-F_\mu(x_-)))}(dy)\,dv$. So almost surely,
	\begin{align*}
	F_\mu^{-1}(U)&=\E\left[F_\nu^{-1}(U)\vert F_\mu^{-1}(U)\right]\\
	&=\int_{v=0}^1\left(\int_{y\in\R}y\,\delta_{F_\nu^{-1}(F_\mu(F_\mu^{-1}(U)_-)+v(F_\mu(F_\mu^{-1}(U))-F_\mu(F_\mu^{-1}(U)_-)))}(dy)\right)\,dv\\
	&=\int_0^1F_\nu^{-1}(F_\mu(F_\mu^{-1}(U)_-)+v(F_\mu(F_\mu^{-1}(U))-F_\mu(F_\mu^{-1}(U)_-)))\,dv.
	\end{align*}
	
	By the inverse transform sampling, we deduce that for $\mu(dx)$-almost all $x\in\R$,
	\begin{equation}\label{contrainteMartingaleComonotonousCoupling}
	\int_0^1F_\nu^{-1}(F_\mu(x_-)+v(F_\mu(x)-F_\mu(x_-)))\,dv=x.
	\end{equation}
	
	Let $(\underline t_n,\overline t_n)_{1\le n\le N}$ denote the irreducible components of $(\mu,\nu)$, whose definition is given by \eqref{defComposantesIrreductibles}. We recall that the choice of any martingale coupling $M$ between $\mu$ and $\nu$ reduces to the choice of a sequence of martingale couplings $(M_n)_{1\le n\le N}$ such that for all $1\le n\le N$, $M_n$ is a martingale coupling between the probability measures $\mu_n$ and $\nu_n$ defined by \eqref{defNuNComposantesIrreductibles}. If for each $n$, $\mu_n$ reduces to a single atom, then we necessarily have $M_n(dx,dy)=\mu_n(dx)\,\nu_n(dy)$, so there is a unique choice of the sequence $(M_n)_{1\le n\le N}$ and therefore $M$, which is the comonotonous coupling. Let us then prove that $\mu_n$ reduces to a single atom.
	
	Let $I$ be the at most countable set of $x\in\R$ such that $\mu(\{x\})>0$ and $F_\nu^{-1}$ is nonconstant on $(F_\mu(x_-),F_\mu(x)]$. Let us show that
	\begin{equation}\label{composantesIrreductiblesComonotonousCoupling}
	\bigcup_{x\in I}(F_\mu(x_-),F_\mu(x))=\bigcup_{n=1}^N\left(F_\mu(\underline t_n),F_\mu((\overline t_n)_-)\right).
	\end{equation}
	
	By Lemma A.8 \cite{ApproxMOTJourdain2}, we have
	\begin{equation}\label{imageComposantesIrreductiblesFmu}
	\bigcup_{n=1}^N\left(F_\mu(\underline t_n),F_\mu((\overline t_n)_-)\right)=\left\{u\in(0,1)\mid\int_0^uF_\mu^{-1}(v)\,dv>\int_0^uF_\nu^{-1}(v)\,dv\right\}.
	\end{equation}
	
	Let $u\in(0,1)$. Suppose first that there exists $t\in\R$ such that $u=F_\mu(t)$. We recall that $(F_\mu^{-1}(U),F_\nu^{-1}(U))$ is distributed according to $\mu(dx)\,\int_{v=0}^1\delta_{F_\nu^{-1}(F_\mu(x_-)+v(F_\mu(x)-F_\mu(x_-)))}(dy)\,dv$. So
	\begin{align*}
	\int_0^{F_\mu(t)}F_\nu^{-1}(v)\,dv&=\int_0^1\1_{\{v\le F_\mu(t)\}}F_\nu^{-1}(v)\,dv=\int_0^1\1_{\{F_\mu^{-1}(v)\le t\}}F_\nu^{-1}(v)\,dv\\
	&=\int_{x\in\R}\left(\int_{v=0}^1\1_{\{x\le t\}}\left(\int_{y\in\R}y\,\delta_{F_\nu^{-1}(F_\mu(x_-)+v(F_\mu(x)-F_\mu(x_-)))}(dy)\right)\,dv\right)\,\mu(dx)\\
	&=\int_{x\in\R}\1_{\{x\le t\}}\left(\int_{v=0}^1F_\nu^{-1}(F_\mu(x_-)+v(F_\mu(x)-F_\mu(x_-)))\,dv\right)\,\mu(dx)\\
	&=\int_{x\in\R}\1_{\{x\le t\}}x\,\mu(dx)=\int_0^1\1_{\{F_\mu^{-1}(v)\le t\}}F_\mu^{-1}(v)\,dv=\int_0^{F_\mu(t)}F_\mu^{-1}(v)\,dv,
	\end{align*}
	where we used \eqref{contrainteMartingaleComonotonousCoupling} for the fifth equality and the inverse transform sampling for the sixth equality. By continuity, we also deduce that for all $t\in\R$, $\int_0^{F_\mu(t_-)}F_\nu^{-1}(v)\,dv=\int_0^{F_\mu(t_-)}F_\mu^{-1}(v)\,dv$.
	
	Suppose now that there exists $x\in\R$ in the set of discontinuities of $F_\mu$ such that $F_\mu(x_-)<u<F_\mu(x)$. According to \eqref{contrainteMartingaleComonotonousCoupling}, we have $\int_{F_\mu(x_-)}^{F_\mu(x)}F_\nu^{-1}(v)\,dv=\mu(\{x\})x=\int_{F_\mu(x_-)}^{F_\mu(x)}x\,dv=\int_{F_\mu(x_-)}^{F_\mu(x)}F_\mu^{-1}(v)\,dv$. 
	
	If $F_\nu^{-1}$ is constant on $(F_\mu(x_-),F_\mu(x)]$, then for all $v\in(F_\mu(x_-),F_\mu(x)]$, $F_\nu^{-1}(v)=x=F_\mu^{-1}(v)$, so
	\begin{align*}
	\int_0^uF_\nu^{-1}(v)\,dv&=\int_0^{F_\mu(x_-)}F_\nu^{-1}(v)\,dv+\int_{F_\mu(x_-)}^uF_\nu^{-1}(v)\,dv=\int_0^{F_\mu(x_-)}F_\mu^{-1}(v)\,dv+\int_{F_\mu(x_-)}^{u}F_\mu^{-1}(v)\,dv\\
	&=\int_0^{u}F_\mu^{-1}(v)\,dv.
	\end{align*}
	
	If $F_\nu^{-1}$ is nonconstant on $(F_\mu(x_-),F_\mu(x)]$, then using the monotonicity of $F_\nu^{-1}$, one can easily show that for all $u\in(F_\mu(x_-),F_\mu(x))$,
	\[
	\frac{1}{u-F_\mu(x_-)}\int_{F_\mu(x_-)}^{u}F_\nu^{-1}(v)\,dv<\frac{1}{F_\mu(x)-F_\mu(x_-)}\int_{F_\mu(x_-)}^{F_\mu(x)}F_\nu^{-1}(v)\,dv.
	\]
	
	We deduce that for all $u\in(F_\mu(x_-),F_\mu(x))$,
	\[
	\int_{F_\mu(x_-)}^{u}F_\nu^{-1}(v)\,dv<\frac{u-F_\mu(x_-)}{F_\mu(x)-F_\mu(x_-)}x\mu(\{x\})=(u-F_\mu(x_-))x=\int_{F_\mu(x_-)}^{u}x\,dv=\int_{F_\mu(x_-)}^{u}F_\mu^{-1}(v)\,dv,
	\]
	and $\int_0^uF_\mu^{-1}(v)\,dv>\int_0^uF_\nu^{-1}(v)\,dv$. With \eqref{imageComposantesIrreductiblesFmu}, we deduce \eqref{composantesIrreductiblesComonotonousCoupling}. Since the intervals $((\underline t_n,\overline t_n))_{1\le n\le N}$ are disjoint, the intervals $((F_\mu(\underline t_n),F_\mu((\overline t_n)_-)))_{1\le n\le N}$ are disjoint as well. By equality of unions of disjoint intervals, we proved that for all $1\le n\le N$, there exists $x\in I$ such that $(F_\mu(\underline t_n),F_\mu((\overline t_n)_-))=(F_\mu(x_-),F_\mu(x))$. So $x\in(\underline t_n,\overline t_n)$ and $\mu((\underline t_n,\overline t_n))=F_\mu((\overline t_n)_-)-F_\mu(\underline t_n)=F_\mu(x)-F_\mu(x_-)=\mu(\{x\})$. So $\mu_n=\delta_x$, and the discussion above concludes that there exists only one martingale coupling between $\mu$ and $\nu$, namely the comonotonous coupling.	
\end{proof}
 We saw in Section \ref{sousSec:DefITMC} that we can build a nonincreasing twin of the inverse transform martingale coupling (see \eqref{QnonincreasingITMC}) as soon as the two marginals satisfy the assumption in Proposition \ref{QNITappartientAQ}. This corresponds to a general inversion of the monoticity of $\varphi$ on $(0,1)$. In the general case, such an inversion is not possible on $(0,1)$, but can be made locally.

Let $\mu,\nu\in\mathcal P_1(\R)$ be such that $\mu<_{cx}\nu$. Since $\mu\neq\nu$, there exists $u\in(0,1)$ such that $\Psi_+(u)>\Psi_-(u)$. Let $v=\Psi_+^{-1}(\Psi_+(u))$. Then $\Psi_+(v)=\Psi_+(\Psi_+^{-1}(\Psi_+(u))=\Psi_+(u)$ so that $v>0$ and $\Psi_+(v)>\Psi_-(u)\ge\Psi_-(v)$. By left-continuity of $\Psi_-$ and $\Psi_+$, there exists $\eta\in(0,v)$ such that $\Psi_+(w)>\Psi_-(w)$ for all $w\in[v-\eta,v]$. By definition of $v$, we have $\Psi_+(v-\eta)<\Psi_+(v)$, so there exists $u_0\in(v-\eta,v)$ such that $(F_\mu^{-1}-F_\nu^{-1})^+(u_0)>0$. Since $u_0\in(v-\eta,v)$, we have $\Psi_+(v)>\Psi_+(u_0)>\Psi_-(u_0)$ so $1>\varphi(u_0)>u_0$ according to \eqref{phiPlusPetitQueU}. By left-continuity of $F_\mu^{-1}$, $F_\nu^{-1}$ and $\varphi$, there exists $\varepsilon\in(0,u_0)$ such that 
\begin{equation}\label{defU0Epsilon0}
\forall u\in[u_0-\varepsilon,u_0],\quad1>\varphi(u)>u_0\quad\textrm{and}\quad F_\mu^{-1}(u)>F_\nu^{-1}(u).
\end{equation}

Since $(u_0-\varepsilon,u_0]\subset\mathcal U_+$, $\Psi_+$ is increasing and is therefore one-to-one onto from $(u_0-\varepsilon,u_0]$ to $(\Psi_+(u_0-\varepsilon),\Psi_+(u_0)]$. Since the set of discontinuities of $\Psi_-^{-1}$ is at most countable, up to choosing $\varepsilon$ smaller, we may also suppose that in addition to \eqref{defU0Epsilon0}, $\varepsilon$ is such that $\Psi_-^{-1}$ is continuous in $\Psi_+(u_0-\varepsilon)$.
Let then $\zeta:[0,1]\to[0,1]$ and $\widetilde\zeta:[0,1]\to[0,1]$ be defined for all $u\in(0,1)$ by
\begin{equation}\label{defZeta}
\zeta(u)=\Psi_-^{-1}(G(\Psi_+(u)))\quad\textrm{and}\quad\widetilde\zeta(u)=\Psi_+^{-1}(G(\Psi_-(u))),
\end{equation}
where $G:u\mapsto u\1_{(u_0-\varepsilon,u_0]^\complement}(\Psi_+^{-1}(u))+\left(\Psi_+(u_0)-u+\Psi_+(u_0-\varepsilon)\right)\1_{(u_0-\varepsilon,u_0]}(\Psi_+^{-1}(u))$.
Let $Q^\zeta$ be the measure defined on $(0,1)^2$ by
\begin{equation}\label{QITMCinversionLocale}
Q^\zeta(du,dv)=\frac1\gamma\left(F_\mu^{-1}-F_\nu^{-1}\right)^+(u)\,du\,\pi^\zeta_+(u,dv)\quad\textrm{where $\pi^\zeta_+(u,dv)=\1_{\{0<\zeta(u)<1\}}\,\delta_{\zeta(u)}(dv)$},
\end{equation}
with $\gamma=\Psi_-(1)=\Psi_+(1)$.

\begin{prop2}\label{QZeta} Let $\mu,\nu\in\mathcal P_1(\R)$ be such that $\mu<_{cx}\nu$. The measure $Q^\zeta$ defined by \eqref{QITMCinversionLocale} is an element of $\mathcal Q$. Moreover,
	\[
	Q^\zeta(du,dv)=\frac{1}{\gamma}(F_\mu^{-1}-F_\nu^{-1})^-(v)\,dv\,\pi^\zeta_-(v,du)\quad\textrm{where $\pi^\zeta_-(v,du)=\1_{\{0<\widetildemid{\zeta}(v)<1\}}\,\delta_{\widetildemid\zeta(v)}(du)$}.
	\]
\end{prop2}

As said above, $\Psi_+$ is one-to-one onto from $(u_0-\varepsilon,u_0]$ to $(\Psi_+(u_0-\varepsilon),\Psi_+(u_0)]$. So, for all $u\in(u_0-\varepsilon,u_0]$, $\Psi_+^{-1}(\Psi_+(u))=u$ and $G(\Psi_+(u))=\Psi_+(u_0)-\Psi_+(u)+\Psi_+(u_0-\varepsilon)$. So
\begin{equation}\label{defZetaAvantU0}
\forall u\in(u_0-\varepsilon,u_0],\quad\zeta(u)=\Psi_-^{-1}(\Psi_+(u_0)-\Psi_+(u)+\Psi_+(u_0-\varepsilon)).
\end{equation}

Since $\Psi_-$ is continuous, $\Psi_-^{-1}$ is one-to-one. Moreover, $\Psi_+$ is increasing on $(u_0-\varepsilon,u_0]$, so for all $u\in(u_0-\varepsilon,u_0]\backslash\{\Psi_+^{-1}(\frac{\Psi_+(u_0)+\Psi_+(u_0-\varepsilon)}{2})\}$, $\zeta(u)\neq\varphi(u)$. Since $(u_0-\varepsilon,u_0]\subset\mathcal U_+$, considering the first marginal of $Q^\zeta$ and $Q^{IT}$, we deduce that $Q^\zeta\neq Q^{IT}$. As a direct consequence of the convexity of $\mathcal Q$, we deduce that $\mathcal Q$ is uncountably infinite.

\begin{cor2} Let $\mu,\nu\in\mathcal P_1(\R)$ be such that $\mu<_{cx}\nu$. Then $\mathcal Q$ is uncountably infinite.
\end{cor2}

\begin{proof}[Proof of Proposition \ref{QZeta}] Let $h:(0,1)^2\to\R$ be a measurable and bounded function. We have
	\begin{align}\label{demoQITMCinversionLocale2}
	\begin{split}
	\int_{(0,1)^2}h(u,v)\,Q^\zeta(du,dv)&=\frac1\gamma\int_{(0,1)^2}h(u,v)(F_\mu^{-1}-F_\nu^{-1})^+(u)\1_{\{0<\zeta(u)<1\}}\,\delta_{\zeta(u)}(dv)\,du\\
	&=\frac1\gamma\int_0^1h(u,\zeta(u))\1_{\{0<\zeta(u)<1\}}\,d\Psi_+(u)\\
	&=\frac1\gamma\int_0^1h(\Psi_+^{-1}(\Psi_+(u)),\zeta(u))\1_{\{0<\zeta(u)<1\}}\1_{\{0<u<1\}}\,d\Psi_+(u),
	\end{split}
	\end{align}
	where the last equality is a consequence of Lemma \ref{FmoinsUnrondF}. By Proposition \ref{prop410RevuzYor},
	\[
	\int_{(0,1)^2}h(u,v)\,Q^\zeta(du,dv)=\frac1\gamma\int_0^{\Psi_+(1)}h(\Psi_+^{-1}(u),\Psi_-^{-1}(G(u)))\1_{\{0<\Psi_-^{-1}(G(u))<1\}}\1_{\{0<\Psi_+^{-1}(u)<1\}}\,du.
	\]
	
	By Lemma \ref{FmoinsUnrondF}, for all $u\in(0,\Psi_+(1))$, $u_0-\varepsilon<\Psi_+^{-1}(u)\le u_0\iff\Psi_+(u_0-\varepsilon)<u\le\Psi_+(u_0)$. Hence $G$ is a piecewise affine function which satisfies $G(G(u))=u$ for all $u\in(0,\Psi_+(1))\backslash\{\Psi_+(u_0)\}$ and $G(G(\Psi_+(u_0)))=\Psi_+(u_0-\varepsilon)$. So by the change of variables $w=G(u)$, we have
	\begin{equation}\label{demoQITMCinversionLocale1}
	\int_{(0,1)^2}h(u,v)\,Q^\zeta(du,dv)=\frac1\gamma\int_0^{\Psi_+(1)}h(\Psi_+^{-1}(G(w)),\Psi_-^{-1}(w))\1_{\{0<\Psi_-^{-1}(w)<1\}}\1_{\{0<\Psi_+^{-1}(G(w))<1\}}\,dw.
	\end{equation}
	
	By continuity of $\Psi_-$ and Proposition \ref{prop410RevuzYor}, using that $\Psi_+(1)=\Psi_-(1)$, we have
	\begin{align*}
	\int_{(0,1)^2}h(u,v)\,Q^\zeta(du,dv)&=\frac1\gamma\int_0^1h(\Psi_+^{-1}(G(\Psi_-(u))),\Psi_-^{-1}(\Psi_-(u)))\1_{\{0<\Psi_-^{-1}(\Psi_-(u))<1\}}\1_{\{0<\Psi_+^{-1}(G(\Psi_-(u)))<1\}}\,d\Psi_-(u)\\
	&=\frac1\gamma\int_0^1h(\widetilde\zeta(u),u)\1_{\{0<\widetildemid\zeta(u)<1\}}\,d\Psi_-(u),
	\end{align*}
	where we used for the last equality that $\Psi_-^{-1}(\Psi_-(u))=u$, $d\Psi_-(u)$-almost everywhere on $(0,1)$ according to Lemma \ref{FmoinsUnrondF}.
	
	Hence $Q(du,dv)=\frac{1}{\gamma}(F_\mu^{-1}-F_\nu^{-1})^-(v)\,dv\,\pi^\zeta_-(v,du)$ where $\pi^\zeta_-(v,du)=\1_{\{0<\widetildemid{\zeta}(v)<1\}}\,\delta_{\widetildemid\zeta(v)}(du)$.
	
	For $h:(u,v)\mapsto1$, \eqref{demoQITMCinversionLocale1} writes
	\[
	Q^\zeta((0,1)^2)=\frac1\gamma\int_0^{\Psi_+(1)}\1_{\{0<\Psi_-^{-1}(w)<1\}}\1_{\{0<\Psi_+^{-1}(G(w))<1\}}\,dw.
	\]
	
	By continuity of $\Psi_-$, Proposition \ref{prop410RevuzYor} and Lemma \ref{FmoinsUnrondF}, $\int_0^{\Psi_+(1)}\1_{\{0<\Psi_-^{-1}(w)<1\}}\,dw=\int_0^1\1_{\{0<\Psi_-^{-1}(\Psi_-(w))<1\}}\,d\Psi_-(w)=\int_0^1d\,\Psi_-(u)=\Psi_-(1)=\Psi_+(1)$. So $0<\Psi_-^{-1}(w)<1$, $dw$-almost everywhere on $(0,\Psi_+(1))$. By a similar reasoning, $0<\Psi_+^{-1}(w)<1$ for $dw$-almost all $w\in(0,\Psi_+(1))$. Since $G$ is piecewise affine and bijective from $(0,\Psi_+(1))\backslash\{\Psi_+(u_0)\}$ to itself, $0<\Psi_+^{-1}(G(w))<1$ for $dw$-almost all $w\in(0,\Psi_+(1))$. Hence
	\[
	Q^\zeta((0,1)^2)=\frac1\gamma\int_0^{\Psi_+(1)}\,dw=1,
	\]
	so $Q^\zeta$ is a probability measure, with first marginal $\frac{1}{\gamma}(F_\mu^{-1}-F_\nu^{-1})^+(u)\,du$ and second marginal $\frac{1}{\gamma}(F_\mu^{-1}-F_\nu^{-1})^-(v)\,dv$.
	
	We have
	\begin{align*}
	Q^\zeta\left(\{(u,v)\in(0,1)^2\mid u<v\}\right)&=\frac{1}{\gamma}\int_0^1\1_{\{u<\zeta(u)\}}\1_{\{0<\zeta(u)<1\}}\,d\Psi_+(u).
	\end{align*}
	
	According to \eqref{phiPlusGrandQueU}, $u<\varphi(u)$, $d\Psi_+(u)$-almost everywhere on $(0,1)$. According to \eqref{defZetaAvantU0} and \eqref{defU0Epsilon0}, for all $u\in(u_0-\varepsilon,u_0]$, $\zeta(u)\ge\zeta(u_0)$ and
	\begin{equation}\label{zetaDeU0}
	\zeta(u_0)=\varphi(u_0-\varepsilon)>u_0.
	\end{equation}
	
	So for all $u\in(u_0-\varepsilon,u_0]$, $\zeta(u)>u_0\ge u$. Moreover, by Lemma \ref{FmoinsUnrondF}, $\Psi_+^{-1}(\Psi_+(u))=u$, $d\Psi_+(u)$-almost everywhere on $(0,1)$. So $\zeta$ coincides with $\varphi$, $d\Psi_+$-almost everywhere on $(u_0-\varepsilon,u_0]^\complement$, hence $u<\zeta(u)$, $d\Psi_+(u)$-almost everywhere on $(0,1)$. So using \eqref{demoQITMCinversionLocale2} for $h=1$, we get that
	\begin{align*}
	Q^\zeta\left(\{(u,v)\in(0,1)^2\mid u<v\}\right)&=\frac{1}{\gamma}\int_0^1\1_{\{0<\zeta(u)<1\}}\,d\Psi_+(u)=Q^\zeta((0,1)^2),
	\end{align*}
	which is equal to $1$ since $Q^\zeta$ is a probability measure on $(0,1)^2$.
\end{proof}

\begin{cor2}\label{corInfiniteCouplagesMartingalesQ} Let $\mu,\nu\in\mathcal P_1(\R)$ be such that $\mu<_{cx}\nu$. Let $u_0\in(0,1)$ and $\varepsilon\in(0,u_0)$ be such that \eqref{defU0Epsilon0} is satisfied and $\Psi_-^{-1}$ is continuous in $\Psi_+(u_0-\varepsilon)$. If $F_\nu^{-1}$ is nonconstant on $(\varphi(u_0-\varepsilon),\varphi(u_0)]$ and if $F_\mu^{-1}$ is such that for all $\varepsilon'\in(0,\varepsilon)$, the set $\{u\in(0,1)\mid F_\mu^{-1}(u_0-\varepsilon')<F_\mu^{-1}(u)<F_\mu^{-1}(u_0)\}$ has positive Lebesgue measure, then there exist uncountably many martingale couplings parametrised by $\mathcal Q$ between $\mu$ and $\nu$.
\end{cor2}

Notice that by left-continuity, the condition on $F_\mu^{-1}$ in the statement of Corollary \ref{corInfiniteCouplagesMartingalesQ} is satisfied if for all $\varepsilon'\in(0,\varepsilon)$, $F_\mu^{-1}$ takes at least three different values on $[u_0-\varepsilon',u_0]$. A direct consequence of Corollary \ref{corInfiniteCouplagesMartingalesQ} is the infinite amount of martingale couplings between $\mu$ and $\nu$ when $F_\mu^{-1}$ and $F_\nu^{-1}$ are increasing, or equivalently when $F_\mu$ and $F_\nu$ are continuous.

\begin{cor2}\label{corInfiniteCouplagesMartingalesQ2} Let $\mu,\nu\in\mathcal P_1(\R)$ be such that $\mu<_{cx}\nu$ and $\mu(\{x\})=\nu(\{x\})=0$ for all $x\in\R$. Then there exist uncountably many martingale couplings parametrised by $\mathcal Q$ between $\mu$ and $\nu$.
\end{cor2}

\begin{proof}[Proof of Corollary \ref{corInfiniteCouplagesMartingalesQ}] Let $\zeta$ be defined by \eqref{defZeta} and let $Q^\zeta$ be the probability measure defined by \eqref{QITMCinversionLocale}. By Proposition \ref{ITMCcasParticulier}, Proposition \ref{QZeta} and Lemma \ref{Petitmdupp}, for $du$-almost all $u\in(u_0-\varepsilon,u_0)$, $F_\nu^{-1}(\zeta(u))>F_\mu^{-1}(u)>F_\nu^{-1}(u)$ and $F_\nu^{-1}(\varphi(u))>F_\mu^{-1}(u)>F_\nu^{-1}(u)$.
	
	Since $F_\nu^{-1}$ is left-continuous and nonconstant on $(\varphi(u_0-\varepsilon),\varphi(u_0)]$, $F_\nu^{-1}$ is nonconstant on $(\varphi(u_0-\varepsilon),\varphi(u_0))$. So there exist $a,b\in(\varphi(u_0-\varepsilon),\varphi(u_0))$ such that $F_\nu^{-1}(a)<F_\nu^{-1}(b)$. Let then $c=\inf\{u\in(a,b)\mid F_\nu^{-1}(u)=F_\nu^{-1}(b)\}$. Let $u\in[a,b]$. If $F_\nu^{-1}(u)=F_\nu^{-1}(b)$, then $c\le u$. Else if $F_\nu^{-1}(u)<F_\nu^{-1}(b)$, then $c\ge u$. We deduce that $a\le c\le b$ and for all $u,v\in(\varphi(u_0-\varepsilon),\varphi(u_0))$ such that $u<c<v$, we have $F_\nu^{-1}(u)<F_\nu^{-1}(b)\le F_\nu^{-1}(v)$.
		
	Using \eqref{zetaDeU0}, we have $F_\nu^{-1}(\varphi(u_0))>F_\nu^{-1}(\varphi(u_0-\varepsilon))=F_\nu^{-1}(\zeta(u_0))$. The map $\varphi$ is left-continuous, and since $\varepsilon$ is such that $\Psi_-^{-1}$ is continuous in $\Psi_+(u_0-\varepsilon)$, $\zeta$ is left-continuous in $u_0$. So there exists $\tau\in(0,\varepsilon)$ such that for all $u\in[u_0-\tau,u_0]$, $\zeta(u)<c<\varphi(u)$. We deduce that for all $u\in[u_0-\tau,u_0]$, $F_\nu^{-1}(\varphi(u))>F_\nu^{-1}(\zeta(u))$. So for $du$-almost all $u\in[u_0-\tau,u_0]$, we have
	\begin{equation}\label{InfiniteCouplagesComparaison}
	F_\nu^{-1}(\varphi(u))>F_\nu^{-1}(\zeta(u))>F_\mu^{-1}(u)>F_\nu^{-1}(u).
	\end{equation}
	
	Let $a,b,c,d\in\R$ be such that $a>b>c>d$. Then
	\begin{align*}
	\left(\frac{c-d}{a-d}\right)a^2+\left(\frac{a-c}{a-d}\right)d^2&=\frac{ca^2-da^2+ad^2-cd^2}{a-d}=\frac{(a-d)(ac-ad+dc)}{(a-d)}=a(c-d)+dc\\
	&>b(c-d)+dc\\
	&=\left(\frac{c-d}{b-d}\right)b^2+\left(\frac{b-c}{b-d}\right)d^2.
	\end{align*}
	
	Thanks to \eqref{InfiniteCouplagesComparaison} and this inequality applied with $(a,b,c,d)=(F_\nu^{-1}(\varphi(u)),F_\nu^{-1}(\zeta(u)),F_\mu^{-1}(u),F_\nu^{-1}(u))$, we deduce that for $du$-almost all $u\in[u_0-\tau,u_0]$,
	\begin{equation}\label{comparaisonMITMQzeta}
	\int_{\R}y^2\,\widetilde m^{IT}(u,dy)>\int_{\R}y^2\,\widetilde m^{Q^\zeta}(u,dy).
	\end{equation}
	
	By Lemma \ref{lienPetitmGrandM}, we have
	\[
	\int_{\R^2}\1_{\{F_\mu^{-1}(u_0-\tau)<x<F_\mu^{-1}(u_0)\}}y^2\,M^{IT}(dx,dy)=\int_{u=0}^1\1_{\{F_\mu^{-1}(u_0-\tau)<F_\mu^{-1}(u)<F_\mu^{-1}(u_0)\}}\int_{y\in\R}y^2\,\widetilde m^{IT}(u,dy)\,du.
	\]
	
	For all $u\in(0,1)$ such that $F_\mu^{-1}(u_0-\tau)<F_\mu^{-1}(u)<F_\mu^{-1}(u_0)$, we have $u\in[u_0-\tau,u_0]$. So by \eqref{InfiniteCouplagesComparaison}, for $Q\in\{Q^{IT},Q^\zeta\}$ and for $du$-almost all $u\in(0,1)$ such that $F_\mu^{-1}(u_0-\tau)<F_\mu^{-1}(u)<F_\mu^{-1}(u_0)$, $y^2\le\max(F_\nu^{-1}(\varphi(u_0))^2,F_\nu^{-1}(u_0-\varepsilon)^2)$, $\widetilde m^Q(u,dy)$-almost everywhere. Therefore, for $Q\in\{Q^{IT},Q^\zeta\}$, we have
	\[
	\int_{u=0}^1\1_{\{F_\mu^{-1}(u_0-\tau)<F_\mu^{-1}(u)<F_\mu^{-1}(u_0)\}}\int_{y\in\R}y^2\,\widetilde m^Q(u,dy)\,du\le\max(F_\nu^{-1}(\varphi(u_0))^2,F_\nu^{-1}(u_0-\varepsilon)^2)<\infty.
	\]
	
	Since by assumption the Lebesgue measure of $\{u\in(0,1)\mid F_\mu^{-1}(u_0-\tau)<F_\mu^{-1}(u)<F_\mu^{-1}(u_0)\}$ is positive, according to \eqref{comparaisonMITMQzeta}, we get that	
	\begin{align*}
	\int_{\R^2}\1_{\{F_\mu^{-1}(u_0-\tau)<x<F_\mu^{-1}(u_0)\}}y^2\,M^{IT}(dx,dy)&>\int_{u=0}^1\1_{\{F_\mu^{-1}(u_0-\tau)<F_\mu^{-1}(u)<F_\mu^{-1}(u_0)\}}\int_{y\in\R}y^2\,\widetilde m^{Q^{\zeta}}(u,dy)\,du\\
	&=\int_{\R^2}\1_{\{F_\mu^{-1}(u_0-\tau)<x<F_\mu^{-1}(u_0)\}}y^2\,M^{Q^{\zeta}}(dx,dy).
	\end{align*} 
		
		So $M^{IT}\neq M^{Q^\zeta}$. By Proposition \ref{couplagesQconvexes}, we deduce that $(M^{\lambda Q^{IT}+(1-\lambda)Q^\zeta})_{\lambda\in[0,1]}$ is a family of distinct martingale couplings between $\mu$ and $\nu$.
\end{proof}

\section{Corresponding super and submartingale couplings}
\label{sec:SubSuperMartingaleCase}

 We recall that two probability measures $\mu,\nu\in\mathcal P_1(\R)$ are in the decreasing (resp. increasing) convex order and denote $\mu\le_{dcx}\nu$ (resp. $\mu\le_{icx}\nu$) if $\int_{\R}f(x)\,\mu(dx)\le\int_{\R}f(x)\,\nu(dx)$ for any decreasing (resp. increasing) convex function $f:\R\to\R$. Let us begin with $\mu,\nu\in\mathcal P_1(\R)$ such that $\mu\le_{dcx}\nu$. We use the definitions of $\Psi_+$, $\Psi_-$, $\varphi$ and $\widetilde\varphi$ given at the beginning of Section \ref{sousSec:DefITMC}. According to Theorem 4.A.3 Chapter 4 \cite{ShakedShanthikumar}, $\mu,\nu\in\mathcal P_1(\R)$ are such that $\mu\le_{dcx}\nu$ iff for all $u\in[0,1]$, $\int_0^uF_\mu^{-1}(v)\,dv\ge\int_0^uF_\nu^{-1}(v)\,dv$. This implies that for all $u\in[0,1]$, $\Psi_+(u)\ge\Psi_-(u)$. Let then $u_d=\widetilde\varphi(1)=\Psi_+^{-1}(\Psi_-(1))$. Since $\Psi_+$ is continuous, $\Psi_+(u_d)=\Psi_-(1)$. If $u_d=0$, then $\Psi_-(1)=0$, which implies that for all $u\in(0,1)$, $F_\mu^{-1}(u)\ge F_\nu^{-1}(u)$. So, for $U$ a random variable uniformly distributed on $(0,1)$, by the inverse transform sampling, $(F_\mu^{-1}(U),F_\nu^{-1}(U))$ is a supermartingale coupling between $\mu$ and $\nu$, that is $\E[F_\nu^{-1}(U)\vert F_\mu^{-1}(U)]\le F_\mu^{-1}(U)$ almost surely. We suppose now $u_d>0$, which is equivalent to $\Psi_-(1)>0$. Let $\gamma_d=\int_0^{u_d}(F_\mu^{-1}-F_\nu^{-1})^+(u)\,du=\int_0^1(F_\mu^{-1}-F_\nu^{-1})^-(u)\,du\in(0,+\infty)$. We note $\mathcal Q^d$ the set of probability measures $Q^d$ on $(0,u_d)\times(0,1)$ such that
\begin{enumerate}[(i)]
	\item $Q^d$ has first marginal $\frac1{\gamma_d}(F_\mu^{-1}-F_\nu^{-1})^+(u)\,du$;
	\item $Q^d$ has second marginal $\frac1{\gamma_d}(F_\mu^{-1}-F_\nu^{-1})^-(v)\,dv$;
	\item $Q^d\left(\{(u,v)\in\mathcal (0,u_d)\times(0,1)\mid u<v\}\right)=1$.
\end{enumerate}

The existence of the inverse transform supermartingale coupling introduced below implies that $Q^d$ is non-empty. Let $Q^d$ be an element of $\mathcal Q^d$. Let $\pi^{Q^d}_+$ and $\pi^{Q^d}_-$ be two sub-Markov kernels such that
\[
Q^d(du,dv)=\frac{1}{\gamma_d}(F_\mu^{-1}(u)-F_\nu^{-1})^+(u)\,du\,\pi^{Q^d}_+(u,dv)=\frac{1}{\gamma_d}(F_\mu^{-1}-F_\nu^{-1})^-(v)\,dv\,\pi^{Q^d}_-(v,du).
\]

We recall the definition of $\mathcal U_+$, $\mathcal U_-$ and $\mathcal U_0$ given by \eqref{defUmoinsUplusUzero}. Let $(\widetilde m^{Q^d}(u,dy))_{u\in(0,1)}$ be the Markov kernel defined by

\begin{equation}
\left\{\begin{array}{r}
\displaystyle\int_{(0,1)}\frac{F_\mu^{-1}(u)-F_\nu^{-1}(u)}{F_\nu^{-1}(v)-F_\nu^{-1}(u)}\delta_{F_\nu^{-1}(v)}(dy)\pi^d_+(u,dv)+\int_{(0,1)}\frac{F_\nu^{-1}(v)-F_\mu^{-1}(u)}{F_\nu^{-1}(v)-F_\nu^{-1}(u)}\pi^d_+(u,dv)\delta_{F_\nu^{-1}(u)}(dy)\jump\\
\mathrm{for}\ u\in\mathcal U_+\cap(0,u_d)\ \textrm{such that}\ \pi^d_+(u,\{v\in(0,1)\mid F_\nu^{-1}(v)>F_\mu^{-1}(u)\})=1; \\
\\
\displaystyle\int_{(0,u_d)}\frac{F_\mu^{-1}(u)-F_\nu^{-1}(u)}{F_\nu^{-1}(v)-F_\nu^{-1}(u)}\delta_{F_\nu^{-1}(v)}(dy)\pi^d_-(u,dv)+\int_{(0,u_d)}\frac{F_\nu^{-1}(v)-F_\mu^{-1}(u)}{F_\nu^{-1}(v)-F_\nu^{-1}(u)}\pi^d_-(u,dv)\delta_{F_\nu^{-1}(u)}(dy)\jump\\
\mathrm{for}\ u\in\mathcal U_-\ \textrm{such that}\ \pi^d_-(u,\{v\in(0,1)\mid F_\nu^{-1}(v)<F_\mu^{-1}(u)\})=1;\\
\\
\delta_{F_\nu^{-1}(u)}(dy)\quad\quad\quad\quad \mathrm{otherwise.}\\
\end{array}
\right.
\label{defPetitmdcx}
\end{equation}

The idea of this construction is as follows: for $u\in\mathcal U_-$, we can associate to $F_\mu^{-1}(u)$ a martingale contribution with $F_\nu^{-1}(u)$ and $F_\nu^{-1}(v)$ as in Section \ref{MesuresQ}. If $F_\mu^{-1}(u)=F_\nu^{-1}(u)$, we associate $F_\nu^{-1}(u)$ to $F_\mu^{-1}(u)$. However, if $u\in\mathcal U_+$, the martingale contribution stops at $u_d$. For $u\in\mathcal U_+$ such that $u>u_d$, we only associate $F_\nu^{-1}(u)<F_\mu^{-1}(u)$ to $F_\mu^{-1}(u)$ since there is no partner $v\in\mathcal U_-\cap(u,1)$ available to construct a martingale contribution: all such possible partners have already been associated to values in $\mathcal U_+\cap(0,u_d)$. Let $(m^{Q^d}(x,dy))_{x\in\R}$ be the Markov kernel defined as in \eqref{defGrandM} with $(\widetilde m(u,dy))_{u\in(0,1)}$ replaced by $(\widetilde m^{Q^d}(u,dy))_{u\in(0,1)}$. Then $\mu(dx)\,m^{Q^d}(x,dy)$ is expected to be a supermartingale coupling between $\mu$ and $\nu$, as the next proposition states.

\begin{prop2}\label{couplageSupermartingale} Let $\mu,\nu\in\mathcal P_1(\R)$ be such that $\mu\le_{dcx}\nu$. Suppose that $u_d>0$. The probability measure $M^{Q^d}(dx,dy)=\mu(dx)\,m^{Q^d}(x,dy)$ is a supermartingale coupling between $\mu$ and $\nu$.
\end{prop2}

\begin{proof}[Proof] With the very same arguments as in Section \ref{MesuresQ}, we show that $M^{Q^d}(dx,dy)$ is a coupling between $\mu$ and $\nu$ (see Proposition \ref{muMcouplageMartingale}).
	
	For $du$-almost all $u\in\mathcal U_+\cap(0,u_d)$ and for $du$-almost all $u\in\mathcal U_-$, the same calculation as \eqref{calculMartingaleUmoins} yields $\int_\R y\,\widetilde m^{Q^d}(u,dy)=F_\mu^{-1}(u)$. For $u\in\mathcal U_0$, by definition of $\widetilde m^{Q^d}$, $\int_\R y\,\widetilde m^{Q^d}(u,dy)=F_\nu^{-1}(u)=F_\mu^{-1}(u)$. For $u\in\mathcal U_+\cap(u_d,1)$, by definition of $\widetilde m^{Q^d}$, $\int_\R y\,\widetilde m^{Q^d}(u,dy)=F_\nu^{-1}(u)<F_\mu^{-1}(u)$.
	
	In conclusion, for $du$-almost all $u\in(0,1)$, $\int_\R y\,\widetilde m^{Q^d}(u,dy)\le F_\mu^{-1}(u)$. Let $h:\R\to\R$ be a measurable nonnegative and bounded function. By Lemma \ref{lienPetitmGrandM},
	\begin{equation}\label{preuveMQDestCouplageSupermartingale}
	\int_{\R\times\R}h(x)(y-x)\,\mu(dx)\,m^{Q^d}(x,dy)=\int_0^1h(F_\mu^{-1}(u))\left(\int_\R(y-F_\mu^{-1}(u))\,\widetilde m^{Q^d}(u,dy)\right)\,du\le0.
	\end{equation}
	
	Therefore, for all $Q^d\in\mathcal Q^d$, $M^{Q^d}(dx,dy)$ is a supermartingale coupling between $\mu$ and $\nu$.
\end{proof}

Notice that $M^{Q^d}$ is a martingale coupling between $\mu$ and $\nu$ iff $\mathcal U_+\cap(u_d,1)=\emptyset$. Indeed, $M^{Q^d}$ is a martingale coupling iff \eqref{preuveMQDestCouplageSupermartingale} is an equality for any nonnegative and bounded function $h:\R\to\R$, which in view of the proof of Proposition \ref{couplageSupermartingale} and the left-continuity of $F_\mu^{-1}$ and $F_\nu^{-1}$ is equivalent to $\mathcal U_+\cap(u_d,1)=\emptyset$. We can also see that $M^{Q^d}$ is a martingale coupling between $\mu$ and $\nu$ iff $\mu$ and $\nu$ have equal means, that is $\Psi_+(1)=\Psi_-(1)$. By continuity of $\Psi_+$, we have $\Psi_+(u_d)=\Psi_-(1)$. So $\Psi_+(1)=\Psi_-(1)$ iff $\Psi_+(u_d)=\Psi_+(1)$, which by left-continuity of $F_\mu^{-1}$ and $F_\nu^{-1}$ is equivalent to $\mathcal U_+\cap(u_d,1)=\emptyset$.


Let $Q^{ITS}$ be the measure defined on $(0,u_d)\times(0,1)$ by
\begin{equation}
Q^{ITS}(du,dv)=\frac{1}{\gamma_d}(F_\mu^{-1}-F_\nu^{-1})^+(u)\,du\,\pi^{ITS}_+(u,dv)\quad\textrm{where $\pi^{ITS}_+(u,dv)=\1_{\{0<\varphi(u)<1\}}\,\delta_{\varphi(u)}(dv)$}.
\end{equation}

\begin{prop2}\label{ITMCsupermartingale} Let $\mu,\nu\in\mathcal P_1(\R)$ such that $\mu\le_{dcx}\nu$. Suppose that $u_d>0$. The measure $Q^{ITS}$ is an element of $\mathcal Q^d$. Moreover, 
	\[
	Q^{ITS}(du,dv)=\frac{1}{\gamma_d}(F_\mu^{-1}-F_\nu^{-1})^-(v)\,dv\,\pi^{ITS}_-(v,du)\quad\textrm{where $\pi^{ITS}_-(v,du)=\1_{\{0<\widetildelow{\varphi}(v)<u_d\}}\,\delta_{\widetildelow\varphi(v)}(du)$}.
	\]
\end{prop2}
 
	  For this particular $Q^{ITS}\in\mathcal Q^d$, Definition \eqref{defPetitmdcx} of the Markov kernel $(\widetilde m^{Q^{ITS}}(u,dy))_{u\in(0,1)}$ becomes
\begin{equation}
\left\{\begin{array}{r}
\displaystyle\left(1-\frac{F_\mu^{-1}(u)-F_\nu^{-1}(u)}{F_\nu^{-1}(\varphi(u))-F_\nu^{-1}(u)}\right)\delta_{F_\nu^{-1}(u)}(dy)+\frac{F_\mu^{-1}(u)-F_\nu^{-1}(u)}{F_\nu^{-1}(\varphi(u))-F_\nu^{-1}(u)}\delta_{F_\nu^{-1}(\varphi(u))}(dy)\jump\\
\mathrm{if}\ u<u_d,\ F_\nu^{-1}(\varphi(u))>F_\mu^{-1}(u)>F_\nu^{-1}(u)\text{ and }\varphi(u)<1;\\
\\
\displaystyle\left(1-\frac{F_\nu^{-1}(u)-F_\mu^{-1}(u)}{F_\nu^{-1}(u)-F_\nu^{-1}(\widetilde\varphi(u))}\right)\delta_{F_\nu^{-1}(u)}(dy)+\frac{F_\nu^{-1}(u)-F_\mu^{-1}(u)}{F_\nu^{-1}(u)-F_\nu^{-1}(\widetilde\varphi(u))}\delta_{F_\nu^{-1}(\widetildelow\varphi(u))}(dy)\jump\\
\mathrm{if}\ F_\nu^{-1}(\widetilde\varphi(u))<F_\mu^{-1}(u)<F_\nu^{-1}(u)\text{ and }\widetilde\varphi(u)<u_d;\\
\\
\delta_{F_\nu^{-1}(u)}(dy)\quad\quad\quad\quad \mathrm{otherwise.}\\
\end{array}
\right.
\label{defPetitmITMCdcx}
\end{equation}

Then $M^{ITS}(dx,dy)=\mu(dx)\,m^{Q^{ITS}}(x,dy)$ is a supermartingale coupling, called the inverse transform supermartingale coupling.

\begin{proof}[Proof of Proposition \ref{ITMCsupermartingale}] A mild adaptation of the proof of Proposition \ref{ITMCcasParticulier} is conclusive. In particular, Lemma \ref{chgtVariablesf+f-} gives the key property to show that $Q^{ITS}\in\mathcal Q^d$:
	\[
	\int_0^{u_d}h(u,\varphi(u))\,d\Psi_+(u)=\int_0^1h(\widetilde\varphi(v),v)\,d\Psi_-(v),
	\]
	for any measurable and bounded function $h:[0,1]^2\to\R$.
\end{proof}

If $\mu,\nu\in\mathcal P_1(\R)$ are such that $\mu\le_{icx}\nu$, denoting by $\overline\mu$ and $\overline\nu$ the respective images of $\mu$ and $\nu$ by $x\mapsto -x$, one can easily see that $\overline\mu\le_{dcx}\overline\nu$. So if $M^{Q^d}(dx,dy)$ denotes a supermartingale coupling between $\overline\mu$ and $\overline\nu$, then the image of $M^{Q^d}(dx,dy)$ by $(x,y)\mapsto(-x,-y)$ is a submartingale coupling between $\mu$ and $\nu$. In particular, the image of the inverse transform supermartingale coupling between $\overline\mu$ and $\overline\nu$ by $(x,y)\mapsto(-x,-y)$ is a submartingale coupling between $\mu$ and $\nu$.
	

\section{Appendix}
\label{Appendix}

We begin with a key result for the construction of the inverse transform martingale coupling.

\begin{lemma}\label{chgtVariablesf+f-}
	Let $f_1,f_2:(0,1)\to\R$ be two measurable nonnegative and integrable functions and $u_0\in[0,1]$ be such that $\int_0^{u_0}f_1(u)\,du=\int_0^1f_2(u)\,du$. Let $\Psi_1:[0,1]\ni u\mapsto \int_0^uf_1(v)\,dv$, $\Psi_2:[0,1]\ni u\mapsto \int_0^uf_2(v)\,dv$ and $\Gamma=\Psi_2^{-1}\circ\Psi_1$ where $\Psi_{2}^{-1}$ denotes the c\`ag pseudo-inverse of $\Psi_2$. Then $\Gamma$ is well defined on $[0,u_0]$ and for any measurable and bounded function $h:[0,1]\to\R$,
	\[
	\int_0^{u_0}h(\Gamma(u))f_1(u)\,du=\int_0^1h(v)f_2(v)\,dv.
	\]
\end{lemma}

The proof of Lemma \ref{chgtVariablesf+f-} relies on the next proposition, which is a well known result of integration by continuous and nondecreasing substitution, whose proof can be found for instance in Proposition $4.10$ Chapter $0$ \cite{Revuz-Yor}.

\begin{prop2}\label{prop410RevuzYor} Let $a,b\in\R$ be such that $a<b$. Let $\Psi:[a,b]\to\R$ be a continuous and nondecreasing function. Then for any Borel function $f:[\Psi(a),\Psi(b)]\to\R$,
	\[
	\int_a^bf(\Psi(s))\,d\Psi(s)=\int_{\Psi(a)}^{\Psi(b)}f(t)\,dt.
	\]
\end{prop2}

\begin{proof}[Proof of Lemma \ref{chgtVariablesf+f-}] Let $h:[0,1]\to\R$ be a measurable and bounded function. Since $\Psi_1$ is nondecreasing and continuous, using Proposition \ref{prop410RevuzYor}, we have
	\[
	\int_0^{u_0}h(\Gamma(u))f_1(u)\,du=\int_0^{u_0}h(\Psi_2^{-1}(\Psi_1(u)))\,d\Psi_1(u)=\int_0^{\Psi_1(u_0)}h(\Psi_2^{-1}(w))\,dw.
	\]
	
	Since $\int_0^{u_0}f_1(u)\,du=\int_0^1f_2(u)\,du$, we have $\Psi_1(u_0)=\Psi_2(1)$, and since $\Psi_2$ is nondecreasing and continuous, using once again Proposition \ref{prop410RevuzYor}, we have
	\[
	\int_0^{\Psi_1(u_0)}h(\Psi_2^{-1}(w))\,dw=\int_0^{\Psi_2(1)}h(\Psi_2^{-1}(w))\,dw=\int_0^1h(\Psi_2^{-1}(\Psi_2(v)))\,d\Psi_2(v).
	\]
	
	Since by Lemma \ref{FmoinsUnrondF}, $\Psi_2^{-1}(\Psi_2(v))=v$, $d\Psi_2(v)$-almost everywhere on $(0,1)$, we conclude that
	\[
	\int_0^1h(\Psi_2^{-1}(\Psi_2(v)))\,d\Psi_2(v)=\int_0^1h(v)f_2(v)\,dv.
	\]
\end{proof}

We complete this section with standard lemmas with their proofs, so that the present article is self-contained.

\begin{lemma}\label{FmoinsUnrondF} Let $I\subset\R$ be an interval, $F:I\to\R$ be a bounded and nondecreasing c\`adl\`ag function, $F(I)$ be the image of $I$ by $F$ and $F^{-1}$ be the left continuous pseudo-inverse of $F$, that is
	\[
	F^{-1}:u\in F(I)\mapsto\inf\{r\in I\mid F(r)\ge u\}
	\]
	
	Then for all $(x,u)\in I\times F(I)$, $F(x)\ge u\iff x\ge F^{-1}(u)$. Moreover, $F^{-1}(F(x))=x$, $dF(x)$-almost everywhere on $I$.
\end{lemma}
\begin{proof} Let $(x,u)\in I\times F(I)$. If $F(x)\ge u$, then by definition of the infimum, $x\ge F^{-1}(u)$. Conversely, if $x\ge F^{-1}(u)$, then let $(r_n)_{n\in\N}\in I^\N$ be a decreasing sequence converging to $F^{-1}(u)$. For all $n\in\N$, $F(r_n)\ge u$. By right-continuity of $F$, we get $F(F^{-1}(u))\ge u$ for $n\to+\infty$ . By monotonicity of $F$, we have $F(x)\ge F(F^{-1}(u))\ge u$.
	
	 Let us now prove the second statement. Let $a=\inf F(I)$ and $b=\sup F(I)$. If $a=b$, then $dF(x)$ is the trivial measure on $I$ so the statement is straightforward. Else, let $G:I\to[0,1]$ be defined for all $x\in I$ by $G(x)=(F(x)-a)/(b-a)$ and let $G^{-1}$ be its left-continuous pseudo-inverse. It is well known that for all $u\in(0,1)$, $G^{-1}(G(G^{-1}(u)))=G^{-1}(u)$. So $G^{-1}(G(G^{-1}(U)))=G^{-1}(U)$, where $U$ is a random variable uniformly distributed on $[0,1]$. By the inverse transform sampling, it implies that $G^{-1}(G(x))=x$, $dG(x)$-almost everywhere on $I$. For all $u\in F(I)$, we have $F^{-1}(u)=G^{-1}((u-a)/(b-a))$, hence $F^{-1}(F(x))=G^{-1}(G(x))=x$, $dG(x)$-almost everywhere on $I$. Since $dG(x)=\frac1{b-a}dF(x)$, $dG(x)$ and $dF(x)$ are equivalent, so $F^{-1}(F(x))=x$, $dF(x)$-almost everywhere on $I$.
	
%
\end{proof}

\begin{lemma}\label{FcomprisEntre0et1} Let $\mu\in\mathcal P(\R)$. Then $F_\mu(x)>0$ and $F_\mu(x_-)<1$, $\mu(dx)$-almost everywhere on $\R$.
	
\end{lemma}
\begin{proof} If $\{x\in\R\mid F_\mu(x)=0\}$ is nonempty, then it is an interval of the form $(-\infty,a]$ or $(-\infty,a)$, depending on whether $F_\mu(a)=0$ or not. If $F_\mu(a)=0$, then $\mu(\{x\in\R\mid F_\mu(x)=0\})=\mu((-\infty,a])=F_\mu(a)=0$. Else, since for all $x<a$, $F_\mu(x)=0$, then $\mu(\{x\in\R\mid F_\mu(x)=0\})=\mu((-\infty,a))=F_\mu(a_-)=0$.
	
	If $\{x\in\R\mid F_\mu(x_-)=1\}$ is nonempty, then it is an interval of the form $[a,+\infty)$ or $(a,+\infty)$, depending whether $F_\mu(a_-)=1$ or not. If $F_\mu(a_-)=1$, then $\mu(\{x\in\R\mid F_\mu(x_-)=1\})=\mu([a,+\infty))=1-F_\mu(a_-)=0$. Else, since for all $x>a$, $F_\mu(x_-)=1$, then $\mu(\{x\in\R\mid F_\mu(x_-)=1\})=\mu((a,+\infty))=1-F_\mu(a)=1-\lim_{x\to a,x>a}F_\mu(x_-)=0$, by right continuity of $F_\mu$.
\end{proof}

\begin{lemma} Let $\mu\in\mathcal P_1(\R)$. Then $\mu$ is symmetric with mean $\alpha\in\R$, that is $(x-\alpha)_\sharp\mu(dx)=(\alpha-x)_\sharp\mu(dx)$ where $\sharp$ denotes the pushforward operation, iff
	\[
	F_\mu^{-1}(u_+)=2\alpha-F_\mu^{-1}(1-u),
	\]
	for all $u\in(0,1)$. In that case, $F_\mu^{-1}(u)=2\alpha-F_\mu^{-1}(1-u)$ for $u\in(0,1)$ up to the at most countable set of discontinuities of $F_\mu^{-1}$.
	\label{caracterisationLoiSymetriqueFonctionQuantile}
\end{lemma}
\begin{proof} Let $U$ be a random variable uniformly distributed on $[0,1]$. Then, by the inverse transform sampling, $F_\mu^{-1}(1-U)\sim\mu$, so $2\alpha-F_\mu^{-1}(1-U)\sim\mu$ since $\mu$ is symmetric with mean $\alpha$. Since $u\mapsto 2\alpha-F_\mu^{-1}(1-u)$ is nondecreasing, then one can show that $2\alpha-F_\mu^{-1}(1-u)=F_\mu^{-1}(u)$, $du$-almost everywhere on $(0,1)$. Indeed, as shown in Lemma A.3 \cite{ApproxMOTJourdain2}, for all $u,q\in(0,1)$ such that $q<u$, if $F_\mu^{-1}(u)<2\alpha-F_\mu^{-1}(1-q)$, then
	\begin{align*}
	\PP(2\alpha-F_\mu^{-1}(1-U)\le F_\mu^{-1}(u))&\le\PP(2\alpha-F_\mu^{-1}(1-U)<2\alpha-F_\mu^{-1}(1-q))\\
	&\le q<u\le\PP(F_\mu^{-1}(U)\le F_\mu^{-1}(u))=\PP(2\alpha-F_\mu^{-1}(1-U)\le F_\mu^{-1}(u)),
	\end{align*}
	which is contradictory, so $F_\mu^{-1}(u)\ge\sup_{q\in(0,u)}(2\alpha-F_\mu^{-1}(1-q))$. By symmetry, $2\alpha-F_\mu^{-1}(1-u)\ge\sup_{q\in(0,u)}F_\mu^{-1}(q)=F_\mu^{-1}(u)$ by left-continuity and monotonicity of $F_\mu^{-1}$. Since $F_\mu^{-1}$ has an at most countable set of discontinuities, then for $du$-almost all $u\in(0,1)$, $2\alpha-F_\mu^{-1}(1-u)=\sup_{q\in(0,u)}(2\alpha-F_\mu^{-1}(1-q))\le F_\mu^{-1}(u)\le 2\alpha-F_\mu^{-1}(1-u)$. Therefore, $2\alpha-F_\mu^{-1}(1-u)=F_\mu^{-1}(u_+)$, $du$-almost everywhere on $(0,1)$ and even everywhere on $(0,1)$ since both sides are right-continuous.
	
	%
\end{proof}

\begin{lemma}\label{FmuXCorrigeUniforme} Let $\mu\in\mathcal P(\R)$, let $X:\Omega\to\R$ be a random variable with distribution $\mu$ and let $V$ be a random variable independent from $X$ and uniformly distributed on $(0,1)$. Let $W:\Omega\to\R$ be the random variable defined by
	\[
	W=F_\mu(X_-)+V(F_\mu(X)-F_\mu(X_-)).
	\]

Then $W$ is uniformly distributed on $(0,1)$, and $F_\mu^{-1}(W)=X$ almost surely.
\end{lemma}
\begin{proof} Let $h:\R\to\R$ be a measurable and bounded function. Then
	\begin{align*}
	\E[h(W)]&=\E[h(F_\mu(X_-)+V(F_\mu(X)-F_\mu(X_-))]=\int_0^1\int_\R h(F_\mu(x_-)+v(F_\mu(x)-F_\mu(x_-)))\,\mu(dx)\,dv\\
	&=\int_0^1\int_\R \1_{\{\mu(\{x\})=0\}}h(F_\mu(x_-)+v(F_\mu(x)-F_\mu(x_-)))\,\mu(dx)\,dv\\
	&\phantom{=}+\int_0^1\int_\R\1_{\{\mu(\{x\})>0\}} h(F_\mu(x_-)+v(F_\mu(x)-F_\mu(x_-)))\,\mu(dx)\,dv\\
	&=\int_\R \1_{\{\mu(\{x\})=0\}}h(F_\mu(x))\,\mu(dx)+\sum_{x\in\R:\mu(\{x\})>0}\int_{F_\mu(x_-)}^{F_\mu(x)}h(v)\,dv\\
	&=\int_0^1 \1_{\{\mu(\{F_\mu^{-1}(u)\})=0\}}h(F_\mu(F_\mu^{-1}(u)))\,du+\sum_{x\in\R:\mu(\{x\})>0}\int_{F_\mu(x_-)}^{F_\mu(x)}h(v)\,dv\\
	&=\int_0^1 \1_{\{\mu(\{F_\mu^{-1}(u)\})=0\}}h(u)\,du+\sum_{x\in\R:\mu(\{x\})>0}\int_{F_\mu(x_-)}^{F_\mu(x)}h(v)\,dv,
	\end{align*}
	where we used for the last but one equality the inverse transform sampling, and for the last equality the fact that $F_\mu(F_\mu^{-1}(u))=u$ if $F_\mu$ is continuous in $F_\mu^{-1}(u)$. One can easily see that for all $x\in\R$ and $u\in(0,1)$,
	\[
	F_\mu(x_-)<u\le F_\mu(x)\implies x=F_\mu^{-1}(u)\implies F_\mu(x_-)\le u\le F_\mu(x),
	\]
	which implies
	\[
	\bigcup_{x\in\R:\mu(\{x\})>0}\left(F_\mu(x_-),F_\mu(x)\right]\subset\{u\in(0,1)\mid\mu(\{F_\mu^{-1}(u)\})>0\}\subset\bigcup_{x\in\R:\mu(\{x\})>0}[F_\mu(x_-),F_\mu(x)],
	\]
	so
	\[
	\sum_{x\in\R:\mu(\{x\})>0}\int_{F_\mu(x_-)}^{F_\mu(x)}h(v)\,dv=\int_0^1 \1_{\{\mu(\{F_\mu^{-1}(u)\})>0\}}h(u)\,du.
	\]
	
	Therefore, $\E[h(W)]=\int_0^1 \1_{\{\mu(\{F_\mu^{-1}(u)\})=0\}}h(u)\,du+\int_0^1 \1_{\{\mu(\{F_\mu^{-1}(u)\})>0\}}h(u)\,du=\int_0^1h(u)\,du$. So $W$ is uniformly distributed on $(0,1)$. Moreover, on $\{F_\mu(X_-)=F_\mu(X)\}$, $W=F_\mu(X)$ and by Lemma \ref{FmoinsUnrondF}, $F_\mu^{-1}(W)=X$ almost surely. Since $V>0$ a.s., on $\{F_\mu(X_-)<F_\mu(X)\}$, a.s., $F_\mu(X_-)<W\le F_\mu(X)$ so $F_\mu^{-1}(W)=X$.
\end{proof}

\bibliographystyle{plain}
\bibliography{5-biblio}

\end{document}